\documentclass[reqno,11pt, psfig]{amsart}
\usepackage{cases,comment}
\usepackage{mathrsfs}
\usepackage[T1]{fontenc}
\usepackage{mathrsfs}
\usepackage{amsmath}
\usepackage{amsmath,latexsym,amssymb,amsfonts,amsbsy, amsthm}
\usepackage{txfonts}
\usepackage{enumerate,enumitem}
\usepackage{bm}
\usepackage[usenames]{color}
\usepackage{xspace,colortbl}
\usepackage{epsfig}
\usepackage{graphicx}
\usepackage{subfigure}
\usepackage{amsmath,amsfonts,amsthm,amssymb,amscd}
\usepackage{color}
\usepackage[title,titletoc,toc]{appendix}
\usepackage{bigints}

\usepackage{hyperref}
\allowdisplaybreaks

\newcommand{\be}{\begin{equation} }
\newcommand{\ee}{\end{equation}}
\newcommand{\bee}{\begin{equation*} }
\newcommand{\eee}{\end{equation*}}
\newcommand{\bse}{\begin{subequations}}
\newcommand{\ese}{\end{subequations}}

\newcommand{\R}{\mathbb{R}}

\newcommand{\rmd}{\mbox{\rm d}}

\newcommand{\Ham}{\mathcal{H}}

\theoremstyle{plain}
\newtheorem{theorem}{Theorem}[section]
\newtheorem{corollary}[theorem]{Corollary}

\newtheorem{lemma}{Lemma}[section]
\newtheorem{proposition}{Proposition}[section]
\theoremstyle{remark}
\newtheorem{remark}{Remark}[section]

\theoremstyle{definition}
\newtheorem{definition}{Definition}

\numberwithin{equation}{section}

\title[Asymptotic stability of smooth CH multi-solitons ]
{Asymptotic stability of smooth solitons and multi-solitons for the
Camassa--Holm equation}

\author[R. M. Chen]{Robin Ming Chen}
\address[R. M. Chen]{Department of Mathematics, University of Pittsburgh, PA 15260}\email{mingchen@pitt.edu}
\author[Y. Lan]{Yang Lan}
\address[Y. Lan]{Yau Mathematical Sciences Center, Tsinghua University, 100084 Beijing, P. R. China}
\email{lanyang@mail.tsinghua.edu.cn}
\author[Y. Liu]{Yue Liu}
\address[Y. Liu]{Department of Mathematics, University of Texas at Arlington, TX 76019}\email{yliu@uta.edu}
\author[Z. Wang]{Zhong Wang}
\address[Z. Wang]{School of Mathematics and Big Data, Foshan University, 528000, China}\email{wangzh79@fosu.edu.cn}

\date{}                                           
\begin{document}
\thispagestyle{empty}

\begin{abstract}

%
%
%

We establish the asymptotic stability of smooth solitons and multi-solitons for the Camassa–Holm (CH) equation in the energy space $H^1(\R)$. We show that solutions initially close to a soliton converge, up to translation, weakly in $H^1(\R)$ as time tends to infinity to a (possibly different) soliton. The analysis is based on a Liouville-type rigidity theorem characterizing solutions that remain localized near a soliton trajectory. 

A central feature of the proof is a complete spectral resolution of the linearized CH operator around a soliton. This linear theory is obtained via the bi-Hamiltonian and integrable structure of the CH equation, through the recursion operator and the completeness of the associated squared eigenfunctions. It provides a substitute for the classical spectral framework used in KdV and gKdV equations, which is unavailable in the nonlocal and variable-coefficient setting of CH.

The spectral resolution yields sharp decay estimates for the linearized flow in exponentially weighted spaces, which in turn lead to the nonlinear rigidity result and the asymptotic stability of a single soliton. Combined with known orbital stability results, this approach extends to well-ordered trains of solitons and to the explicit multi-soliton solutions generated by the inverse scattering method. As an additional application, we revisit the linearized problems associated with other integrable dispersive equations, including the KdV and mKdV equations, from the perspective of squared-eigenfunction expansions.

\end{abstract}

\maketitle

completeness relation

\vskip 0.2cm

\maketitle

 \setcounter{tocdepth}{1}

\tableofcontents

\setcounter{tocdepth}{1}
\section{Introduction}\label{intro}
We consider the Camassa--Holm (CH ) equation \cite{CH93, FF81}
\begin{equation}\label{eq1}
 u_{t}-u_{xxt}+2\omega u_{x}+3uu_{x}-2u_{x}u_{xx}-uu_{xxx}=0, \quad (t,x)\in \mathbb{R}\times \mathbb{R},
\end{equation}
where $\omega\geq0 $ and $u(t,x)$ denotes the fluid velocity in dimensionless space-time variables. This equation models the unidirectional propagation of shallow-water waves over a flat bottom \cite{CH93,J02}, and a rigorous hydrodynamic justification within the framework of shallow-water asymptotics can be found in \cite{cola}.

The CH equation (\ref{eq1}) is a completely integrable system: it admits a Lax pair and an infinite number of conserved quantities \cite{BBS98,CH93,C,C01,CM99}. It is capable of describing both permanent and breaking waves \cite{C00,CE98,M}. In the natural energy space $H^1(\mathbb{R})$, both the smooth solitary waves (corresponding to $\omega > 0$) and the peaked solitary waves ($\omega = 0$) are orbitally stable \cite{CS02,CS00}. Moreover, the peaked solitary waves are proved to be asymptotically stable \cite{M18}. In addition, the CH flow arises as the geodesic equation for the right-invariant $H^1$ metric on the Bott--Virasoro group for $\omega>0$ \cite{CK02,CK03,V05}, and on the diffeomorphism group for $\omega = 0$ \cite{CK02,CK03}.

The CH equation (\ref{eq1}) also possesses a bi-Hamiltonian  structure\cite{CH93,FF81}. Writing the momentum density as
\[
m: = u-u_{xx},
\]
one has
\begin{equation}\label{eq:recursion}
\begin{split}
 &\frac{\partial m}{\partial t}=\mathcal{J}_2\frac{\delta \Ham_{2}[m]}{\delta m}=\mathcal{J}_1\frac{\delta \Ham_{1}[m]}{\delta m}, \\
 &\mathcal{J}_1:=-(2\omega \partial_x +m\partial_x+\partial_x m),\quad \mathcal{J}_2:=-(\partial_x-\partial_x^{3}),
\end{split}
\end{equation}
with the two Hamiltonians
\begin{eqnarray}\label{eq2a}
&& \Ham_{1}[m] = \Ham_{1}(u)=E(u)=\frac{1}{2}\int_{\mathbb{R}} m u \rmd x, \\ \label{eq2b}
&& \Ham_{2}[m] = \Ham_{2}(u)=F(u)=\frac{1}{2}\int_{\mathbb{R}}(u^{3}+uu_{x}^{2}+2\omega u^{2}) \, \rmd x.
 \end{eqnarray}
Using the identity $m = (1 - \partial^2_x)u$, equation \eqref{eq1} may also be written as an infinite-dimensional Hamiltonian PDE of the form
\begin{equation}\label{Hamilton form}
 u_{t}=\mathcal{J}\frac{\delta \Ham_{2}(u)}{\delta u},\qquad \mathcal{J}:=(1-\partial_x^2)^{-1}\mathcal{J}_2(1-\partial_x^2)^{-1}=-\partial_x(1-\partial_x^2)^{-1},
\end{equation}
where $\mathcal{J}$ is a skew symmetric, bounded operator on $L^2(\R)$.

More generally, there exist  infinitely  many conservation laws
(multi-Hamiltonian structures) $\Ham_n[m]$, $n=0,\pm1, \pm2,\ldots$, linked through the Lenard recursion scheme \cite{L05}:
\begin{equation}\label{eq2aa}
 \mathcal{J}_2\frac{\delta \Ham_{n}[m]}{\delta m}=\mathcal{J}_1\frac{\delta \Ham_{n-1}[m]}{\delta m}.
\end{equation}
Algorithmic constructions of these conserved quantities may be found in \cite{CL05,FS99,I05,L05}.

\subsection{CH solitons}\label{subsec intro solitons}
By means of inverse scattering theory, the evolution of rapidly decaying initial data for the CH equation can be described through purely algebraic methods. Solutions decompose into a well-structured family of functions determined by the scattering data. The {\em{soliton resolution conjecture}} states that, for a large class of dispersive equations, any global solution decomposes as $t\rightarrow +\infty$ into a finite sum of (rescaled and translated) solitons plus a dispersive radiation term governed by the corresponding linearized equation. For the CH equation, the solitonic component of this decomposition consists of multi-soliton solutions, which we describe below.

It is known that the CH equation \eqref{eq1} possesses smooth solitary-wave solutions called {\em{solitons}} if $ \omega > 0$ \cite{CH94} and peaked solitary waves termed \emph{peakons} if $ \omega = 0$ \cite{CH93}. These solitary waves arise as constrained minimizers of the Hamiltonian in the $H^1$ topology. When $ \omega > 0$,  the CH equation possesses a smooth soliton profile $\varphi_c$, depending on a translation parameter $ x_0 \in \mathbb{R}$, of the form
 \begin{equation}\label{eq:so}
u(t,x) =\varphi_c(x-ct-x_0),\ c>2\omega,  \, t\in\R, \, x_0 \in \mathbb{R},
\end{equation}
which can be written parametrically as follows \cite{J03,LZ04}:
\begin{eqnarray*}
&&u(t,x)=\frac{c-2\omega}{1+(2\omega/c)\sinh^2 \theta},\\ &&\theta=\frac{1}{2\sqrt{\omega}}\sqrt{1-\frac{2\omega}{c}}\big(y-c\sqrt{\omega}t\big),\\
&&x=\frac{y}{\sqrt{\omega}}+\ln\frac{\cosh(\theta-\theta_0)}{\cosh(\theta+\theta_0)},\quad \theta_0:=\tanh^{-1}\sqrt{1-\frac{2\omega}{c}}.
\end{eqnarray*}
Substitution of \eqref{eq:so} into \eqref{eq1} shows that $\varphi_c > 0$ satisfies
\begin{equation}\label{eq:stationary}
-c\varphi_c + c \partial_x^2 \varphi_{c}+\frac32\varphi_c^2+2\omega\varphi_c = \varphi_c \partial_x^2 \varphi_{c}+\frac12(\partial_x \varphi_{c})^2, \quad x \in\mathbb{R}.
\end{equation}
For $\omega>0$, such a solitary wave exists if and only if $ c>2\omega$.
Conversely, each $ c>2\omega$ determines a unique (up to translation) smooth soliton profile $\varphi_c$. As shown in \cite{CS02}, these profiles satisfy:
\begin{enumerate}[label=\arabic*)]
\item $\varphi_c$ is smooth, positive, and even, with peak height $c-2\omega$, and strictly decreasing away from the peak. \\
\item $\varphi_c$ is concave while its values lie in $ \left [c-\frac{\omega}{2}-\sqrt{c\omega+\frac{\omega^2}{4}}, \, c-2\omega \right]$, and convex elsewhere.\\
\item  Multiplying both sides of \eqref{eq:stationary} by $\partial_x \varphi_c$ and integrating over $\R$ leads to
\[
(\partial_x \varphi_{c})^2 (c-\varphi_c)=\varphi_c^2(c-2\omega-\varphi_c),
\]
from which one obtains $|\partial_x \varphi_c|\leq \varphi_c$ and the exponential decay
\[
\varphi_c(x)=O\left(\exp\left(-\sqrt{1-\frac{2\omega}{c}}|x|\right)\right) \quad \text{for}\ |x|\rightarrow\infty.
\]
Furthermore, \eqref{eq:stationary} implies
\begin{equation}\label{eq:potential}
\varphi_c - \partial_x^2 \varphi_c=\frac{\omega\varphi_c(2c-\varphi_c)}{(c-\varphi_c)^2}>0.
\end{equation}
\item If $\varphi_c$ reaches its maximum at $x=0$, then $\varphi_c$ converges
uniformly on every compact subset of $\R$, as $\omega \to 0$, to the peakon profile $ce^{-|x|}$.
\end{enumerate}

Beyond isolated solitons, the CH equation \eqref{eq1} also possesses {\em{multi-solitons}} \cite{cht, CGI06, LZ04, M05}. These solutions can be expressed in a parametric form analogous to the one-soliton case. As $t\rightarrow\infty$, an $N$-soliton $U^{(N)}(x;c_j^0,x^0_j-c^0_jt)$ decomposes into $N$ asymptotically noninteracting solitons with speeds $c^0_j>0$ and space shifts $x_j^0$, $j=1,2,\cdot\cdot\cdot,N$, namely
\begin{equation}\label{eq:asymptotic behavior}
U^{(N)}(t)\sim\sum_{n=1}^N\varphi_{c^0_j}(\cdot-c^0_jt-x_j^{\pm}),\quad\  t\rightarrow \pm \infty,
\end{equation}
for appropriate phase parameters $x_j^\pm \in \mathbb{R}$ depending on the initial data $(c^0_j, x^0_j)$ \cite{M05}.

Finally, the CH equation \eqref{eq1} enjoys the important invariance property that if $\omega>0$ and $m(0, x)+\omega > 0$, then $m(t, x)+\omega > 0$ for all $t \in \mathbb{R}$, see \cite{C, C01, CM99}.

\subsection{Nonlinear stability of smooth CH solitons}\label{subsec intro stab}

There is an extensive literature on the stability of CH peakons (the case $\omega = 0$), culminating in the asymptotic stability result of Molinet \cite{M18}. In this paper, however, we focus exclusively on the stability of smooth solitons arising when $\omega > 0$.

In the peakon setting, the soliton profiles admit explicit expressions and the associated momentum density $m = u - u_{xx}$ is a finite sum of Dirac measures. This explicit structure has been exploited in a decisive way by Molinet \cite{M18}, where a rigidity result is first established at the level of the momentum density and then transferred to the solution $u$. 

By contrast, for smooth CH solitons, neither the soliton profile nor the corresponding momentum density admits an explicit representation. As a consequence, the strategy developed for peakons does not directly apply. In particular, rigidity cannot be accessed through pointwise or measure-theoretic arguments and instead requires a fundamentally different approach.

We begin by recalling the well-known results on orbital stability. The nonlinear orbital stability of the single soliton $\varphi_c$ was established by Constantin--Strauss \cite{CS02} using the general spectral method introduced by Benjamin \cite{B72} and developed by Grillakis--Shatah--Strauss \cite{GSS87}. Their proof relies on the Hamiltonian structure of the CH equation and on a careful analysis of the constrained variational properties of the solitary waves. For multi-soliton configurations, a general framework for establishing the orbital stability of trains of $N$ solitary waves in Hamiltonian dispersive equations was provided by Martel--Merle--Tsai \cite{MMT02} in the context of gKdV. Adapting this approach, El~Dika and Molinet \cite{EM07} proved the orbital stability of $N$-soliton trains for the CH equation in the energy space $H^1(\mathbb{R})$. A key step in their argument is the derivation of an almost-monotonicity property for localized versions of the conserved quantities $\Ham_1$ and $\Ham_2$. More recently, the Lyapunov stability of smooth multi-solitons was obtained in \cite{WL20}.

We emphasize that all the above results concern orbital or Lyapunov stability. To the best of our knowledge, no prior work has established the asymptotic stability of smooth CH solitons, either for a single soliton or for multi-soliton configurations. The present paper provides the first such results. Our method combines the integrable structure of the CH hierarchy, a detailed spectral analysis of the linearized operator, weighted semigroup estimates, and a nonlinear rigidity argument, as described in the following subsections.

\subsection{Main results}

The main objective of this paper is to establish the asymptotic stability of smooth solitons and  multi-solitons of the CH equation \eqref{eq1}. Before stating our results, we introduce the natural function space setting for our initial data. Following \cite{CM01}, define
\begin{equation}\label{eq:Yspace}
Y := \left\{ u\in H^1(\R); \quad m=u-u_{xx}\in \mathcal M(\R) \right\},
\end{equation}
where $\mathcal M(\R)$ denotes the set of Radon measures with bounded total variation. We denote by $Y_+ \subset Y$ the closed subspace
\[
Y_+ = \{u \in Y : u-u_{xx} \in \mathcal M_{+}(\R)\},
\]
where $\mathcal M_{+}(\R)$ is the set of nonnegative finite Radon measures on $\R$.

We also recall the notion of right-moving $H^{1}$-localized solutions, introduced in \cite{EM07} and used by El~Dika in the study of the BBM equation \cite{E05}.
\begin{definition}[\cite{EM07}]\label{de3.1}
We say that an $H^{1}$ solution of the CH equation \eqref{eq1} is \emph{$H^{1}$-localized} if there exists $c_{1}>0$ {and} a $\mathcal {C}^{1}$ function $x(\cdot)$ with $\inf\limits_{t \in \mathbb{R}} \dot{x}(t)>c_{1}$, such that for any $\varepsilon>0$, there exists $R_{\varepsilon}>0$ such that for all $t\in \R$,
\begin{equation*}
\int_{|x|>R_{\varepsilon}} u^{2}(t, x+x(t))+u_{x}^{2}(t, x+x(t)) \rmd x<\varepsilon.
\end{equation*}
\end{definition}

Our first main result is a rigidity statement: any solution of the CH flow that remains uniformly close to a soliton and is right-moving $H^1$-localized must itself be exactly a soliton.

\begin{theorem}[Rigidity property for the CH flow] \label{th1.1}
Let $c>2\omega$, $u_0\in Y_+$ and $u\in \mathcal{C}(\R;H^1(\R))$ be the solution of  \eqref{eq1} emanating from $u_0$. There exists $\alpha_0>0$ such that if $\|u_0-\varphi_c\|_{H^1}<\alpha_0$ and $u(t)$ is $H^1$-localized, then there exist $y_0\in\R$ and $c^\star>2\omega$ such that
 \[u(t,x)=\varphi_{c^\star}(x-c^\star t-y_0),\ \text{for}\ (t,x)\in\R^2.\]
\end{theorem}

The rigidity theorem allows us to follow the strategy of Martel--Merle \cite{MM01,MM05,MM08} to obtain asymptotic stability of a single smooth CH soliton in $H^1(\mathbb{R})$.

\begin{theorem}[Asymptotic stability of smooth solitons] \label{th1.2}
Let $c>2\omega$ be fixed, $u_0\in Y_+$ and $u\in \mathcal{C}(\R;H^1(\R))$ be the solution of  \eqref{eq1} emanating from $u_0$. There exists $\alpha_0>0$ such that if $\|u_0-\varphi_c\|_{H^1}<\alpha_0$, then there exist  $c^\star>2\omega$ and a $\mathcal {C}^{1}$ function $x(t):\R\rightarrow\R$ with $\lim\limits_{t\rightarrow+\infty}\dot{x}(t)=c^\star$ such that
\begin{equation}\label{weak asympt lim}
u(t,x+x(t))\rightharpoonup\varphi_{c^\star} \ \text{in} \ H^1(\R), \ \text{as}\ t\rightarrow +\infty.
\end{equation}
Moreover,
\begin{equation}\label{asympt lim}
\lim_{t\rightarrow+\infty}\|u(t)-\varphi_{c^\star}(\cdot-x(t))\|_{H^1(x>2\omega t)}=0.
\end{equation}
\end{theorem}

While the nonlinear structure of the proof follows the general Liouville-modulation philosophy developed by Martel and Merle, the present result relies on a fundamentally different linear mechanism. In the KdV and gKdV settings, the Liouville argument is closed using detailed spectral information for a local Schr\"odinger operator, which is available explicitly and feeds directly into virial and coercivity estimates. In the CH equation, no such linear theory is available a priori. The novelty of this work lies in replacing the classical Schr\"odinger-based framework with a complete spectral resolution of the linearized CH operator obtained from the recursion operator and the associated squared eigenfunctions. This integrability-based linear theory provides the rigidity mechanism leading to Theorem \ref{th1.1} and serves as the foundation for the asymptotic stability analysis developed in this paper. More discussion regarding the methodology will be presented in Section \ref{subsec method}.

It is known in \cite{EM07} that the orbital stability of well-ordered trains of CH smooth solitons in $H^1$ was established. Combining this result with Theorem~\ref{th1.2} and following the strategy of \cite{MMT02}, we extend asymptotic stability to the configuration of well-ordered trains of solitons.

\begin{theorem}[Asymptotic stability of well-ordered soliton trains] \label{th1.3}
 Let $ 0 < 2\omega < c_1< c_2< \cdot\cdot\cdot< c_N$. There exist $L_0>0$ and $\alpha_0>0$ such that if $u\in \mathcal{C}(\R;H^1(\R))$ is the solution of  \eqref{eq1} emanating from $u_0\in Y_+$, with
 \begin{equation*}
 \left\| u_0-\sum_{j=1}^N\varphi_{c^0_j}(\cdot-x_j^0) \right\|_{H^1}<\alpha_0,\quad x_j^0-x_{j-1}^0\geq L_0,\quad j=2,3,\cdot\cdot\cdot,N,
 \end{equation*}
then there exist $c^\star_1,c^\star_2,\cdot\cdot\cdot, c^\star_N$ satisfying  $2\omega< c^\star_1< c^\star_2<\cdot\cdot\cdot< c^\star_N$ and $\mathcal {C}^{1}$ functions $x_j(t):\R\rightarrow\R$, $j=1,2,\cdot\cdot\cdot,N$, with $\lim\limits_{t\rightarrow+\infty}\dot{x}_j(t)=c^\star_j$ such that
\begin{equation*}
u(t,x+x_j(t))\rightharpoonup\varphi_{c^\star_j} \ \text{in} \ H^1(\R), \ \text{as}\ t\rightarrow +\infty.
\end{equation*}
Moreover,
\begin{equation*}
\lim_{t\rightarrow+\infty} \left\| u(t)-\sum_{j=1}^N\varphi_{c^\star_j}(\cdot-x_j(t)) \right\|_{H^1(x>2\omega t)}=0.
\end{equation*}
\end{theorem}

Because the CH equation \eqref{eq1} is completely integrable, explicit $N$-soliton solutions $U^{(N)}(x;c_j^0,x^0_j-c^0_jt)$ are known \cite{M05}. As a corollary of Theorem~\ref{th1.3}, we obtain orbital and asymptotic stability for this entire explicit family.

\begin{corollary}[Asymptotic stability of smooth multi-solitons] \label{co1.3}
 Let $ 0 < 2\omega < c^0_1< c^0_2< \cdot\cdot\cdot< c^0_N$ be $N$ velocities and $x^0_1,x^0_2,\cdot\cdot\cdot,x^0_N\in\R$. Then for any $\varepsilon>0$, there exist  $\alpha_1>0$ such that if $u\in \mathcal{C}(\R;H^1(\R))$ is the solution of  \eqref{eq1} emanating from $u_0\in Y_+$, with
 \begin{equation*}
 \left\| u_0-U^{(N)}(\cdot;c^0_j,x^0_j) \right\|_{H^1}<\alpha_1,
 \end{equation*}
  then there exist $\mathcal {C}^{1}$ functions $x_j(t):\R\rightarrow\R$, $j=1,2,\cdot\cdot\cdot,N$, with $\lim\limits_{t\rightarrow+\infty}\dot{x}_j(t)=c^\star_j$, such that
\begin{equation}\label{stability}
\left\| u(t)-U^{(N)}(\cdot;c^0_j,x_j(t)) \right\|_{H^1}<\varepsilon, \ \text{for any}\ t\geq0.
\end{equation}
Moreover,
\begin{equation}\label{asympt lim3}
\lim_{t\rightarrow+\infty} \left\| u(t)-U^{(N)}(\cdot;c^\star_j,x_j(t)) \right\|_{H^1(x>2\omega t)}=0.
\end{equation}
\end{corollary}

\subsection{Challenges and methodology}\label{subsec method}

The asymptotic stability of smooth CH solitons presents structural and analytic challenges that are notably different from, and more difficult than, those encountered in many dispersive models such as the modified KdV (mKdV) or generalized KdV (gKdV). The linearized operator around a CH soliton is nonlocal, involving the inverse Helmholtz operator $(1 - \partial_x^2)^{-1}$, so its spectral analysis cannot rely on classical Sturm--Liouville or Schr\"odinger operator techniques. Even the basic spectral features, such as the precise location and multiplicity of eigenvalues, the structure of the continuous spectrum, or the completeness of radiation modes, were unknown prior to this work.

A further difficulty, absent in KdV, is that the CH soliton profile $\varphi_{c}$ does not possess a closed-form expression. While the asymptotic stability theory of Martel and Merle \cite{Ma06,MM01,MM05,MM08} for general nonlinearities does not rely on explicit formulas for the soliton itself, it nevertheless exploits in an essential way the explicit spectral information available for the linearized KdV operator. Indeed, the Schr\"odinger operator arising in the linearization around the KdV soliton has an explicitly known potential, and its eigenvalues and eigenfunctions can be computed in closed form. These explicit spectral data enter directly in the virial identity and coercivity arguments, especially in the identification of the unique negative eigenvalue and the detailed structure of the corresponding eigenfunction. For general gKdV equations this information is inherited from the KdV case through a perturbative argument. In contrast, for CH such an approach is infeasible: the soliton profile is implicit, the operator is nonlocal, and there is no obvious ODE structure from which one could extract the negative direction by classical arguments (see Remark \ref{rk spec of L}). Thus, none of the virial tools used in the gKdV literature can be imported into CH in any straightforward manner.


A different approach appears in the asymptotic stability theory for the Benjamin--Bona--Mahony (BBM) equation \cite{MW96}, where the authors rely on a refined Evans function analysis of the linearized operator, leading to asymptotic stability for all but a discrete exceptional set of soliton speeds. For CH, however, the integrable structure completely bypasses the need for Evans function techniques and allows us to obtain a full description of the spectrum for every admissible speed. As a consequence, our analysis applies uniformly across the full range $c > 2\omega$, with no exceptional values, and the rigidity and asymptotic stability arguments rest directly on the spectral resolution provided by integrability.



Another perspective on asymptotic stability is provided by the more recent work of Germain--Pusateri--Rousset on the mKdV equation \cite{GPR16}. Instead of relying primarily on coercivity and virial identities, they combine a localized modulation analysis near the soliton with a detailed study of the dispersive radiation. This latter component is handled through the space-time resonances framework introduced by Germain--Masmoudi--Shatah \cite{GMS09}, which enables one to track nonlinear interactions between oscillatory modes and to capture the modified scattering behavior of the radiation. This refined dispersive machinery is particularly well suited for equations such as mKdV, whose linear dispersion relation is local and regular, whose soliton profile is explicit, and for which precise Fourier-based estimates allow a sharp description of resonant and non-resonant interactions.

However, these structural advantages are not available for the CH equation. The linearized CH operator is fundamentally nonlocal, with a symbol involving rational functions of the Fourier variable; as a result, the underlying ``dispersion relation'' does not lend itself to a clean analysis of resonant sets, stationary phase, or multilinear oscillatory integrals --- features that are essential in the space-time resonances method. Moreover, the absence of an explicit soliton profile means that one cannot isolate the discrete spectral component by direct projection arguments, nor can one express the radiation as a simple perturbation governed by a local dispersive equation. The nonlocality of the operator further obscures the structure of nonlinear interactions, making it exceedingly difficult to identify or exploit cancellations among resonant contributions.

Our approach circumvents these obstacles by invoking tools from the integrable structure of the CH equation. The key object is the recursion operator of the CH hierarchy. In the Lax pair formulation, this operator generates the higher flows and intertwines the spectral problems associated with the spatial part of the Lax pair. A crucial observation is that this recursion operator commutes with the linearized operator at the soliton. This commutativity serves as the main analytic pathway from integrability to spectral theory. It enables us to construct the eigenfunctions of the linearized operator explicitly, in terms of squared eigenfunctions of the Lax operator, thereby providing a complete and fully explicit spectral resolution despite the absence of a closed formula for the soliton (see Section \ref{subsec spec multi}). Through this mechanism, we obtain a complete characterization of the discrete and continuous spectrum, including the fact that the translation mode generates the only discrete eigenvalue, that no unstable eigenvalues exist, and that the continuous spectrum arises through the transform of the Lax spectrum under the squared eigenfunction map.

With the explicit spectral structure in hand, we establish sharp semigroup decay estimates in exponentially weighted Sobolev spaces following the framework of Pego--Weinstein's work for KdV \cite{PW94}. The weighted spaces serve to isolate the neutral translational eigenmode and to shift the continuous spectrum strictly into the stable half-plane, yielding exponential decay for the linearized flow. To convert the linearized decay into a full nonlinear asymptotic stability statement, a rigidity argument is required:  the only solutions close to solitons which exhibit regularity and strong decay properties are the solitons itself. In our CH setting, the rigidity proof relies critically on the spectral resolution produced by the recursion operator machinery and the weighted semigroup decay. The spectral decomposition allows us to control separately the translational mode and the dispersive radiation, while the weighted decay ensures that all radiation is exponentially damped in the moving frame. A localized modulation argument is then used to track the soliton center and guarantees that the modulated solution remains uniformly aligned with a member of the soliton family. The exponential decay of the linearized semigroup prevents the accumulation of radiation near the soliton and forces the solution to relax to a single modulated profile, and hence yields a nonlinear Liouville-type theorem for the CH solitons.

The integrable-structure-based spectral analysis also extends smoothly to multi-soliton configurations. Each soliton carries its own localized spectral resolution, and because the solitons separate linearly in time, their interactions decay sufficiently rapidly to allow the weighted semigroup estimates and modulation arguments to run uniformly in regions associated with different soliton centers. It is useful here to distinguish between well-separated trains of solitons, for which the strategy of Martel--Merle--Tsai \cite{MMT02} applies once the asymptotic stability of a single soliton is known, and the exact multi-solitons family produced by the inverse scattering method. For the latter, the recent work of Wang--Liu \cite{WL20} establishes orbital stability of multi-solitons using the bi-Hamiltonian structure and the completeness of squared eigenfunctions. This provides a robust control on perturbations tangent to the multi-solitons manifold and serves as the natural entry point for the asymptotic theory. By contrast, in the KdV setting the work of Killip--Vi\c{s}an \cite{KV22} provides orbital stability at low regularity but does not address asymptotic convergence. In the CH setting, once orbital stability is known, the linear-in-time separation of the constituent solitons implies that every multi-solitons become, for large times, well approximated by a well-separated train. The absence of any persistent amplification of radiation modes and the exponential decay in weighted spaces ensure that cross-interactions between distinct solitons remain negligible for large times. As a result, each constituent soliton behaves asymptotically like an isolated soliton governed by its own spectral and dynamical parameters. This yields the asymptotic stability of multi-solitons for the CH equation, again without relying on explicit formulas for the soliton profiles or fine spectral data (exact negative eigenvalues and eigenfunctions), and without invoking any dispersive or virial identities.

\subsection{Organization of the paper}

The remaining parts of the paper  are organized as follows. In Section~\ref{sec_2}, we review the Hamiltonian formulation of the CH equation and collect several facts from the inverse scattering theory that will be used throughout the paper. These tools allow us to carry out a detailed spectral analysis of the recursion operator and of the linearized operator around smooth multi-solitons. We also establish the modulation theory for multi-solitons and derive the $H^{1}$-monotonicity formulas that play a crucial role in our arguments. With these preliminaries in place, the proof of Theorem~\ref{th1.2} proceeds in three main steps. Section~\ref{sec_3} contains the proof of asymptotic stability, which follows from the nonlinear Liouville theorem (Theorem~\ref{th1.1}). The remaining two steps concern the proof of Theorem~\ref{th1.1}. In Section~\ref{sec 4}, we reduce the nonlinear problem to the linear Liouville theorem (Theorem~\ref{thm1.4}) and establish the latter by combining the spectral information for the operator. $\mathcal{J}L_1$ with sharp semigroup decay estimates in the weighted space $H_a^1(\R)$. Section~\ref{sec 6} revisits and refines the orbital stability theory for soliton trains in $H^1$, extending the results of \cite{EM07}. Section~\ref{sec 7} contains the proof of the asymptotic stability of smooth multi-solitons. Finally, in Section~\ref{sec 8}, we show how our approach extends to yield rigidity results for other completely integrable models such as the KdV and mKdV equations.

\section{Preliminaries}\label{sec_2}
In this section we collect some preliminaries for the CH equation.
This section is divided into six parts. In Section \ref{sec_2.1}, we present the equivalent eigenvalue problem of the CH equation and the basic facts from which, through the inverse scattering transform, the conservation laws and multi-solitons of the CH equation are derived. In Section \ref{sec_2.2}, the bi-Hamiltonian formation of the CH equation is considered, and the recursion operators are introduced to compute the conservation laws at the multi-solitons. The spectrum of the recursion operators is obtained through the usage of squared eigenfunctions. This allows us to perform the spectral analysis for various linearized operators at the smooth solitons. The spectral information of the linearized operators at multi-solitons requires the understanding of the corresponding recursion operator at those multi-solitons, which is presented in Section \ref{subsec spec multi}. In Section \ref{sec_2.3} some necessary well-posedness results of the CH equations are presented. Section \ref{sec_2.4} aims at establishing $H^1$- monotonicity formulae related to the conservation laws and  decay properties for the $H^1$-localized solutions. Section \ref{sec_2.5} focuses on the modulation theory for multi-solitons.

%
%

\subsection{Eigenvalue problem, conservation laws and multi-solitons}\label{sec_2.1}

The CH equation \eqref{eq1} can be expressed as a compatibility condition of the following  Lax pair \cite{CH93}
\begin{eqnarray} \label{eq3}
 &&\Psi_{xx}=\left(\frac{1}{4}+\lambda
(m+\omega)\right)\Psi,
 \\\label{eq4}
&&\Psi_{t}=\left(\frac{1}{2\lambda}-u\right)\Psi_{x}+\frac{u_{x}}{2}\Psi+\eta\Psi,
\end{eqnarray}
 with a spectral parameter $\lambda $ and a constant $\eta$ for a proper normalization of the
eigenfunctions. Equation \eqref{eq3} is the spectral problem associated with \eqref{eq1}.  Let
$k^{2}=-\frac{1}{4}-\lambda \omega$, i.e.
\begin{equation} \label{lambda}
\lambda(k)= -\frac{1}{\omega}\Big( k^{2}+\frac{1}{4}\Big).
\end{equation}
The spectrum of the problem (\ref{eq3}) is
described in \cite{C,C01}. The continuous spectrum in terms of $\lambda$
corresponds to $k\in \R$. The discrete spectrum consists of
finitely many points $k_{n}=i\kappa _{n}$, $n=1,\ldots,N$ where
$\kappa_{n}$ is real and $0<\kappa_{n}<\frac12$.

A basis in the space of solutions of (\ref{eq3}) can be introduced
in the spirit of the Jost solutions of the CH equation,
$f^+(x,k)$ and $\bar{f}^+(x,\bar{k})$. For all real $k\neq 0$ the Jost solutions are completely determined by their asymptotics as $x\rightarrow +\infty$ \cite{C01}:
\begin{equation} \label{eq6}
\lim_{x\to +\infty }e^{-ikx}f^+(x,k)= 1.
\end{equation}
Another basis, $f^-(x,k)$
and $\bar{f}^-(x,\bar{k})$, can be introduced and determined by the asymptotics at
$x\rightarrow -\infty$ for all real $k\neq 0$:
\begin{equation} \label{eq6'}
\lim_{x\to -\infty }e^{ikx} f^-(x,k)= 1.
\end{equation}

Since $m(x)$ and $\omega$ are real one gets that if $f^+(x,k)$
and $f^-(x,k)$ are solutions of (\ref{eq3}), then
\begin{equation*}
 \bar{f}^+(x,\bar{k}) = f^+(x,-k), \qquad \mbox{and} \qquad
 \bar{f}^-(x,\bar{k}) = f^-(x,-k)
\end{equation*}
are also solutions of (\ref{eq3}). In particular, for real $k\neq 0$, one has
\begin{equation}
 \label{eq5aa} \bar{f}^{\pm}(x,k)=f^{\pm}(x, -k),
\end{equation}
and for each $k\in \mathbb{R}$ there exists $a(k), b(k) \in \mathbb{C}$ such that
\begin{equation} \label{eq8}
 f^{-}(x,k)=a(k)f^+(x,-k)+b(k)f^+(x,k).
\end{equation}
From  \eqref{eq8}, letting $x\rightarrow \pm\infty$, we obtain
\begin{eqnarray}
&&\lim_{x\rightarrow-\infty} \left( f^{+}(x,k)-a(k)e^{ikx}+b(-k)e^{-ikx} \right)=0, \nonumber\\
&&\lim_{x\rightarrow+\infty} \left( f^{-}(x,k)-a(k)e^{-ikx}-b(k)e^{ikx} \right)=0.\nonumber
\end{eqnarray}
 The Wronskian $W(f_{1},f_{2}) :=
f_{1}\partial_{x}f_{2}-f_{2}\partial_{x}f_{1}$ of any pair of
solutions of (\ref{eq3}) does not depend on $x$. Therefore
\begin{equation}
 \label{eq9} W(f^-(x,k), f^-(x,-k))= W(f^+(x,-k), f^+(x,k))=2ik.
\end{equation}
 Computing the Wronskians $W(f^-,f^+)$ and $W(\bar{f}^+,f^-)$
and using (\ref{eq8}), (\ref{eq9}) we obtain:
\begin{equation}\label{eq10a}
a(k)=(2ik)^{-1} W(f^-(x,k), f^+(x,k)), \qquad b(k)=-(2ik)^{-1} W(f^-(x,k), f^+(x,-k)).\
\end{equation}
%

From (\ref{eq8}) and (\ref{eq9}), it follows that for all $k \in \mathbb{R}$
\begin{equation}\label{eq10a1}
a(k)a(-k)-b(k)b(-k)=1.
\end{equation}
It is well known \cite{C01} that $f^+(x,k)e^{-ikx}$ and
$f^-(x,k)e^{ikx}$ have analytic extensions to the upper half of
the complex $k$-plane. Likewise
$\bar{f}^+(x,\bar{k})e^{i\bar{k}x}$ and $\bar{f}^-(x,\bar{k})
e^{-i\bar{k}x}$ allow analytic extension to the lower half of the
complex $k$-plane. An important consequence of these properties is
that $a(k)$ (and hence $b(k)$) also allows analytic extension to the upper half of
the complex $k$-plane and
\begin{equation*}
\bar{a}(\bar{k}) = a(-k), \qquad \bar{b}(\bar{k}) = b(-k),
\end{equation*}
As a result,  (\ref{eq10a1}) reduces to the form
\begin{equation} \label{eq10}
|a(k)|^{2}-|b(k)|^{2}=1, \qquad k \in \mathbb{R}.
\end{equation}

At the points $i\kappa_n$ of the discrete spectrum, $a(k)$ has
simple zeroes \cite{C01},
\begin{equation*}
    a(k) = (k-i\kappa_n)\dot{a}_n +\frac{1}{2} (k-i\kappa_n)^2\ddot{a}_n
    + \cdots,
\end{equation*}
and so that the Wronskian $W(f^-,f^+)$ in (\ref{eq10a}) vanishes. Thus
$f^-$ and $f^+$ are linearly dependent:
\begin{equation}\label{eq200}
f^-(x,i\kappa_n)=b_nf^+(x,i\kappa_n).
\end{equation}
In other words, the discrete spectrum is simple, and there is only
one (real) linearly independent eigenfunction corresponding to
each eigenvalue $i\kappa_n$
\begin{equation*}
f_n^-(x):= f^-(x,i\kappa_n).
\end{equation*}
From the above and (\ref{eq6}), (\ref{eq6'}), it follows
that $f_n^-(x)$ decays exponentially for $x\to\pm\infty$, which
allows us  to show that $f_n^-(x)\in L^2(\R)$. Moreover,
for compactly supported potentials $m(x)$, from (\ref{eq200}) and
(\ref{eq8}), we have
\begin{equation*}
b_n= b(i\kappa_n), \qquad
b(-i\kappa_n)=-\frac{1}{b_n} .
\end{equation*}
The above result can
be extended to Schwartz-class potentials by an appropriate limiting
procedure. The asymptotics of $f_n^-$, according to (\ref{eq5aa}),
(\ref{eq6}), (\ref{eq200}) are
\[
\begin{aligned}
& f_n^-(x)=e^{\kappa_n x}+o(e^{\kappa_n x}), &
\qquad & x\rightarrow -\infty, \\
& f_n^-(x)= b_n e^{-\kappa_n x}+o(e^{-\kappa_n x}), & \qquad &
x\rightarrow \infty.
\end{aligned}
\]
The sign of $b_n$ obviously depends on the number of the zeroes
of $f_n^-$. Suppose that
$0<\kappa_{1}<\ldots<\kappa_{N}<\frac12$, then from the
oscillation theorem for the Sturm-Liouville problem,
$f_n^-$ has exactly $n-1$ zeroes. Therefore
\begin{equation*}
 b_n= (-1)^{n-1}|b_n|.
\end{equation*}
The sets
\begin{equation*}
\mathcal{S^{\pm}} := \left\{
\frac{b(\pm k)}{a(k)}\quad (k>0),\quad \kappa_n,\quad
C_n^{\pm} := \frac{b_n^{\pm1}}{i\dot{a}_n},\quad n=1,\ldots N \right\}
\end{equation*}
 are called scattering data. Here the dot stands for the derivative
with respect to $k$ and $\dot{a}_n := \dot{a}(i\kappa_n)$,
$\ddot{a}_n:= \ddot{a}(i\kappa_n)$, etc.
The time evolution of the scattering data is presented in \cite{CGI07}
as follows.
\[
\begin{aligned}
& a(k,t)=a(k,0), & \qquad &  b(k,t)=b(k,0)e^{\frac{i
k}{\lambda }t}, \\
& \frac{1}{a(k,t)}=\frac{1}{a(k,0)}, & \qquad &
\frac{b(\pm k,t)}{a(k,t)}=\frac{b(\pm k,0)}{a(k,0)}e^{\pm\frac{i
k}{\lambda }t}, & \qquad C_n^{\pm} (t)= C_n^{\pm} (0)\exp \left( \pm\frac{4\omega
\kappa_n}{1-4\kappa_n^2}t \right).
\end{aligned}
\]
In other words, $a(k)$ is independent on $t$ and will serve as a
generating function of the conservation laws. In particular, the integral
\begin{equation*}
\Lambda:= \int _{-\infty}^{\infty} \left( \sqrt{\frac{m(x)+\omega}
{\omega}}-1 \right)\text{d}x,
\end{equation*}
as well as all the coefficients $\mathcal{I}_k$ in the asymptotic expansion
\begin{equation*}
 \ln a(k)= -i\Lambda k+\sum_{s=1}^{\infty}\frac{\mathcal{I}_{s}}{k^{2s+1}},
\end{equation*}
must be integrals of motion. The integral $\Lambda$ is the unique Casimir function for the CH equation.
The densities $p_{s}$ of  $\mathcal{I}_{s}=\int _{-\infty}^{\infty}p_{s}
\text{d}x$ can be expressed in terms of $m(x)$ using a set of
recurrent relations  obtained  in \cite{I05}.

Let us define the following {\em{squared eigenfunctions}}
\begin{equation*}
F^{\pm}(x,k):= (f^{\pm}(x,k))^2, \qquad F^{\pm}_n(x):=
F^{\pm}(x,i\kappa_n),\quad n=1,2,\cdot\cdot\cdot,N.
\end{equation*}
one can also derive the following
relations for the variations of the scattering data with respect to the potential $m$ as follows \cite{CI06}:
\begin{eqnarray} \label{eq37}
 &&\frac{\delta
a(k)}{\delta m(x)}=-\frac{\lambda}{2ik}f^+(x,k)f^-(x,k),\\
&&\frac{\delta b(k)}{\delta
m(x)}=\frac{\lambda}{2ik}f^+(x,-k)f^-(x,k), \label{eq38}\\
&&\frac{\delta
\ln \lambda_n}{\delta m(x)}=\frac{iF_n^-(x)}{\omega b_n
\dot{a}_n}. \label{eq39}
\end{eqnarray}
Moreover, from \cite[Proposition 5]{CGI07}, there holds the following useful completeness relation for the squared eigenfunctions:
 \begin{eqnarray} \label{completeness relations}
&& \frac{\omega \theta(x-y)}{\sqrt{m(x)+\omega} \sqrt{m(y)+\omega}} = -\frac{1}{2\pi i}\int_{\R}\frac{F^-(x,k)F^+(y,k)}{ka^2(k)}\rmd k \nonumber \\
&& \quad + \sum_{n=1}^N\frac{1}{i\kappa_n\dot{a}_n^2} \left[ F_n^-(x)\dot{F}_n^+(y)+F_n^-(x)\dot{F}_n^+(y) - \left( \frac{1}{i\kappa_n}+\frac{\ddot{a}_n}{\dot{a}_n} \right) F_n^-(x)F_n^+(y) \right],
\end{eqnarray}
where $\theta(x)$ is the step function, and again, dot denotes derivative in $k$. For any function $f(x)$ which vanishes for $x\rightarrow\pm\infty$, one can expand it over the squared eigenfunctions $F^+(x,k)$ and $F^-(x,k)$. Indeed, we multiply \eqref{completeness relations} by $\frac12m_yf(y)+(m(y)+\omega)f_y$ and integrate over $y$ to have
\begin{equation}\label{refined completeness relations}
\pm \omega f(x)=-\frac{1}{2\pi i}\int_{\R}\frac{F^{\mp}(x,k)\xi^{\pm}_{f}(k)}{ka^2(k)}\rmd k+\sum_{n=1}^N\frac{1}{i\kappa_n\dot{a}_n^2} \left[ \dot{F}_n^{\mp}(x)\xi^{\pm}_{f,n}+F_n^{\mp}(x)\dot{\xi}^{\pm}_{f,n} \right],
\end{equation}
where
\begin{align*}
\xi^{\pm}_{f}(k) & = \int_{\R}F^{\pm}(y,k) \left[ \frac12m_yf(y)+(m(y)+\omega)f_y \right]\rmd y, \\
\xi^{\pm}_{f,n} & = \int_{\R}F_n^{\pm}(y) \left[ \frac12m_yf(y)+(m(y)+\omega)f_y \right]\rmd y, \\
\dot{\xi}^{\pm}_{f,n} & = \int_{\R}\dot{F}_n^{\pm}(y) \left[ \frac12m_yf(y)+(m(y)+\omega)f_y \right]\rmd y - \left( \frac{1}{i\kappa_n}+\frac{\ddot{a}_n}{\dot{a}_n} \right)\xi^{\pm}_{f,n}.
\end{align*}
%

The inverse scattering is simplified into the important case of the so-called
reflectionless potentials, where the scattering data is confined to the case
the reflection coefficient $\frac{b(k)}{a(k)} = 0$ for all real $k$. This class of potentials corresponds to the $N$-solitons of the CH equation. In this case $b(k) = 0$ and $|a(k)| = 1$ (see \eqref{eq10}). The $N$-solitons can be calculated in a parametric form \cite{CGI06} which depends on the discrete spectrum $k_n=i\kappa_n$ of \eqref{eq3} for $n=1,2,\cdot\cdot\cdot,N$ and $0<\kappa_{1}<\kappa_{2}<\ldots<\kappa_{N}<\frac12$, which are the zeros $a(k)$ in the imaginary axis. Moreover,
\[
i\dot{a}_n=\frac{1}{2\kappa_n}e^{\Lambda\kappa_n}\prod_{j\neq n}\frac{\kappa_n-\kappa_j}{\kappa_n+\kappa_j}, \quad\text{where} \quad \Lambda=\sum_{n=1}^{N}\ln\Big(\frac{1+2\kappa_n}{1-2\kappa_n}
\Big)^2.
\]
The CH $N$-solitons can be expressed in a parametric form as follows \cite{CGI06}
\begin{equation*}
u(t,x)=\frac{\omega}{2}\int_0^\infty e^{-|x-g(t,\xi)|}\xi^{-2}g_\xi^{-1}(t,\xi)\rmd \xi-\omega,
\end{equation*}
where $g(t,\xi)$ can be expressed through the scattering data as
\[
g(t,\xi):=\ln\int_0^{\xi} \left[ 1-\sum_{n,p}\frac{C_n^{+}(t)y^{-2\kappa_n}}{\kappa_n+1/2}A_{np}(t,y) \right]^{-2}\rmd y, \quad \text{with} \quad A_{np}(t,y):=\delta_{pn}+\frac{C_n^{+}(t)y^{-2\kappa_n}}{\kappa_n+\kappa_p}.
\]
The CH $N$-solitons are also presented in \cite{M05} in a parametric form using elementary determinant theory. The author demonstrates that these CH $N$-solitons exhibit the asymptotic behavior \eqref{eq:asymptotic behavior} as time approaches infinity, along with a formula for the phase shift.

\subsection{Recursion operators and completeness relations of the squared eigenfunctions}\label{sec_2.2}
In this subsection, we analyze the spectral information of recursion operators and the linearized operators around the soliton profile $\varphi_c$, which will be crucial in our proof of linear Liouville property (see Theorem \ref{thm1.4} below).

From the bi-Hamiltonian structure \eqref{eq:recursion}, Lenard relation \eqref{eq2aa} and the relation $(1-\partial_x^2)\frac{\delta \Ham_{n}[m]}{\delta m}=\frac{\delta \Ham_{n}(u)}{\delta u}$, one has the following relations for $\Ham_n$,
\begin{equation}\label{eq:recursion1}
 \frac{\delta \Ham_{n+1}[m]}{\delta m} = (1-\partial_x^2)^{-1} \left[ 2\omega +m+\partial_x^{-1} ( m\partial_x) \right] \frac{\delta \Ham_{n}[m]}{\delta m} =: \mathcal{K}[m]\frac{\delta \Ham_{n}[m]}{\delta m},
 \end{equation}
 \begin{equation} \label{eq:recursion2}
 \frac{\delta \Ham_{n+1}(u)}{\delta u} = \left[ 2\omega +m+\partial_x^{-1}( m\partial_x) \right] (1-\partial_x^2)^{-1}\frac{\delta \Ham_{n}(u)}{\delta u} =: \mathcal{R}(u)\frac{\delta \Ham_{n}(u)}{\delta u},
\end{equation}
where $\partial_x^{-1}$ is the inverse derivative operator defined by $\partial_x^{-1}f=\int_{-\infty}^xf\rmd x$ with $f\in \mathcal S(\R)$. If we take $\mathcal S^\ast(\R):= \left\{\partial_x^{-1}f : \ f\in \mathcal S(\R) \right\}$, with a bilinear form given by
\[
\langle f,g\rangle=\int_{-\infty}^{\infty}f(x)g(x)\rmd x,
\]
then $\partial_x^{-1}:\mathcal S^\ast(\R)\rightarrow \mathcal S(\R)$ is skew-symmetric and $\partial_x\partial_x^{-1}=id$.

It is not difficult (from the bi-Hamiltonian structure \eqref{eq:recursion}) to see
 \begin{equation*}
 \mathcal{K}[m]=\mathcal{J}_2^{-1}\mathcal{J}_1.
 \end{equation*}
Moreover, the linear operators $\mathcal{K}(m)$ and $\mathcal{R}(u)$ are nonlocal, satisfying
\begin{equation}\label{eq:similar}
\mathcal{R}(u) = (1-\partial_x^2)\mathcal{K}[m](1-\partial_x^2)^{-1},
\end{equation}
namely, they are similar to each other. Moreover, one can check that it holds that
\begin{equation*}
\mathcal{K}[m] \left( 1-\sqrt{\frac{\omega}{m+\omega}} \right) = u, \qquad \mathcal{R}(u)(1-\partial_x^2) \left( 1-\sqrt{\frac{\omega}{m+\omega}} \right) = m.
\end{equation*}
These can be viewed as in a special case $n=0$ in \eqref{eq:recursion1} and \eqref{eq:recursion2}, respectively. The associated conservation laws are
 \[
 \Ham_0[m] := \int_{\R} \left( \sqrt{m(x)+\omega}-\sqrt{\omega} \right)^2\rmd x, \qquad \frac{\delta \Ham_{0}[m]}{\delta m}=1-\sqrt{\frac{\omega}{m+\omega}}.
 \]
Therefore, the CH hierarchy with the choice of the dispersion law $\Omega(z)$ can be expressed as follows:
\begin{equation*}
m_t + \sqrt{m+\omega} \left[ \sqrt{m+\omega} \ \Omega(2\mathcal{K}[m]) \left(1-\sqrt{\frac{\omega}{m+\omega}} \right) \right]_x=0.
\end{equation*}
In particular, if the dispersion law $\Omega(z)=z$, then the above becomes the CH equation.

Denoting adjoint operator of $\mathcal{K}[m]$ by $\mathcal{K}^{\ast}[m]:=\mathcal{J}_1\mathcal{J}_2^{-1}$, then the adjoint of $\mathcal{R}(u)$ is
\begin{equation}\label{eq:adjoint of R}
\mathcal{R}^{\ast}(u)=(1-\partial_x^2)^{-1}\mathcal{K}^{\ast}[m](1-\partial_x^2)=(1-\partial_x^2)^{-1}\mathcal{J}_1\mathcal{J}_2^{-1}(1-\partial_x^2),
\end{equation}
and it is not difficult to see that the operators $\mathcal{R}(u)$ and $\mathcal{R}^{\ast}(u)$ satisfy
\begin{equation}\label{eq:R Rstar}
\mathcal{R}^{\ast}(u)\mathcal{J}=\mathcal{J}\mathcal{R}(u).
\end{equation}
The understanding of the spectral information of the recursion operators $\mathcal{R}(u)$ and $\mathcal{R}^{\ast}(u)$ is essential in the (spectral) stability of multi-solitons of the CH equation \eqref{eq1}, see \cite{WL20}.

It is shown in \cite{CGI07} that the conservation laws $\Ham_j$ can be expressed as follows:
\begin{equation}\label{eq:conservation laws}
\Ham_j(u)=-\int_0^{\infty}(-2)^{1-j}\frac{k\rho(k)}{\lambda^j}dk+\frac{(-2)^{2-j}}{\omega}\sum_{n=1}^N\int\frac{\kappa_n^2}{\lambda_n^{j+2}}d\kappa_n,
\end{equation}
where $\rho(k):=\frac{2k}{\pi \omega\lambda^2}\ln|a(k)|$. The first term on the right-hand side of \eqref{eq:conservation laws} is the contribution
from radiations and the second term comes from the solitons.

Let us consider the quantities $\Ham_j(\varphi_c)$ which are related to the $1$-soliton profile $\varphi_c$. From \eqref{eq:stationary} and \eqref{eq:recursion2} it follows that the $1$-soliton $\varphi_c$ with speed $c$ satisfies the following variational principle
\begin{equation*}
\frac{\delta }{\delta u}\left(-\Ham_{j+1}(u) + c \Ham_{j}(u)\right)=0, \quad  j\in \mathbb Z_+,
\end{equation*}
that is,
\begin{equation}\label{eq:1-sol variaprinciple}
-\Ham_{j+1}'(\varphi_c) + c \Ham'_{j}(\varphi_c)=0.
\end{equation}
Now multiply  \eqref{eq:1-sol variaprinciple} with $\frac{ d\varphi_c}{d c}$. For each $j$ one has
\[
\frac{\rmd \Ham_{j+1}(\varphi_c)}{\rmd c}=c\frac{\rmd \Ham_{j}(\varphi_c)}{\rmd c} = \cdots =c^j\frac{\rmd \Ham_{1}(\varphi_c)}{\rmd c},
\]
and therefore
\begin{equation*}
\Ham_{j+1}(\varphi_c)=\int_0^cy^j\frac{\rmd \Ham_{1}(\varphi_y)}{\rmd y}\rmd y,
\end{equation*}
Calculating $\Ham_{1}(\varphi_c)$ is challenging because $\varphi_c$ is not given explicitly. However, considering that the reflection coefficient $\rho(k)$ vanishes for $u = U^{(N)}$, with associated discrete eigenvalues $i\kappa_n$ for $n=1,2,\ldots,N$ where $0 < \kappa_n < \frac12$. In particular, for the one-soliton case of $\varphi_c$ with discrete eigenvalue $i\kappa$, we can derive the following formula from \eqref{eq:conservation laws}:
%
\begin{equation}\label{eq:1-sol value}
\Ham_j(\varphi_c)=\frac{(-2)^{2-j}}{\omega}\int\frac{\kappa^2}{\lambda^{j+2}}d\kappa, \quad \ \lambda=-\frac{1}{\omega}\Big( \frac{1}{4}-\kappa^2\Big)<0.
\end{equation}
More precisely, one has the following for the conservation laws $\Ham_1$ and $\Ham_2$
\begin{equation*}
\Ham_1(\varphi_c)=\omega^2 \left[ \ln\frac{1-2\kappa}{1+2\kappa}+\frac{4\kappa(1+4\kappa^2)}{(1-4\kappa^2)^2} \right], \qquad
\Ham_2(\varphi_c)=\omega^3 \left[ \ln\frac{1-2\kappa}{1+2\kappa}+\frac{4\kappa(3+32\kappa^2-48\kappa^4)}{3(1-4\kappa^2)^3} \right].
\end{equation*}
From \eqref{eq:1-sol variaprinciple} (multiplying it with $\frac{d \varphi_c}{ d\kappa}$) and \eqref{eq:1-sol value}, one has
\[
\frac{(-2)^{1-j}\kappa^2}{\omega\lambda^{j+3}}=\frac{\rmd \Ham_{j+1}(\varphi_c)}{\rmd \kappa} = c\frac{\rmd \Ham_{j}(\varphi_c)}{\rmd \kappa}=c\frac{(-2)^{2-j}\kappa^2}{\omega\lambda^{j+2}}.
\]
On the other hand, one  can represent the wave speed $c$ with respect to $\kappa$ as follows
\begin{equation*}
c=-\frac{1}{2\lambda}=\frac{2\omega}{1-4\kappa^2}>2\omega.
\end{equation*}
Therefore, the quantities $\Ham_j(\varphi_c)$ can be computed explicitly with respect to the wave velocity $c$.
Moreover, from \eqref{eq:1-sol value} and the above wave speed formula, the derivative of  $\Ham_1(\varphi_c)$ with respect to the wave speed $c$ can be computed in the following form,
\begin{equation*}
\frac{\rmd \Ham_1(\varphi_c)}{\rmd c}=\frac{\rmd \Ham_1(\varphi_c)}{\rmd \kappa}\frac{\rmd \kappa}{\rmd c}= - \frac{\kappa(1-4\kappa^2)^2}{8\omega^2\lambda^3} =4\kappa c> 0.
\end{equation*}
Similarly, from \eqref{eq:1-sol variaprinciple} and the above, one has
\begin{equation}\label{eq:Fderivative to c}
\frac{\rmd \Ham_2(\varphi_c)}{\rmd c}=c\frac{\rmd \Ham_1(\varphi_c)}{\rmd c}=4\kappa c^2> 0.
\end{equation}
For simplicity, we denote $\varphi_c$ by $\varphi$. Then by \eqref{eq:recursion2}, we have
\begin{equation*}
 \Ham'_{n+1}(\varphi)=\mathcal{R}(\varphi) \Ham'_{n}(\varphi),
\end{equation*}
where $\mathcal{R}(\varphi)$ is given in \eqref{eq:recursion2}
with $m = m_{\varphi}:=\varphi-\varphi_{xx}$.

The recursion operator $\mathcal{R}(\varphi)$ and the adjoint recursion operator $\mathcal{R}^{\ast}(\varphi)$ (see \eqref{eq:adjoint of R}) play important roles in the spectral analysis of linearized operator
\begin{eqnarray}\label{formula of Ln}
L_n := -\Ham''_{n+1}(\varphi) + c\Ham''_{n}(\varphi),
\end{eqnarray} through the following iterative and commutative operator identities.
\begin{lemma}[\cite{WL20}]\label{le3.5}
For all $n \in \mathbb{N}$, the recursion operator $\mathcal{R}(\varphi)$, the adjoint recursion operator $\mathcal{R}^{\ast}(\varphi)$, and the linearized operator $L_n$ satisfy the following operator identities.
\begin{eqnarray}
&&-\Ham''_{n+2}(\varphi)+c\Ham''_{n+1}(\varphi)=\mathcal{R}(\varphi) \left[-\Ham''_{n+1}(\varphi)+c\Ham''_{n}(\varphi) \right],\label{eq:recursion Hamiltonianre}\\
&&L_n\mathcal{J}\mathcal{R}(\varphi)=\mathcal{R}(\varphi)L_n\mathcal{J}, \label{operator identity1}\\
&&\mathcal{J}L_n\mathcal{R}^{\ast}(\varphi)=\mathcal{R}^{\ast}(\varphi)\mathcal{J}L_n,\label{operator identity2}
\end{eqnarray}
where $\mathcal{J}$ is the skew symmetric operator defined in \eqref{Hamilton form}.
\end{lemma}
\begin{proof}
Identity \eqref{eq:recursion Hamiltonianre} follows from \eqref{eq:recursion2} and the definition of Gateaux derivative and variational principle of soliton profile $\Ham_n'(\varphi)=c^{n-1}\Ham_1'(\varphi)=c^{n-1}m_{\varphi}$.

Thus we only need to show \eqref{operator identity2}, since \eqref{operator identity1} can be obtained by taking adjoint on \eqref{operator identity2}.  Notice that from
\eqref{eq:recursion Hamiltonianre}, one has that the operator $\mathcal{R}(\varphi)L_n=L_{n+1}$ is self-adjoint. This in turn implies that $$(\mathcal{R}(\varphi)L_n)^\ast=\mathcal{R}(\varphi)L_n=L_n\mathcal{R}^\ast(\varphi),$$
On the other hand, in view of  \eqref{Hamilton form} and \eqref{eq:R Rstar}, one has
$$
\mathcal{J}L_n\mathcal{R}^\ast(\varphi)=\mathcal{J}\mathcal{R}(\varphi)L_n=\mathcal{R}^\ast(\varphi)\mathcal{J}L_n,
$$
which is the desired \eqref{operator identity2}.
\end{proof}

\begin{remark}\label{re3.6}
Identities \eqref{operator identity1} and \eqref{operator identity2} hold also for the multi-solitons of the CH equation.
In particular, let $U^{(N)}$ be the CH $N$-soliton profile. The corresponding Lyapunov functional $I_N$ is given by \cite{WL20}
\begin{equation*}
I_N(u)=(-1)^{N} \left[\Ham_{N+1}(u)+\sum_{n=1}^N\mu_n \Ham_{n}(u) \right],
\end{equation*}
where $\mu_n$ are Lagrange multipliers determined by
\begin{equation*}
\mu_n=(-1)^{N-n+1}\sigma_{N-n+1},
\end{equation*}
with $\sigma_s$ being elementary symmetric functions of $c_1, c_2, ..., c_N$
\[
\sigma_1=\sum_{j=1}^Nc_j, \quad \sigma_2=\sum_{j<k}c_jc_k, \quad \ldots, \quad
 \sigma_N=\prod_{j=1}^Nc_j.
\]
The self-adjoint second variation operator is defined as follows
\begin{eqnarray}\label{eq:linearized n-soliton operator}
\mathcal L_N:=I_N''(U^{(N)}).
\end{eqnarray}
 Then it is easy to see that similar to  Lemma \ref{le3.5}, the following operator identities hold
\begin{eqnarray}
&&\mathcal{L}_N\mathcal{J}\mathcal{R}(U^{(N)})=\mathcal{R}(U^{(N)})\mathcal{L}_N\mathcal{J},\label{operator identity3}\\
&&\mathcal{J}\mathcal{L}_N\mathcal{R}^{\ast}(U^{(N)})=\mathcal{R}^{\ast}(U^{(N)})\mathcal{J}\mathcal{L}_N.\label{operator identity4}
\end{eqnarray}
\end{remark}

\medskip

An immediate consequence of \eqref{operator identity1} is that $\mathcal{R}(\varphi)$ commutes with $L_n\mathcal{J}$, and thus they share the same eigenfunctions. Similarly,  \eqref{operator identity2} implies that the same property holds for $\mathcal{R}^\ast(\varphi)$ and $\mathcal{J}L_n$.

We will conclude this subsection with the following results revealing the spectral information of the recursion operator $\mathcal{R}(\varphi)$, its adjoint $\mathcal{R}^\ast(\varphi)$, the linear operators $\mathcal{J}L_n$ and $L_n\mathcal{J}$. Note that the recursion operators are nonlocal which are not easy to study directly. However, by employing the operator identity \eqref{eq:similar} and the properties of the squared eigenfunctions $F^{\pm}(x,k)$, one could have the following result.

\begin{lemma}[Spectrum of $\mathcal{R}(\varphi)$] \label{lem3.6}
The recursion operator $\mathcal{R}(\varphi)$, defined in $L^2(\R)$ with domain $H^1(\R)$,  has only one  eigenvalue $c$  associated with the eigenfunction $m_{\varphi}$; the essential spectrum is the interval $(0,2\omega]$; and the corresponding eigenfunctions do not have spatial decay and are not in $L^2(\R)$. Moreover, the kernel of $\mathcal{R}(\varphi)$ is trivial, and  the inverse of $\mathcal{R}(\varphi)$ reads
\[
\mathcal{R}^{-1}(\varphi) = \frac12(1-\partial_x^2)\frac{1}{\sqrt{m_\varphi+\omega}}\partial_x^{-1}\frac{1}{\sqrt{m_\varphi+\omega}}\partial_x.
\]
\end{lemma}

\begin{proof} Consider the Jost solutions $f^{\pm}(x,k)$ of the eigenvalue problem \eqref{eq3} with the potential $m(x)=m_\varphi=\varphi-\varphi_{xx}$ and the asymptotics in \eqref{eq6} and \eqref{eq6'}. In this case there is an eigenvalue $k=i\kappa_1$ ($0<\kappa_1<\frac12$) which generates the soliton profile $\varphi_c$. It is then found that the squared eigenfunctions $F^{\pm}(x,k)=(f^{\pm}(x,k))^2$ satisfy
\begin{equation*}
 \bigg(-\partial_x^2+1+2\lambda\big(m_\varphi(x)+2\omega+\partial_x^{-1}(m_\varphi(x)\partial_x)\big)\bigg)F^{\pm}(x,k)=0.
\end{equation*}
This in turn implies that
\begin{equation}\label{eq:Ksquared eigen}
(1-\partial_x^2)^{-1}\left( m_\varphi(x)+2\omega+\partial_x^{-1}(m_\varphi(x)\partial_x)\right)F^{\pm}(x,k)=\mathcal{K}[m_\varphi]F^{\pm}(x,k)=-\frac{1}{2\lambda}F^{\pm}(x,k).
\end{equation}
By \eqref{eq:Ksquared eigen} and \eqref{eq:similar}, it follows that for $k \in \R$,
\begin{eqnarray}
&&\mathcal{R}(\varphi)(1-\partial_x^2)F^{\pm}(x,k)=-\frac{1}{2\lambda}(1-\partial_x^2)F^{\pm}(x,k), \label{eq:Reigen} \\
&&\mathcal{R}(\varphi)(1-\partial_x^2)F_1^{\pm}(x)=-\frac{1}{2\lambda_1}(1-\partial_x^2)F_1^{\pm}(x)=c(1-\partial_x^2)F_1^{\pm}(x).\label{eq:Reigen1}
\end{eqnarray}
Moreover, in view of the element $\dot{F}_1^{\pm}(x)$ in the completeness relation \eqref{completeness relations}, it follows that
\begin{equation*}
\mathcal{R}(\varphi)(1-\partial_x^2)\dot{F}_1^{\pm}(x)=-\frac{1}{2\lambda_1}(1-\partial_x^2)\dot{F}_1^{\pm}(x)-\frac{i\kappa_1}{\omega\lambda_1^2}(1-\partial_x^2)F_1^{\pm}(x).
\end{equation*}
Since $\mathcal{R}(\varphi)m_{\varphi}=cm_{\varphi}$, it is  inferred  from \eqref{eq:Reigen1} that $F_1^{\pm}(x)\sim \varphi(x)$. On account of  \eqref{eq:Reigen}, the essential spectrum of $\mathcal{R}(\varphi)$ is given by the set $\left\{ \lambda : -\frac{1}{2\lambda}=\frac{\omega}{2k^2+\frac12}, k\in \R \right\}$, which is exactly $(0,2\omega]$.
The associated generalized eigenfunctions $(1-\partial_x^2)F^{\pm}(x,k)$ possess no spatial decay and thus are not in $L^2(\R)$, which can be seen from \eqref{eq6} and \eqref{eq6'}.

On the other hand, a simple direct computation shows that the kernel of $\mathcal{R}(\varphi)$ is trivial except for $\omega=0$. In particular, consider $v(x):=(1-\partial_x^2)(1/\sqrt{m_\varphi+\omega})$. One deduces that $\mathcal{R}(\varphi)v=\omega>0$. The inverse of $\mathcal{R}(\varphi)$ can also be verified directly. The proof of the lemma is complete.
\end{proof}

\begin{remark} \label{re3.7} In view of  \eqref{eq:similar} or \eqref{eq:Ksquared eigen}, one can conclude that the operator $\mathcal{K}[m_\varphi]$ shares the same spectra with the operator $\mathcal{R}(\varphi)$. For instance, we have $\mathcal{K}(m_\varphi)\varphi=c\varphi$, the kernel of which is trivial and the inverse is $$\mathcal{K}^{-1}[m_\varphi]=\frac{1}{2\sqrt{m_\varphi+\omega}}\partial_x^{-1}\frac{1}{\sqrt{m_\varphi+\omega}}(1-\partial_x^2)^{-1}\partial_x.$$
\end{remark}


Next consider the adjoint recursion operator $\mathcal{R}^\ast(\varphi)$. As discussed before, $\mathcal{R}^\ast(\varphi)$ and $\mathcal{J}L_n$ share the same eigenfuntions, and thus is more relevant to the spectral stability problems of solitons. Recall from \eqref{eq:adjoint of R} that
$$
\mathcal{R}^{\ast}(u)=(1-\partial_x^2)^{-1}\mathcal{J}_1\mathcal{J}_2^{-1}(1-\partial_x^2).
$$
The observation of \eqref{eq:Ksquared eigen} reveals that
$$
\mathcal{J}_1\mathcal{J}_2^{-1}\mathcal{J}_1F^{\pm}(x,k)=\mathcal{J}_1\mathcal{K}[m]F^{\pm}(x,k)=-\frac{1}{2\lambda}\mathcal{J}_1F^{\pm}(x,k).
$$
From above one can verify that the eigenfunctions of $\mathcal{R}^{\ast}(u)$ are
\begin{equation}\label{eq:Kstarsquared eigen}
(1-\partial_x^2)^{-1}\mathcal{J}_1F^{\pm}(x,k)=-\partial_x \mathcal{K}[m]F^{\pm}(x,k)=\frac{1}{2\lambda}(F^{\pm}(x,k))_x.
\end{equation}

\begin{lemma}[Spectrum of $\mathcal{R}^\ast(\varphi)$] \label{le3.81}
The adjoint recursion operator $\mathcal{R}^\ast(\varphi)$, defined in $L^2(\R)$ with domain $H^1(\R)$, has only one  eigenvalue $c$ associated with the eigenfunction $\varphi_x$; the essential spectrum is the interval $(0,2\omega]$; and the corresponding eigenfunctions do not have spatial decay and hence are not in $L^2(\R)$. Moreover, the kernel of $\mathcal{R}^\ast(\varphi)$ is trivial, and  the inverse of $\mathcal{R}^\ast(\varphi)$ is $$\big(\mathcal{R}^\ast(\varphi)\big)^{-1}=\frac{1}{2}\partial_x\frac{1}{\sqrt{m_\varphi+\omega}}\partial_x^{-1}\frac{1}{\sqrt{m_\varphi+\omega}}(1-\partial_x^2)^{-1}.$$
\end{lemma}
\begin{proof}
Consider the Jost solutions $f^{\pm}(x,k)$ of the eigenvalue problem \eqref{eq3} with the potential $m_\varphi$ and the asymptotics in  \eqref{eq6} and \eqref{eq6'}. The soliton profile $\varphi$ is  generated by the eigenvalue $k=i\kappa_1$ ($0<\kappa_1<\frac12$). By \eqref{eq:Kstarsquared eigen},  the following relations hold
\begin{eqnarray*}
&&\mathcal{R}^\ast(\varphi)\big(F^{\pm}(x,k)\big)_x=-\frac{1}{2\lambda}\big(F^{\pm}(x,k)\big)_x,\quad \text{for} \ k\in\R, \\ 
&&\mathcal{R}^\ast(\varphi)\big(F_1^{\pm}(x)\big)_x=-\frac{1}{2\lambda_1}\big(F_1^{\pm}(x)\big)_x=c\big(F_1^{\pm}(x)\big)_x,  \\ 
&&\mathcal{R}^\ast(\varphi)\big(\dot{F}_1^{\pm}(x)\big)_x=-\frac{1}{2\lambda_1}\big(\dot{F}_1^{\pm}(x)\big)_x-\frac{i\kappa_1}{\omega\lambda_1^2}\big(F_1^{\pm}(x)\big)_x. 
\end{eqnarray*}
Since  $\mathcal{R}^\ast(\varphi)\varphi_x=c\varphi_x$, by the second equation above, $c$ is the only eigenvalue. In view of the first equation above, the essential spectrum of $\mathcal{R}^\ast(\varphi)$ is  $\left\{\lambda : -\frac{1}{2\lambda}=\frac{\omega}{2k^2+\frac12}, k \in \mathbb{R} \right\}$, which is the interval $(0,2\omega]$.
From \eqref{eq6} and \eqref{eq6'} we see that the associated generalized eigenfunctions $\big(F^{\pm}(x,k)\big)_x$ do not decay and hence are not in $L^2(\R)$.

Similarly, a direct computation shows that the kernel of $\mathcal{R}^\ast(\varphi)$ is trivial except for $\omega=0$. The inverse of $\mathcal{R}^\ast(\varphi)$ can be verified directly. This completes the proof of Lemma \ref{le3.81}.
\end{proof}

Now we are in position of deriving the spectrum of linear operators $\mathcal{J}L_n$ and $L_n\mathcal{J}$. This approach is motivated by \eqref{operator identity1} and \eqref{operator identity2}, allowing us to reduce the problem to analyzing the spectrum of the recursion operator $\mathcal{R}(\varphi)$ and the adjoint recursion operator $\mathcal{R}^\ast(\varphi)$. Next, we show that the eigenfunctions of $\mathcal{R}^\ast(\varphi)$ (or equivalently, $\mathcal{J}L_n$), together with the generalized kernel of $\mathcal{J}L_n$, form an orthogonal basis in $L^2(\R)$. This serves as a completeness relation, analogous to \eqref{completeness relations} and \eqref{refined completeness relations}. It then follows that the spectrum of $\mathcal{J}L_n$ lies on the imaginary axis which yields the spectral stability of solitons.

Start with the operator $\mathcal{J}L_n$. Since $L_n=\mathcal{R}^{n-1}(\varphi)L_1$, whose principle part is
\[
\left[ 2\omega(1-\partial_x^2)^{-1} \right]^{n-1}(-c\partial_x^2+c-2\omega),
\]
the symbol of the principle part of the operator $\mathcal{J}L_n$ is
\begin{equation}\label{eq:symbol JL}
\varrho_{n,c}(\zeta):=-2^{n-1}\omega^{n-1}\frac{i\zeta(c\zeta^2+c-2\omega)}{(1+\zeta^2)^n}, \ \zeta\in\R.
\end{equation}

\begin{proposition} \label{pr3.5-3}
For $n \ge 1$, the operators $\mathcal{J}L_n$ defined in $L^2(\R)$ with domain $H^1(\R)$  and the adjoint recursion operator $\mathcal{R}^\ast(\varphi)$ share the same eigenfunctions. Moreover, the essential spectrum of  $\mathcal{J}L_n$ is contained in $i\R$, the kernel is spanned by the function $\varphi_x$, and the generalized kernel is spanned by $\frac{\partial \varphi}{\partial c}$.
\end{proposition}
\begin{proof}
The fact that the operators $\mathcal{J}L_n$ and the adjoint recursion operator $\mathcal{R}^\ast(\varphi)$ share the same eigenfunctions is inferred by the operator identity \eqref{operator identity2}. By Lemma \ref{le3.81}, one can compute the spectrum of the operator $\mathcal{J}L_n$ directly by employing the squared eigenfunctions as follows
\begin{align*}
& \mathcal{J}L_n\big(F^{\pm}(x,k)\big)_x = \varrho_{n,c}(\pm2k)\big(F^{\pm}(x,k)\big)_x, \quad \text{for} \quad k \in \R, \\
& \mathcal{J}L_n\big(F_1^{\pm}(x)\big)_x\sim\mathcal{J}L_n \varphi_x = 0, \qquad \mathcal{J}L_n\big(\dot{F}_1^{\pm}(x)\big)_x\sim\mathcal{J}L_n\frac{\partial \varphi}{\partial c}= c^{n-1} \varphi_x.
\end{align*}
In view of \eqref{eq:symbol JL}, the essential spectrum of $\mathcal{J}L_n$ is  $\varrho_{n,c}(\pm2k)$ for $ k\in\R$ which is contained in the imaginary axis, which gives the desired result in Proposition \ref {pr3.5-3}.
\end{proof}

For the adjoint operator of $\mathcal{J}L_n$, namely the operator $-L_n\mathcal{J}$, which commutes with the recursion operator $\mathcal{R}(\varphi)$ \eqref{operator identity1}, we have the following result.

\begin{proposition} \label{le3.82}
For $n \ge 1$, the operators $L_n\mathcal{J}$ defined in $L^2(\R)$ with domain $H^1(\R)$ and the recursion operator $\mathcal{R}(\varphi)$ share the same eigenfunctions. Moreover, the essential spectrum of  $L_n\mathcal{J}$ is contained in $i\R$, the kernel is spanned by the function $m_{\varphi}$ and the generalized kernel is spanned by $\partial_x^{-1}\big(\frac{\partial m_{\varphi}}{\partial c}\big)$.
\end{proposition}
\begin{proof}
From \eqref{operator identity1} we see immediately that $L_n\mathcal{J}$ and $\mathcal{R}(\varphi)$ share the same eigenfunctions. By Lemma \ref{re3.6}, one can compute the spectrum of the operator $L_n\mathcal{J}$ directly by employing the squared eigenfunctions as follows
\begin{align*}
& L_n\mathcal{J}(1-\partial_x^2)F^{\pm}(x,k)=-L_n\big(F^{\pm}(x,k)\big)_x=\varrho_{n,c}(\pm2k)(1-\partial_x^2)F^{\pm}(x,k), \\
& L_n\mathcal{J}(1-\partial_x^2)F_1^{\pm}(x)\sim L_n\mathcal{J}m_{\varphi}=-L_n\varphi_x=0, \\
& L_n\mathcal{J}(1-\partial_x^2)\dot{F}_1^{\pm}(x)=-L_n\big(\dot{F}_1^{\pm}(x)\big)_x\sim L_n\frac{\partial \varphi}{\partial c}=-c^{n-1}m_{\varphi}.
\end{align*}
In view of \eqref{eq:symbol JL}, the essential spectrum of $L_n\mathcal{J}$ is $\varrho_{n,c}(\pm2k)$ for $ k\in\R$, which is purely imaginary. On the other hand, it is inferred from  the last equality above that $\big(\dot{F}_1^{\pm}(x)\big)_x\sim \frac{\partial \varphi}{\partial c}$.  Hence the generalized kernel is
\[
(1-\partial_x^2)\dot{F}_1^{\pm}(x)\sim \partial_x^{-1} \left( \frac{\partial m_{\varphi}}{\partial c}\right),
\] which implies the desired result in  the proposition.
\end{proof}

On account of  Proposition \ref{pr3.5-3} and Proposition \ref{le3.82},  we now have the two function sets as follows. The first set
\begin{equation}\label{eq:set JL}
\left\{\big(F^{\pm}(x,k)\big)_x \ \  \text{for }  k\in\R;\quad \varphi_x; \quad \frac{\partial \varphi}{\partial c}  \right\}
\end{equation}
consists of  linearly independent eigenfunctions and generalized kernel of the operator $\mathcal{J}L_n$. Moreover, they are essentially orthogonal under the $L^2$-inner product.
The second set
\begin{equation}\label{eq:set LJ}
\left\{(1-\partial_x^2)F^{\pm}(x,k) \ \  \text{for}  \  k\in\R;\quad m_{\varphi}; \quad\partial_x^{-1} \left( \frac{\partial m_{\varphi}}{\partial c} \right)  \right\}
\end{equation}
consists of linearly independent eigenfunctions and generalized kernel of the operator $L_n\mathcal{J}$. Notice that from the expression of soliton profile \eqref{eq:so} or \eqref{eq:stationary}, as the functions $\varphi,m_\varphi$ are even and localized, one sees that $L_n\frac{\partial \varphi(-x)}{\partial c}=-m_\varphi$ also holds, then the function $\frac{\partial \varphi}{\partial c}$ is also localized and even. The nonzero inner product of the elements of the sets \eqref{eq:set JL} and \eqref{eq:set LJ} are the following:
\begin{align*}
&\int_{\R}\big(F^{\pm}(x,k)\big)_x(1-\partial_x^2)\overline{F^{\pm}(x,l)}\rmd x=\mp2\pi ik|a(k)|^2\delta(k-l),\quad \text{for} \ k,l\in\R, \\
& \int_{\R}\varphi_x\partial_x^{-1}\big(\frac{\partial m_{\varphi}}{\partial c}\big)\rmd x=-\int_{\R}\varphi\frac{\partial m_{\varphi}}{\partial c}\rmd x=-\frac{\rmd \Ham_1(\varphi)}{\rmd c}=-4\kappa c, \\
& \int_{\R}\frac{\partial \varphi}{\partial c}m_{\varphi}\rmd x=\frac{\rmd \Ham_1(\varphi)}{\rmd c}=4\kappa c.
\end{align*}
The corresponding closure relation is
\begin{equation}\label{eq:closure relation}
\mp \int_{\R}\frac{1}{2\pi ik|a(k)|^2}\big(F^{\pm}(x,k)\big)_x(1-\partial_y^2)\overline{F^{\pm}(y,k)}\rmd k + \frac{1}{4\kappa c}\bigg(\varphi_x\partial_{y}^{-1}\big(\frac{\partial m_{\varphi}}{\partial c}\big)+\frac{\partial \varphi}{\partial c}m_{\varphi}(y)\bigg)=\delta(x-y),
\end{equation}
which indicates that  any function $z(x)$ which vanishes as $x\rightarrow\pm\infty$  can be expanded over the above two sets \eqref{eq:set JL} and \eqref{eq:set LJ}. By comparing \eqref{eq:closure relation} to \eqref{refined completeness relations}, one can take the derivative of \eqref{refined completeness relations} with respect to $x$ and insert $z(x)=f'(x)$ to have the following decomposition:
\begin{eqnarray} \label{decomposition of z}
&& z(x)=\int_{\R} \left( F^{\pm}(x,k) \right)_x P^{\pm}(k)\rmd k+\beta \varphi_x+\gamma \frac{\partial \varphi}{\partial c},
\end{eqnarray}
where the coefficients $P^{\pm}$, $\beta$ and $\gamma$ are related to the coefficients in \eqref{refined completeness relations}. 
Similarly, one can also decompose the function $z(x)$ on the second set \eqref{eq:set LJ} by applying the operator $1-\partial_x^2$ to \eqref{refined completeness relations}.

\subsection{The recursion operators at $N$-solitons}\label{subsec spec multi}
This subsection devotes to showing the spectral information of the recursion operators $\mathcal{R}(U^{(N)})$ and the adjoint $\mathcal{R}^{\ast}(U^{(N)})$ and the linearized operators $\mathcal{J}\mathcal{L}_N$, $\mathcal{L}_N\mathcal{J}$ and $\mathcal{L}_N$ (defined in $L^2(\R)$ with domain $H^1(\R)$) around the CH $N$-solitons profile $U^{(N)}$.

As stated in Remark \ref{re3.6}, the operator identities \eqref{operator identity3} and \eqref{operator identity4}  hold also for the $N$-solitons.
One has, by similar arguments of Lemma \ref{lem3.6} and \ref{le3.81}, the following:

\begin{lemma} \label{le3.62} The recursion operator $\mathcal{R}(U^{(N)})$, defined in $L^2(\R)$ with domain $H^1(\R)$, has $N$ discrete eigenvalues $c_j$  for $j=1,2,\cdot\cdot\cdot,N$, with associated  eigenfunction $G_j(U^{(N)})$ defined as follows:
\begin{eqnarray}
&&G_j(U^{(N)}) := \left( \Ham'_{N}+(\mu_{N-1}+c_j) \Ham'_{N-1} + \cdots + \frac{\mu_0}{c_j}\Ham'_1 \right)(U^{(N)}).\label{eq:Fj}
\end{eqnarray}
 The essential spectrum of $\mathcal{R}(U^{(N)})$ is the interval $(0,2\omega]$, and the corresponding eigenfunctions do not decay and thus are not in $L^2(\R)$.
\end{lemma}

\begin{proof} Consider the Jost solutions $f^{\pm}(x,k)$ of the eigenvalue problem \eqref{eq3} with the potential $m(x)=m_{U^{(N)}}:=U^{(N)}-U^{(N)}_{xx}$ and the asymptotic expressions in \eqref{eq6} and \eqref{eq6'}. In this case there are $N$ eigenvalues $k=i\kappa_j$ ($0<\kappa_j<\frac12, j=1,2,\cdot\cdot\cdot,N$) which generates the $N$-solitons profile $U^{(N)}$. It is then found that the squared eigenfunctions $F^{\pm}(x,k)=(f^{\pm}(x,k))^2$ satisfy
\begin{equation}\label{eq:Nsquared eigen}
 \bigg(-\partial_x^2+1+2\lambda\big(m_{U^{(N)}}(x)+2\omega+\partial_x^{-1}(m_{U^{(N)}}(x)\partial_x)\big)\bigg)F^{\pm}(x,k)=0.
\end{equation}
This in turn implies that
\begin{equation*}
(1-\partial_x^2)^{-1} \left( m_{U^{(N)}}+2\omega+\partial_x^{-1}(m_{U^{(N)}}(x)\partial_x) \right)F^{\pm}(x,k)=\mathcal{K}[m_{U^{(N)}}]F^{\pm}(x,k)=-\frac{1}{2\lambda}F^{\pm}(x,k).
\end{equation*}
By \eqref{eq:Nsquared eigen} and \eqref{eq:similar}, for $ j=1,2,\cdot\cdot\cdot,N$, it follows that
\begin{eqnarray}
&&\mathcal{R}(U^{(N)})(1-\partial_x^2)F^{\pm}(x,k)=-\frac{1}{2\lambda}(1-\partial_x^2)F^{\pm}(x,k),\quad \text{for} \ k\in\R, \label{eq:Reigen10} \\ &&\mathcal{R}(U^{(N)})(1-\partial_x^2)F_j^{\pm}(x)=-\frac{1}{2\lambda_j}(1-\partial_x^2)F_j^{\pm}(x)=c_j(1-\partial_x^2)F_j^{\pm}(x).\label{eq:R eigen11}
\end{eqnarray}
Moreover, in view of $\dot{F}_j^{\pm}(x)$ in the completeness relation \eqref{completeness relations}, it follows that
\begin{equation*}
\mathcal{R}(U^{(N)})(1-\partial_x^2)\dot{F}_j^{\pm}(x)=-\frac{1}{2\lambda_j}(1-\partial_x^2)\dot{F}_j^{\pm}(x)-\frac{i\kappa_j}{\omega\lambda_j^2}(1-\partial_x^2)F_j^{\pm}(x).
\end{equation*}
Since the Euler-Lagrange equation of profile $U^{(N)}$ satisfies
\[
0=I'_N(U^{(N)})=\big(\mathcal{R}(U^{(N)})-c_j\big)G_j(U^{(N)}),\quad j=1,2,\cdot\cdot\cdot,N.
\]
Therefore one has $\mathcal{R}(U^{(N)})G_j(U^{(N)})=c_jG_j(U^{(N)})$, and it is  inferred  from \eqref{eq:R eigen11} that $F_j^{\pm}(x)\sim (1-\partial_x^2)^{-1}F_j(U^{(N)})$.
On account of  \eqref{eq:Reigen10}, the essential spectrum of $\mathcal{R}(U^{(N)})$ is given by the set $\left\{\lambda : -\frac{1}{2\lambda}=\frac{\omega}{2k^2+\frac12} \right\}$ for $k\in\R$, which is equal to the interval $(0,2\omega]$.
The associated generalized eigenfunctions $(1-\partial_x^2)F^{\pm}(x,k)$ do not decay and are not in $L^2(\R)$, which can be seen from  the asymptotics of $f^{\pm}(x,k)$ in \eqref{eq6} and \eqref{eq6'}.
From \eqref{eq:R eigen11}, the recursion operator has $N$ discrete eigenvalues $c_j$  with associated eigenfunctions $F_j(U^{(N)})$ defined in \eqref{eq:Fj}.
The proof of the lemma is completed.
\end{proof}

Similarly, one has the following for the adjoint recursion operator $\mathcal{R}^\ast(U^{(N)})$ around the $N$-solitons.
\begin{lemma} \label{le3.83} The adjoint recursion operator $\mathcal{R}^\ast(U^{(N)})$, defined in $L^2(\R)$ with domain $H^1(\R)$, has $N$ discrete eigenvalues $c_j$  associated with the eigenfunctions $\mathcal{J}G_j(U^{(N)})$. The essential spectrum is the interval $(0,2\omega]$, and the corresponding eigenfunctions do not decay and are not in $L^2(\R)$.
\end{lemma}
\begin{proof}
Consider the Jost solutions $f^{\pm}(x,k)$ of the spectral problem \eqref{eq3} with the potential $m(x)=m_{U^{(N)}}:=U^{(N)}-U^{(N)}_{xx}$ and the asymptotic expressions in \eqref{eq6} and \eqref{eq6'}. In this case there are $N$ eigenvalues $k=i\kappa_j$ ($0<\kappa_j<\frac12, j=1,2,\cdot\cdot\cdot,N$) which generates the $N$-solitons proflie. It is then found (similar to \eqref{eq:Kstarsquared eigen}) that
\begin{eqnarray}
&&\mathcal{R}^\ast(U^{(N)})\big(F^{\pm}(x,k)\big)_x=-\frac{1}{2\lambda}\big(F^{\pm}(x,k)\big)_x, k\in\R, \label{eq:Reigen44'} \\
&&\mathcal{R}^\ast(U^{(N)})\big(F_j^{\pm}(x)\big)_x\sim\mathcal{R}^\ast(U^{(N)})\mathcal{J}G_j(U^{(N)})=-\frac{1}{2\lambda_j}\mathcal{J}G_j(U^{(N)})\nonumber\\
&&\mathcal{R}^\ast(U^{(N)})U^{(N)}_{x_j}=c_jU^{(N)}_{x_j}.\label{eq:R eigen55'}\\
&&\mathcal{R}^\ast(U^{(N)})\big(\dot{F}_j^{\pm}(x)\big)_x=-\frac{1}{2\lambda_j}\big(\dot{F}_j^{\pm}(x)\big)_x
-\frac{i\kappa_j}{\omega\lambda_j^2}\big(F_j^{\pm}(x)\big)_x.\nonumber
\end{eqnarray}
Since from \eqref{eq:R eigen55'}, the adjoint recursion operator has $N$ discrete eigenvalues $c_j$  with associated eigenfunctions $\mathcal{J}F_j(U^{(N)})$. On account of  \eqref{eq:Reigen44'}, the essential spectrum of $\mathcal{R}^\ast(U^{(N)})$ is the interval $(0,2\omega]$.
The associated generalized eigenfunctions $\big(F^{\pm}(x,k)\big)_x$ do not decay and are not in $L^2(\R)$, which can be seen from  the asymptotics of $f^{\pm}(x,k)$ in \eqref{eq6} and \eqref{eq6'}. This completes the proof of Lemma \ref{le3.83}.
\end{proof}

Recall that $\mathcal{J}=-\partial_x(1-\partial_x^2)^{-1}$ and $\mathcal{L}_N=I''_N(U^{(N)})$ is the second variation operator in \eqref{eq:linearized n-soliton operator}. We now  consider the operator $\mathcal{J}\mathcal{L}_N$ whose principle part is $\mathcal{J}I''_N(0)$. Direct computation shows that the symbol of the principle part of the operator $\mathcal{J}\mathcal{L}_N$ evaluated at $2k$ is
\begin{eqnarray}\label{eq:symbol JLM}
\varrho_{N}(2k)&:=&\frac{-2ik}{1+4k^2}\sum_{n=1}^N(-1)^{n-1}\sigma_{j,N-n}\frac{\big( 2\omega\big)^{n-1}(4c_jk^2+c_j-2\omega)}{(1+4k^2)^{n-1}},
\end{eqnarray}
where $\sigma_{j,k}$ are the elementally symmetric functions of $c_1,c_2,\cdot\cdot\cdot,c_{j-1},c_{j+1},\cdot\cdot\cdot,c_N$ as follows,
$$\sigma_{j,0}=1, \ \sigma_{j,1}=\sum_{l=1,l\neq j}^Nc_l, \ \sigma_{j,2}=\sum_{l<k,k,l\neq j}c_lc_k, ...,
\ \sigma_{j,N}=\prod_{l=1,l\neq j}^Nc_l.$$
By employing the operator identity \eqref{operator identity4}, one has the following statement related to the spectrum of the operator $\mathcal{J}\mathcal L_N$.
\begin{proposition} \label{pr4.5-1}
The operators $\mathcal{J}\mathcal L_N$, defined in $L^2(\R)$ with domain $H^1(\R)$, and the adjoint recursion operator $\mathcal{R}^\ast(U^{(N)})$, share the same eigenfunctions. Moreover, the essential spectrum of  $\mathcal{J}\mathcal L_N$ is contained in $i\R$. The kernel is $N$-fold and spanned by $\mathcal{J}G_j(U^{(N)})$ for $j=1,2,\cdot\cdot\cdot,N$ with  $G_j(U^{(N)})$ defined in \eqref{eq:Fj}. The generalized kernel is spanned by the functions $\frac{\partial U^{(N)}}{\partial c_j}$.
\end{proposition}
\begin{proof}

From \eqref{operator identity4}, we know that the operators $\mathcal{J}\mathcal L_N$  and $\mathcal{R}^\ast(U^{(N)})$ share the same eigenfunctions. Similar to the argument of Lemma \ref{le3.81}, one can compute the spectrum of the operator $\mathcal{J}\mathcal L_N$ directly by employing the squared eigenfunctions. Indeed, consider the Jost solutions $f^{\pm}(x,k)$ of the spectral problem \eqref{eq3} with the potential $m(x)=m_{U^{(N)}}:=U^{(N)}-U^{(N)}_{xx}$ and the asymptotics in \eqref{eq6} and \eqref{eq6'}. In this case there are $N$ eigenvalues $k=i\kappa_j$ ($0<\kappa_j<\frac12, j=1,2,\cdot\cdot\cdot,N$) which generates the $N$-solitons profile. It is then found that the squared eigenfunctions satisfy
\begin{align*}
& \mathcal{J}\mathcal L_N\big(F^{\pm}(x,k)\big)_x=\varrho_{N}(2k)\big(F^{\pm}(x,k)\big)_x,\quad \text{for} \ k\in\R, \\
& \mathcal{J}\mathcal L_N\mathcal{J}G_j(U^{(N)})\sim \mathcal{J}\mathcal L_NU^{(N)}_{x_j}=0, \qquad \mathcal{J}\mathcal L_N\frac{\partial U^{(N)}}{\partial c_j}=-\mathcal{J}G_j(U^{(N)}).
\end{align*}
In view of \eqref{eq:symbol JLM} and the above, the essential spectrum of $\mathcal{J}\mathcal L_N$ is $\varrho_{N}(2\lambda)$ for $ \lambda\in\R$, which is contained in the imaginary axis, and this yields the desired result in Proposition \ref {pr4.5-1}.
\end{proof}

Concerning the adjoint operator of $\mathcal{J}\mathcal L_N$, namely the operator $-\mathcal L_N\mathcal{J}$, which commutes with the recursion operator $\mathcal{R}(U^{(N)})$, one has the following.

\begin{proposition} \label{le3.84}
The operators $\mathcal L_N\mathcal{J}$, defined in $L^2(\R)$ with domain $H^1(\R)$, and the recursion operator $\mathcal{R}(U^{(N)})$, share the same eigenfunctions. Moreover, the essential spectrum of  $\mathcal L_N\mathcal{J}$ is contained in $i\R$, the kernel is $N$-fold and spanned by $G_j(U^{(N)})$, and the generalized kernel is spanned by $\mathcal{J}^{-1}\frac{\partial U^{(N)}}{\partial c_j}$.
\end{proposition}
\begin{proof}
The fact that the operators $\mathcal L_N\mathcal{J}$ and the adjoint recursion operator $\mathcal{R}(U^{(N)})$ share the same eigenfunctions is due to the operator identity \eqref{operator identity3}. By Lemma \ref{le3.62}, one can compute the spectrum of the operator $\mathcal L_N\mathcal{J}$ directly by employing the squared eigenfunctions:
\begin{align*}
& \mathcal L_N(1-\partial_x^2)F^{\pm}(x,k)=\varrho_{N}(2\lambda)(1-\partial_x^2)F^{\pm}(x,k),\quad \text{for} \ k\in\R, \\
& \mathcal L_N\mathcal{J}G_j(U^{(N)})=0, \quad \mathcal L_N\mathcal{J}\mathcal{J}^{-1}\frac{\partial U^{(N)}}{\partial c_j}=-G_j(U^{(N)}).
\end{align*}
Recalling \eqref{eq:symbol JLM}, the essential spectrum of $\mathcal L_N\mathcal{J}$ is $\varrho_{N}(2\lambda)$ for $ \lambda\in\R$, which is contained in the imaginary axis. The proof of Proposition \ref{le3.84} is thus complete.
\end{proof}

On account of  Proposition \ref{pr4.5-1} and Proposition \ref{le3.84}, for $j=1,2,\cdot\cdot\cdot,N$, we now have the two function sets as follows. The first set
\begin{equation}\label{eq:set JLM}
\left\{ \big(F^{\pm}(x,k)\big)_x, \quad  k\in\R;\quad \mathcal{J}G_j(U^{(N)})\sim U^{(N)}_{x_j};\quad \frac{\partial U^{(N)}}{\partial c_j} \right\}
\end{equation}
consists of  linearly independent eigenfunctions and generalized kernel of the operator $\mathcal{J}\mathcal L_N$. They are not orthogonal under the $L^2$-inner product for $N\geq2$, however, they possess some orthogonality property with respect to the following symplectic form:
\[
\aleph\big(f,g\big) := \langle f,\mathcal{J}^{-1}g\rangle=\int_{\R} f(x)\mathcal{J}^{-1}g(x)\rmd x.
\]
Indeed, since from IST and the formula of $N$-solitons, one has
\[
\sum_{j=1}^NU^{(N)}_{x_j}=-\partial_xU^{(N)}, \qquad \mathcal{J} \Ham'_n(U^{(N)})=-\sum_{j=1}^Nc^{n-1}_jU^{(N)}_{x_j}.
\]
For $j\neq k$, we compute
\begin{align}
0 = \partial_{x_k}E(U^{(N)}) & = \langle E'(U^{(N)}),U^{(N)}_{x_k}\rangle=\langle (1-\partial_x^2)U^{(N)},U^{(N)}_{x_k}\rangle, \nonumber\\
\aleph \left( U^{(N)}_{x_j},U^{(N)}_{x_k} \right)& = \int_{\R} U^{(N)}_{x_j}(x)\mathcal{J}^{-1}U^{(N)}_{x_k}(x)\rmd x\sim\int_{\R} U^{(N)}_{x_j}(x)G_j(U^{(N)}(x))\rmd x=0,\nonumber\\
\aleph \left( U^{(N)}_{x_j},\frac{\partial U^{(N)}}{\partial c_j} \right) & = 4\kappa_jc_j, \qquad  \aleph \left( U^{(N)}_{x_j},\frac{\partial U^{(N)}}{\partial c_k} \right)=0,\nonumber \\
\aleph \left( \frac{\partial U^{(N)}}{\partial c_j},\frac{\partial U^{(N)}}{\partial c_k} \right) & = \int_{\R} \frac{\partial U^{(N)}}{\partial c_j}(x)\mathcal{J}^{-1}\frac{\partial U^{(N)}}{\partial c_k}(x)\rmd x=0.\nonumber
\end{align}

The second set
\begin{equation}\label{eq:set LJM}
\left\{ (1-\partial_x^2)F^{\pm}(x,k), \quad k\in\R;\quad G_j(U^{(N)});\quad  \mathcal{J}^{-1}\frac{\partial U^{(N)}}{\partial c_j} \right\}
\end{equation}
consists of linearly independent eigenfunctions and generalized kernel of the operator $\mathcal L_N\mathcal{J}$. Similar to the case $N=1$, for any function $z(x)$ which vanishes as $x\rightarrow\pm\infty$, it can be expanded over the above two bases \eqref{eq:set JLM} and \eqref{eq:set LJM}. For example, for any  $z(x)\in L^2(\R)$, one has the following decomposition:
\begin{eqnarray*} 
&& z(x)=\int_{\R}(F^{\pm}(x,k)\big)_xP^{\pm}(k)\rmd k+\sum_{j=1}^N\beta_j\mathcal{J}F_j(U^{(N)})+\sum_{j=1}^N\gamma_j \frac{\partial U^{(N)}}{\partial c_j}
\end{eqnarray*}
with the complex-valued coefficients $P^{\pm}(k)$, $\beta_j$ and $\gamma_j$ that can be computed explicitly. Similarly, one can also decompose the function $z(x)$ on the second set \eqref{eq:set LJM} by applying $1-\partial_x^2$ to \eqref{refined completeness relations}.

\begin{remark}
It reveals from Proposition \ref{pr4.5-1} that the spectrum of $\mathcal{J}\mathcal L_N$ lies on the imagery axis, which implies that the CH $N$-solitons are spectrally stable in $H^N(\R)$. Therefore, $k_r=k^{-}_i=k_c=0$, where $k_r$ refers to the number of real positive eigenvalues of $\mathcal{J}\mathcal L_N$, $k_c$ is the number of eigenvalues
in the open first quadrant of the complex plane and $k^-_i$ is the number of purely imaginary
eigenvalues in the upper half-plane with negative Krein signature. Therefore, one has
 \[
k_r+2k_i^-+2k_c=0.
 \]
 On the other hand, the Hamilton-Krein index theorem (see for example \cite{LZ17}) reveals that
\[
n(\mathcal L_N)-p(D)=k_r+2k_i^-+2k_c=0,
\]
which verifies the stability criterion of \cite{WL20}, and hence concludes the Lyapunov stability of CH $N$-solitons.
\end{remark}
\subsection{Well-posedness theory}\label{sec_2.3}
In this subsection, we recall some well-posedness results for the CH equation (\ref{eq1}) that will be useful for proving the main results. The Cauchy problem associated with the CH equation (\ref{eq1}) has been extensively studied. Without attempting an exhaustive review, we highlight a few key results and refer to the cited literature for additional details on the well-posedness of (\ref{eq1}).

Considering strong solutions, for an initial data $u_0\in H^s(\R)$ with $s>\frac{3}{2}$, it is known \cite{LO01} that the CH equation (\ref{eq1}) has a unique local solution in $\mathcal{C}\big([0,T); H^s(\R)\big) \cap \mathcal{C}^1 \big([0, T); H^{s-1}(\R)\big)$ for some $T>0$ with $\Ham_1$ and $\Ham_2$ conserved. Moreover,  if the initial momentum potential $m_0 + \omega > 0 $ with $ m_0 =  (1- \partial_x^2) u_0,$ then $u$ is global in time \cite{CE982}. However, if $ m_0 $ changes sign ($ \omega = 0$), singularities may appear in the solution
in finite time in the form of wave breaking (the wave profile remains bounded but its slope
becomes unbounded) \cite{C00, CE98, LO01}.

As for the weak solution theory, the well-posedness in $ H^1(\R) \cap W^{1, \infty}(\R) $ was established recently in the following result.

\begin{proposition}[\cite {lps}] \label{pr2.5-2} Let $u_0\in H^1(\R) \cap W^{1, \infty} (\R).$  Then there exists
$T > 0 $ and a unique solution to the CH equation  (\ref{eq1}) such that
\begin{eqnarray*}
&&u\in \mathcal{C}\big([0, T); H^1(\R) \cap W^{1, \infty} (\R) \big)\cap \mathcal{C}^1 \big((0,T); L^{2}(\R)\big).
\end{eqnarray*}
\end{proposition}

What is more relevant to our problem is the solution theory in the natural energy space $H^1$. The existence of global weak solutions was obtained in \cite{xz,bc1,bc2}, and the uniqueness of conservative solutions was later proved in \cite{bcz}.
The following existence and uniqueness result is derived in \cite{CM01} and \cite{M18}, which will be crucial in our analysis.

\begin{proposition}[\cite{CM01,M18}]\label{pr2.5}
Let $u_0\in Y_+$ be given. Then the Cauchy problem of the CH equation (\ref{eq1}) has a unique solution
$u\in \mathcal{C}^1\big(\R;L^2(\R)\big)\cap \mathcal{C}\big(\R; H^1(\R)\big)$ such that $m\in \mathcal{C}_{ti}(\R;\mathcal M_{+})$.
Moreover, for any sequence $\{u_{0,n}\} \subset Y_{+}$, such that $u_{0,n}\rightharpoonup\ast u_0$ in $Y$, the emanating sequence of solutions $\{u_{n}\}\subset \mathcal{C}^1\big(\R;L^2(\R)\big)\cap \mathcal{C}\big(\R; H^1(\R)\big)$ satisfies, for any $T>0$, as $n \to +\infty$,
\begin{eqnarray}
&&u_n\rightharpoonup u \quad \text{ in}\quad \mathcal{C}_{w}\big([-T, T];H^1(\R)\big),\label{weak conv}\\
&& (1-\partial_x^2) u_n \rightharpoonup^\ast m \quad \text{ in}\quad \mathcal{C}_{w}\big([-T, T];\mathcal M\big). \label{weak star}
\end{eqnarray}
\end{proposition}

\subsection{Monotonicity formula and exponential decay of $H^1$-localized solutions}\label{sec_2.4}
The aim of this subsection is to prove certain monotonicity formula related to the conservation laws $E=\Ham_1$ and $F=\Ham_2$, from which we will establish the exponential decay of $H^1$-localized solutions (see Definition \ref{de3.1}) in the spirit of \cite{EM07}.

For a non-negative even function $\phi \in C_{c}^{\infty}([-1,1])$, such that $\int_{\R} \phi =\frac12$, we set
$$
g(x):= \phi*e^{-|x|} \quad \text{and} \quad \Psi_{K}(x) :=  \frac{1}{K} \int_{-\infty}^{x} g\left(\frac{y}{K}\right)dy.
$$
It is easy to check that $\Psi_{K} \geq 0$ is increasing, $\Psi_{K} \rightarrow 1$ as $x \rightarrow +\infty$  and
$$
\left|\Psi_{K}\right|+\left|\Psi_{K}^{\prime}\right|+\left|\Psi_{K}^{\prime\prime}\right|+\left|\Psi_{K}^{\prime\prime\prime}\right| \leq C_{K}\mathrm{e}^{\frac{x}{K}} \ \ \text{for}\ \  x \leq 0, \qquad \Psi_{K}^{\prime\prime\prime} = \frac{1}{K^{2}}(\Psi_{K}^{\prime}-2\phi) \leq \frac{1}{K^{2}}\Psi_{K}^{\prime}.
$$
Choosing $ 0<\alpha<1$, such that $(1-\alpha)^2c_{1}>2\omega$, let
\begin{equation*}
E_{x_{0},t_{0}}(t) := \int_{\mathbb{R}}(u^2+u_{x}^2) \Psi_{K}\big(x-x(t)-x_{0}-\alpha(x(t_{0})-x(t))\big)\rmd x.
\end{equation*}
Note that $E_{x_{0},t_{0}}(t)$ is close to $ \|u(t)\|_{H^{1}\big(x>(1-\alpha) x(t)+X_{0}\big)}^{2}$, where $X_{0} := x_{0}+\alpha x(t_{0})$. Moreover, since $\phi$ is even, there holds the following, for $t_{0}\in \mathbb{R}$,
\begin{equation*}
E_{x_{0},t_{0}}(t_{0}) \geq \frac{1}{2} \|u(t,x+x(t_0))\|_{H^{1}(x_{0},+\infty)}^{2}.
\end{equation*}

\begin{proposition}[\cite{EM07}] \label{th3.1}
Let $u \in Y(\R)$ be an $H^{1}$-localized solution of the CH equation with $c_{1}>2 \omega$. Then there exists $C>0$ depending  on $\varepsilon \mapsto R_{\varepsilon}$, $\left\|u_{0}\right\|_{H^{1}}, \omega>0 $ and $ K>\sqrt{c_{1} /\left(c_{1}-2 \omega\right)}$, such that for all $t \in \mathbb{R}$ and $x \in \mathbb{R}$,
\begin{equation*}
|u(t, x+x(t))|+|u_x(t, x+x(t))| \leq C \mathrm{e}^{-\frac{|x|}{K}}.
\end{equation*}
\end{proposition}

Proposition \ref{th3.1} mainly follows from the almost monotonicity of the energy $E(u)$ at the right of an $H^{1}$-localized solutions.
\begin{lemma}[\cite{EM07}]\label{3.1.}
Let $u \in Y(\R)$ be an $H^{1}$-localized solution of the CH equation with $c_{1}>2\omega$. Then there exist $C>0$ and $B>0$ depending only on $\varepsilon \mapsto R_{\varepsilon}, \|u_0\|_{H^{1}},\omega \geq 0 $ and $ 0<\alpha<1$, such that for all $ t \leq t_{0}, \ x_{0}>B $ and
\begin{equation*}
K>\sqrt{\frac{(1-\alpha)^{2} c_{1}}{(1-\alpha)^{2} c_{1}-2 \omega}},
\end{equation*}
there holds
\begin{equation*}
E_{x_{0}, t_{0}}\left(t_{0}\right)-E_{x_{0}, t_{0}}(t) \leq C \mathrm{e}^{-\frac{x_{0}}{K}} .
\end{equation*}
\end{lemma}

%

For $K > 0$ we consider the function $\Psi_K$ as above. For $\alpha > 0, x_0 > 0$ and $(t_0, t) \in \mathbb{R}^2$, we define
$$
F_{x_0,t_0}(t) := \frac{1}{2}\int_{\R} \left(u^3(t,x)+uu_x^2(t,x)+2\omega u^2(t,x)\right)\Psi_K\big(x-x(t)+x_0-\alpha(x(t_0)-x(t))\big) \rmd x.
$$
The monotonicity result relative to $F_{x_0,t_0}(t)$ says that if $u(0)$ is sufficiently close  to $\varphi_{c}$ in $H^1(\mathbb{R})$, then $F_{x_0,t_0}(t)$ is almost decreasing with respect to time.
\begin{proposition}\label{pro4.3}
  Let $c > 2\omega$. There exist $a_1 > 0,B > 0,C > 0,\alpha > 0$ and $K > 0$ such that if $|u_0-\varphi_{c}|_{H^1} < a_{1}$, then for all $x_0>2B$,  $t_0 \in \mathbb{R}$ and  $t \geq t_0$,
\begin{equation*}
F_{x_0,t_0}(t) \le  F_{x_0,t_0}(t_0) + Ce^{-\frac{x_0}{K}}.
\end{equation*}
\end{proposition}
\begin{proof}
The proof follows along the same lines as the proof of the $H^1$ monotonicity. We set $y(t) := x(t)-x_0+\alpha\left(x\left(t_0\right)-x(t)\right)$. Recall that $t \geqslant t_0$ and $x_0>2 B>0$. Differentiating $F_{x_0, t_0}(t)$ with respect to $t$, we obtain
\begin{equation*}
\begin{aligned}
F_{x_0, t_0}^{\prime}(t) & = -\frac{1-\alpha}{2} x'(t) \int_{\R}\left(u^3+uu_x^2+2\omega u^2\right)\Psi_K^{\prime}(x-y(t)) \rmd x \\
& \quad +\frac{1}{2}\int_{\R}\left(4\omega u+3u^2-u_x^2-2uu_{xx}\right)u_t \Psi_k(x-y(t)) \rmd x-\int_{\R}uu_xu_t\Psi'_k(x-y(t)) \rmd x \\
& = F_1+F_2+F_3 .
\end{aligned}
\end{equation*}
 We  choose $a_1$ sufficiently small such that $2\omega<c_2=\frac{2\omega+c_1}{2}<x'(t)<c_3=\frac{3 c_1-2\omega}{2}$. We then take $\alpha>0$ such that $(1-\alpha) x'(t)>c_4=\frac{2\omega+c_1}{2}$, note that  $2\omega<c_4=c_2<c_1<c_3$. Since from \eqref{Hamilton form},
 $u_{t}=-\partial_x(1-\partial_x^2)^{-1}F'(u)$. If we denote by $$(1-\partial_x^2)h:=4\omega u+3u^2-u_x^2-2uu_{xx},$$ then $h=u^2+(1-\partial_x^2)^{-1}(4\omega u+2u^2+u_x^2)$ and $u_{t}=-\frac{1}{2}h'$.
Let us compute
\begin{equation*}
  \begin{aligned}
    F_2 &=\frac{1}{2}\int_{\R}\left(4\omega u+3u^2-u_x^2-2uu_{xx}\right)u_t \Psi_k(x-y(t)) \rmd x\\
    &=-\frac{1}{4} \int_{\R} h'(h-h_{xx})\Psi_k(x-y(t)) \rmd x=\frac{1}{8} \int_{\R} (h^2_x-h^2)\Psi'_k(x-y(t)) \rmd x\\
    &\leq -\frac{1}{8} \int_{\R} h^2\Psi'_k(x-y(t)) \rmd x\leq C\|u\|^2_{H^1}\int_{\R}(u^2+u_x^2)\Psi'_k(x-y(t)) \rmd x.
    \end{aligned}
\end{equation*}
\begin{equation*}
  \begin{aligned}
    F_3 &=-\int_{\R}uu_xu_t\Psi'_k(x-y(t))dx=\frac{1}{2}\int_{\R}uu_xh'\Psi'_k(x-y(t)) \rmd x\\
    &=\frac{1}{2}\int_{\R}uu_x\big(2uu_x+(1-\partial_x^2)^{-1}(4\omega u+2u^2+u_x^2)_x\big)\Psi'_k(x-y(t)) \rmd x\\
    &\leq C\|u\|^2_{H^1}\int_{\R}(u^2+u_x^2)\Psi'_k(x-y(t)) \rmd x.
    \end{aligned}
\end{equation*}
Therefore, from Lemma \ref{3.1.} and $\Psi'_k(x-y(t))\leq e^{\frac{B-x_0+\alpha c_2(t_0-t)}{K}}$ for $|x-x(t)|<B$,
\begin{equation*}
\begin{aligned}
F_{x_0, t_0}'(t)\leq C\|u\|^2_{H^1}\int_{\R}(u^2+u_x^2)\Psi'_k(x-y(t)) \rmd x\leq Ce^{\frac{B-x_0+\alpha c_2(t_0-t)}{K}}.
\end{aligned}
\end{equation*}
The proof of the proposition is completed by integrating the above from $t_0$ to $t$.
\end{proof}

We now introduce two quantities, $E_{R}$ and $E_{L}$, related to the conservation law $E(u)$, where the subscript $R$ and $L$ denote the right and left, respectively.
$$
E_{R}(t) := \frac{1}{2}\int_{\R} (u^2(t,x)+u_{x}^{2}(t,x))\Psi_K (x-x(t)+x_{0}) \rmd x,
$$
and
$$
E_{L}(t) := \frac{1}{2}\int_{\R} (u^2(t,x)+u_{x}^{2}(t,x))(1-\Psi_K(x-x(t)+x_{0})) \rmd x.
$$
These quantities are close to $\frac{1}{2}\|u(t)\|^{2}_{H^1(x>x(t)+x_{0})}$ and $\frac{1}{2}\|u(t)\|^{2}_{H^1(x<x(t)-x_{0})}$, respectively. The following corollary states that $E_{R}$ is almost
non-increasing and $E_{L}$ is almost non-decreasing.

\begin{corollary}\label{cor 1}
    Let  $a_{1}$, $B>0$, $C>0$, and $K>0$  be as in Lemma \ref{3.1.}. If  $\|u_{0}-\varphi_{c}\|_{H^{1}}<a_{1}$, then the following inequalities hold for all real  $t$, for all  $t^{\prime} \le  t$, and for all  $x_{0}>B$,
\[
E_{R}(t) \le  E_{R}\left(t^{\prime}\right)+C e^{-\frac{x_{0}}{K}},\quad \text {and } \quad E_{L}(t) \geqslant E_{L}\left(t^{\prime}\right)-C e^{-\frac{x_{0}}{K}} .
\]
\end{corollary}
\begin{proof}
  Let $t\in \mathbb{R}$, and $t^{\prime} \le  t$. One has $x-x(t^{\prime})-x_{0}-\alpha(x(t)-x(t^{\prime})) \le  x-x(t^{\prime})-x_{0}.$ Since $\Psi_K$ is non decreasing, we get $I_{x_{0},t}(t^{\prime}) \le  E_{R}(t^{\prime})$ . On the other hand, $E_{R}(t)= I_{x_{0},t}(t)$ and Lemma \ref{3.1.} implies that $E_{R}(t) \le  E_{R}(t^{\prime})+Ce^{-\frac{x_{0}}{K}}$ for all real $t^{\prime} \le  t$ and for all $x_{0} > B$.

To prove that $E_{L}$ is almost non decreasing, we use the fact that the CH equation is invariant under the transformation $(t, x) \longmapsto (-t,-x)$. For a given solution $u$, define $\tilde{u}$ by $\tilde{u}(t, x) = u(-t, -x)$ which is also a solution of the CH equation. Since $\|\tilde{u}_{0}-\varphi_{c}\|_{H^1}=\|{u}_{0}-\varphi_{c}\|_{H^1}$, we may decompose $\tilde{u}_{0}$ as follows:
  $$
  \tilde{u}(t,x)=\varphi_{\tilde{c}(t)}(x-\tilde{x}(t))+\tilde{v}(t,x-\tilde{x}(t)).
  $$
Using the uniqueness property of the decomposition we find that $\tilde{x}(t)=-x(-t),\tilde{c}(t)=c(-t)$, and $\tilde{v}(t,x)=v(-t,-x)$, Now let $t, t^{\prime} \in \mathbb{R}$ such that $t^{\prime} \le  t$, we use the previous result concerning $E_R$ with $\tilde{u}_{0}$ between $-t$ and $-t^{\prime}$ to obtain
\begin{equation*}
  \begin{aligned}
\frac{1}{2}  &\int_\R \Psi_K(x-(-x(t^{\prime})+x_{0}))(u^2(t^{\prime},-x)+u^{2}_{x}(t^{\prime},-x)) \rmd x \\
 &\le  \frac{1}{2} \int_\R \Psi_K (x-(-x(t)+x_{0}))(u^2(t,-x)+u^{2}_{x}(t,-x)) \rmd x+Ce^{-\frac{x_{0}}{K}}.
  \end{aligned}
\end{equation*}
Then the fact that $\Psi_K(x)=1-\Psi_K(-x)$ implies
\begin{equation*}
  \begin{aligned}
\frac{1}{2}  &\int_\R (1-\Psi_K(-x-(x(t^{\prime})-x_{0})))(u^2(t^{\prime},-x)+u^{2}_{x}(t^{\prime},-x)) \rmd x \\
& \le  \frac{1}{2} \int_\R (1-\Psi_K(-x-(x(t)-x_{0})))(u^2(t,-x)+u^{2}_{x}(t,-x)) \rmd x+Ce^{-\frac{x_{0}}{K}}.
  \end{aligned}
\end{equation*}
Changing $x$ into $-x$ yields that
$$
E_{L}(t) \geqslant E_{L}(t^{\prime}) - Ce^{-\frac{x_{0}}{K}}.
$$
This concludes the proof of the corollary.
\end{proof}
For $x_0 > 2B$, we define
\[
F_{R}(t) := \frac{1}{2}\int_{\R} \left(u^3(t,x)+uu_x^2(t,x)+2\omega u^2(t,x)\right)\Psi_K\big(x-x(t)+x_0\big) \rmd x.
\]
We can easily deduce from Proposition \ref{pro4.3} that the following corollary holds true.
\begin{corollary}\label{cor FR}
  Let $a_1,B,C$ and $K$ be as in Propositions \ref{pro4.3}. if $\|u_0-\varphi_{c}\|_{H^1} < a_1$, then  for all  $t$, for all $t^{\prime} \le  t$ and for all  $x_0$ with $x_0 > 2B$, there holds
  $$
F_{R}(t) \le  F_{R}(t^{\prime})+Ce^{-\frac{x_0}{K}}.
  $$
\end{corollary}

\subsection{Modulation theory}\label{sec_2.5}
The aim of this subsection is to prove that if $u$ is a solution of the CH equation which remains close to the manifold of the sum of $N$ solitons for $t\in[0,t_0]$, then for the same time interval, we can decompose $u$ as the sum of $N$ modulated solitons plus a function $v(t)$, which remains small in $H^1(\R)$, as follows,
\begin{equation}\label{decompps}
u(t,x)=\sum_{j=1}^{N}\varphi_{c_{j}(t)}(x-x_{j}(t))+v(t),
\end{equation}
for $\mathcal {C}^{1}$ functions $x_j : [0, t_0] \to \R$, with $j=1,2,\cdot\cdot\cdot,N$ and the following $2N$ orthogonality relations:
$$
\left \langle v(t,\cdot),(1-\partial_{x}^{2})\partial_{x}\varphi_{c_{j}(t)}(\cdot -x_{j}(t))\right \rangle = \left\langle v(t,\cdot),(1-\partial_{x}^{2})\varphi_{c_{j}(t)}(\cdot -x_{j}(t)) \right\rangle =0.
$$
We fix $N$ velocities
\[2\omega<c_1^0<c_2^0<\cdot\cdot\cdot<c_N^0,\]
and
\begin{equation}\label{minimal velo}
\sigma_0 := \frac14\min \left\{ 2\omega,\ 2\sqrt{1-\frac{2\omega}{c_1^0}}, \ c_1^0-2\omega, \ c_2^0-c_1^0,\cdots, \ c_N^0-c_{N-1}^0 \right\}>0.
\end{equation}
We denote by $\mathcal U(\alpha,L)$ the neighborhood of size $\alpha$ of all the sum of $N$ solitons of velocities $c_j^0$ such that the distance of their space shifts $x_j$ is large than $L$, i.e.
\begin{equation}\label{4.3}
\mathcal U(\alpha,L) := \left\{ u \in H^{1}(\mathbb{R}) : \inf\limits_{x^0_{j}-x^0_{j-1}>L} \left\| u-\sum_{j=1}^{N}\varphi_{c^0_{j}}(\cdot-x^0_{j}) \right\|_{H^1}<\alpha \right\}.
\end{equation}
\begin{lemma}\label{le4.1}
There exist $L_1>0$, $\alpha_{1}>0$ and $K_{1}>0$, such that for some $L>L_{1}$, $\alpha \in (0,\alpha_{0})$ and $t_0>0$, if the solution $u(t) \in \mathcal U(\alpha,L)$ on $[0,t_{0}]$ with $u \in Y([0,t_0])$, then there exist $2N$ unique $\mathcal{C}^1$ functions
 $$
 c_{j}(t):[0,t_{0}] \to (2\omega,+\infty),\qquad x_{j}(t):[0,t_{0}] \to \R,\ j=1, \cdots,N,
 $$
 such that if we define $v(t)$ by
\[
v(t) := u(t)-\sum_{j=1}^NR_j(t),\quad \text{where} \quad R_j(t) := \varphi_{c_j(t)}(\cdot-x_j(t)),
\]
then the following properties hold true for $j=1,2,\cdot\cdot\cdot,N$, and $t\in[0,t_0]$,
\begin{eqnarray}
&&\left \langle v(t),(1-\partial_{x}^{2})\partial_xR_j(t)\right \rangle = \left\langle v(t),(1-\partial_{x}^{2})R_j(t) \right\rangle =0,\label{4.6}\\
&&\|v(t)\|_{H^1}+|c_j(t)-c_j^0|\leq K_1\alpha,\label{4.5}\\
&&|\dot{c}_j(t)|+|\dot{x}_j(t)-c_j(t)|\leq K_1 \left( \int_\R e^{-\sigma_0|x-x_j(t)|}(v^2+v_x^2)(t)\rmd x \right)^{\frac12}+K_1 e^{-\frac{\sigma_0(L+\sigma_0t)}{4}}.\label{4.7}
\end{eqnarray}
\end{lemma}

\begin{proof}
We first prove the stationary case of $u(\cdot) \in \mathcal U(\alpha,L)$. Let $L>0$, $X^0 := (x_j^0)\in\R^N$ such that $x_{j+1}^0-x_{j}^0>L$, and set $R_{X^0} := \sum\limits_{j=1}^N\varphi_{c_j^0}(\cdot-x_j^0)$. We denote by $B_{H^{1}}(R_{X^0},\alpha)$ the ball in $H^1(\R)$ of center $R_{X^0}$ and radius $\alpha$. Now we define the following mapping:
$$
\begin{aligned}
\mathcal Y:\prod_{j=1}^N(c_j^0-\alpha,c_j^0+ \alpha)\times\prod_{j=1}^N(-\alpha,\alpha) \times B_{H^{1}}\left(R_{Z}, \alpha\right) & \rightarrow \mathbb{R}^{N}\times\mathbb{R}^{N},
\end{aligned}
$$
by $\mathcal Y=(\mathcal Y^{1,1},\cdot\cdot\cdot,\mathcal Y^{1,N},\mathcal Y^{2,1},\cdot\cdot\cdot,\mathcal Y^{2,N})$, where
$$
\mathcal Y^{1,j}\left(c_{1}, \ldots, c_{N},y_{1}, \ldots, y_{N}, u\right) := \int_{\mathbb{R}}\left(u(x)-\sum_{k=1}^{N} \varphi_{c_k^0}\left(\cdot-x^0_{k}-y_{k}\right)\right)\left(1-\partial_{x}^{2}\right) \varphi_{c_j^0}\left(\cdot-x^0_{j}-y_{j}\right) \rmd x.
$$
and
$$
\mathcal Y^{2,j}\left(c_{1}, \ldots, c_{N},y_{1}, \ldots, y_{N}, u\right) := \int_{\mathbb{R}}\left(u(x)-\sum_{k=1}^{N} \varphi_{c_k^0}\left(\cdot-x^0_{k}-y_{k}\right)\right)\left(1-\partial_{x}^{2}\right)\partial_x \varphi_{c_j^0}\left(\cdot-x^0_{j}-y_{j}\right) \rmd x.
$$
 The dominated convergence theorem and the smoothness of $\varphi_c$ reveal that $\mathcal Y$ is $\mathcal C^1$. In order to apply the implicit function theorem, we need to compute the partial derivatives of $\mathcal Y$ at the point $M_0=(c^0_{1}, \ldots, c^0_{N},0, \ldots, 0, R_{X^0})$ for all $j=1, \ldots, N$,
$$
\frac{\partial \mathcal Y^{1,j}}{\partial c_{j}}(M_0)=-\int_{\mathbb{R}}\partial_{c^0_{j}}\varphi_{c_j^0}(\cdot-x^0_{j})\left(1-\partial_{x}^{2}\right)\varphi_{c_j^0}(\cdot-x^0_{j}) \mathrm{d} x=-\frac{\rmd \Ham_{1}(\varphi_{c_j^0})}{\rmd c_j^0}=-4\kappa_0c_j^0<0.
$$
We also integrate by parts to compute
\begin{eqnarray*}
&&\frac{\partial \mathcal Y^{1,j}}{\partial y_{j}}(M_0)=\int_{\mathbb{R}} \partial_{x} \varphi_{c_j^0}(\cdot-x^0_{j})\left(1-\partial_{x}^{2}\right) \varphi_{c_j^0}(\cdot-x^0_{j}) \mathrm{d} x=0,\\
&&\frac{\partial \mathcal Y^{2,j}}{\partial c_{j}}(M_0)=-\int_{\mathbb{R}} \left(1-\partial_{x}^{2}\right)\partial_{x} \varphi_{c_j^0}(\cdot-x^0_{j}) \partial_{c^0_{j}}\varphi_{c_j^0}(\cdot-x^0_{j}) \mathrm{d} x=0,\\
&&\frac{\partial \mathcal Y^{2,j}}{\partial y_{j}}(M_0)=\int_{\mathbb{R}} \left(1-\partial_{x}^{2}\right)\partial_{x} \varphi_{c_j^0}(\cdot-x^0_{j}) \partial_{x}\varphi_{c_j^0}(\cdot-x^0_{j}) \mathrm{d} x=\|\partial_{x}\varphi_{c_j^0}\|^2_{H^1}\geq C_1(c_j^0)>0.
\end{eqnarray*}
Hence, using the exponential decay of $\varphi^{(n)}_{c}$, $\partial_c\varphi_{c}$ and the assumption that $x_{j+1}^0-x_{j}^0>L$, we infer that for $L_{1}$ large enough (recall that $L>L_{1}$ ) and $j\neq k$, there holds
\begin{eqnarray*}
&&\left|\frac{\partial \mathcal Y^{1,j}}{\partial c_{k}}(M_0) \right| = \left| \int_{\mathbb{R}} \left(1-\partial_{x}^{2}\right)\varphi_{c_j^0}(\cdot-x^0_{j}) \partial_{c^0_{k}}\varphi_{c_k^0}(\cdot-x^0_{k}) \mathrm{d} x \right|\nonumber\\
&&\leq C\int_{\mathbb{R}}e^{-\sigma_0(|x-x_j^0|+|x-x_k^0|)}\rmd x\leq Ce^{-\frac{\sigma_0}{2}|x_j^0-x_k^0|}\leq Ce^{-\frac{\sigma_0L}{2}}.
\end{eqnarray*}
In a same manner, for $j\neq k$, one has
\begin{eqnarray*}
&&\left| \frac{\partial \mathcal Y^{1,j}}{\partial y_{k}}(M_0) \right| = \left| \int_{\mathbb{R}} \left(1-\partial_{x}^{2}\right)\varphi_{c_j^0}(\cdot-x^0_{j}) \partial_{x}\varphi_{c_k^0}(\cdot-x^0_{k}) \mathrm{d} x \right| \leq Ce^{-\frac{\sigma_0L}{2}},\\
&&\left| \frac{\partial \mathcal Y^{2,j}}{\partial y_{k}}(M_0) \right| = \left| \int_{\mathbb{R}} \left(1-\partial_{x}^{2}\right)\partial_{x}\varphi_{c_j^0}(\cdot-x^0_{j}) \partial_{c^0_{k}}\varphi_{c_k^0}(\cdot-x^0_{k}) \mathrm{d} x \right| \leq Ce^{-\frac{\sigma_0L}{2}},\\
&&\left| \frac{\partial \mathcal Y^{2,j}}{\partial y_{k}}(M_0) \right| = \left| \int_{\mathbb{R}} \left(1-\partial_{x}^{2}\right)\partial_{x} \varphi_{c_j^0}(\cdot-x^0_{j}) \partial_{x}\varphi_{c_k^0}(\cdot-x^0_{k}) \mathrm{d} x \right| \leq Ce^{-\frac{\sigma_0L}{2}}.
\end{eqnarray*}
Therefore, we deduce that $D_{\left(c_1, \ldots, c_{N},y_1, \ldots, y_{N}\right)} Y(M_0)=D+P$, where $D$ is an invertible diagonal matrix with $\left\|D^{-1}\right\| \leq\left(C_{2}\right)^{-1}$ and $\|P\| \leq O\left(\mathrm{e}^{-\frac{\sigma_0L}{2}}\right)$. Note that $C_{2}$ depends only on $c^0_{j}$ and $L$. Hence, there exists $L_{1}>0$, such that for $L>L_{1}$, the matrix $D_{\left(c_1, \ldots, c_{N},y_1, \ldots, y_{N}\right)} Y(M_0)$ is invertible with an inverse matrix of norm smaller than $2\left(C_{2}\right)^{-1}$. From the implicit function theorem, we deduce that there exist $\alpha_{0}>0$ and $C^{1}$ functions $\left(c_{j},  y_{j}\right)^N_{j=1}$ from $B_{H^1}\left(R_{X^0}, \alpha_{0}\right)$ to a neighbourhood of $(c_1^0,\ldots, c_N^0,0, \ldots, 0)$, which are uniquely determined, such that
$$
\mathcal Y\left(c_{1}(u), \ldots, c_{N}(u),y_{1}(u), \ldots, y_{N}(u), u\right) = 0 \qquad  \forall \ u \in B_{H^1}\left(R_{X^0}, \alpha\right),\ 0<\alpha\leq\alpha_0.
$$
Moreover, there exits $K_{1}>0$, such that if $u \in B_{H^1}\left(R_{X^0}, \alpha\right)$, with $0<\alpha \leq \alpha_{0}$, then
\begin{equation*}
\sum_{j=1}^{N}\left|c_{j}(u)-c_j^0\right|+\sum_{j=1}^{N}\left|y_{j}(u)\right| \leq K_{1} \alpha .
\end{equation*}
Note that $\alpha_{0}$ and $C_{0}$ depend only on $c_{1}$ and $L_{1}$ and not on the point $X^0$ provided that $x_{j+1}^0-x_{j}^0>L\geq L_1$. For $u \in B_{H^1}\left(R_{X^0}, \alpha_{0}\right)$, we set $x_{j}(u) := x^0_{j}+y_{j}(u)$, then $x_j$ is a $\mathcal C^1$ function on $B_{H^1}\left(R_{X^0}, \alpha_{0}\right)$, such that
\[
x_{j}(u)\geq x_{j-1}(u)+L-2K_1\alpha_{0}.
\]
We are now able to define the modulation of $u\in\mathcal U(\alpha,L)$ for $L\geq L_1$ and $0<\alpha\leq\alpha_1$, with $\alpha_1$ to be chosen later. Indeed, for $\alpha\leq\alpha_1$, one can cover $\mathcal U(\alpha,L)$ as follows:
\[
U(\alpha,L)\subseteq \bigcup_{X\in \R^N, \ x_{j+1}-x_{j}>L}B_{H^1}\left(R_{X}, \rho_0\right),
\]
where $\alpha_1\leq \rho_0\leq \alpha_0$, and $\rho_0$ is chosen such that if $u\in B_{H^1}\left(R_{X}, \rho_0\right)\bigcap B_{H^1}\left(R_{\tilde{X}}, \rho_0\right)$, then the modulation of $u$ is uniquely defined due to the uniqueness in the implicit function theorem.

Now we define the modulation of the CH solution $u$, such that $u(t)\in\mathcal U(\alpha,L)$ for all $t\in[0,t_0]$, by setting for $j=1,\ldots,N$ and $t\in[0,t_0]$,
\begin{eqnarray*}
&& c_j(t)=c_j(u(t)), \ x_j(t)=x_j(u(t)),\\
&& v(t)=u(t)-\sum_{j=1}^N\varphi_{c_j(t)}(\cdot-x_j(t))=u(t)-\sum_{j=1}^NR_j(t).
\end{eqnarray*}
These functions satisfy \eqref{4.6} and \eqref{4.5}. To verify \eqref{4.7}, we substitute $u(t)$ into the CH equation and employ the equation of $\varphi_{c_j(t)}$, it is found that $v(t)$ satisfies for all $t\in[0,t_0]$,
\begin{equation}\label{4.15}
\begin{split}
&(1-\partial_{x}^{2}) v_{t}-x'_j(t)(1-\partial_{x}^{2}) v_{x}-\sum_{j=1}^{N}\left(\dot{x}_{j}-c_{j}\right)\left(1-\partial_{x}^{2}\right) \partial_{x} R_{j}+2 \omega v_{x}+\sum_{j=1}^{N}\dot{c}_j \left(1-\partial_{x}^{2}\right)\partial_{c_j} R_{j} \\
=&-\frac{1}{2}\left(1-\partial_{x}^{2}\right) \partial_{x}\left[\left(v+\sum_{j=1}^{N} R_{j}\right)^{2}-\sum_{j=1}^{N} R_{j}^{2}\right]  \\
& - \partial_{x}\left[\left(v+\sum_{j=1}^{N} R_{j}\right)^{2}-\sum_{j=1}^{N} R_{j}^{2}+\frac{1}{2}\left(v_{x}+\sum_{j=1}^{N} \partial_{x} R_{j}\right)^{2}-\frac{1}{2} \sum_{j=1}^{N}\left(\partial_{x} R_{j}\right)^{2}\right].
\end{split}
\end{equation}
Thanks to \eqref{4.5}, one can choose $\alpha_1$ sufficiently small such that $\sigma_0$ can also be chosen smaller with respect to $c_j(t)$, then there holds the following
\[
|R_j(t,x)|\leq Ce^{-\sigma_0|x-x_j(t)|}.
\]
Taking the $L^{2}$-scalar product of \eqref{4.15} with $R_{j}$ and $\partial_{x} R_{j}$, integrating by parts, using the decay of $R_{\mathrm{j}}$ and its derivatives, we find
\begin{equation*}
|\dot{c}_j(t)|+\left|\dot{x}_{j}(t)-c_{j}(t)\right|\leq C \left( \int_\R e^{-\sigma_0|x-x_j(t)|}(v^2+v_x^2)(t)\rmd x \right)^{\frac12}+C\sum_{j\neq k}e^{-\frac{\sigma_0}{2}|x_j^0-x_k^0|}.
\end{equation*}
Taking $\alpha_{1}$ small enough and $L_{1}$ large enough, we get $|x_{j}(t)-x_{k}(t)|\geq \frac{L}{2}+\sigma_0t$, it follows from the above that  \eqref{4.7} holds true. The proof of  Lemma \ref{le4.1} is complete.
\end{proof}

In fact, by employing Lemma \ref{le4.1}, one can improve the estimates of parameters as follows.
\begin{corollary}\label{para improve}
For any $\varepsilon>0$, there exists $\delta=\delta(\varepsilon)>0$ such that if $\|u_0-\varphi_c\|_{H^1}<\delta$, then the modulation parameters $(x(t),c(t)$ given in Lemma \ref{le4.1} satisfy the following: for some constant $C>0$ and for all $t\in \R$,
\[
|\dot{c}(t)|\leq C\|v\|_{H^1}^2,\quad  \left|\dot{x}(t)-c(t)\right|\leq C\|v\|_{H^1}.
\]
\end{corollary}
One can also obtain the almost monotonicity of functionals that are very close to the energy at the right of the $j$-th bump of $u$, for $j=1,\dots,N-1$.
Setting $\Psi_{K} := \Psi(\frac{\cdot}{K})$, we introduce, for $j \in \{2,\cdots,N\}$,
$$
I_{j}(t)=I_{j}(t,u(t)) := \int_{\mathbb{R}}(u^{2}(t)+u_{x}^{2}(t))\Psi_{j,K}(x)\mathrm{d}x,
$$
where $\Psi_{j,K}(x) := \Psi_{K}(x-y_{j}(t))$, with $y_{j}(t) := \frac{x_{j-1}(t)+x_{j}(t))}{2}$ for $j=2,\cdots,N$. Note that $I_{j}(t)$ is close to $\|u(t)\|_{H^{1}(x>y_{j}(t))}$ and thus measures the energy at the right of the $(j-1)$-th bump of $u$.

\begin{lemma}[\cite{EM07}]\label{le4.2}
Let $u \in Y([0,T])$ be a solution of the CH equation, such that $u(t) \in \mathcal
U(\alpha,L/2)$ on $[0,t_{0}]$. There exist $\alpha>0$ and $L_{0}>0$ depending only on $\sigma_{0}$, such that if $0<\alpha<\alpha_{0}$ and $L \geq L_{0}$, then
\begin{equation*}
I_{j}(t)-I_{j}(0) \leq O(e^{-\sigma_0 L}), \qquad \forall j \in {2,\cdots,N}, \ \forall t \in [0,t_{0}].
\end{equation*}
\end{lemma}

\medskip

Recall that the Hessian of $cE-F$ around a solitary wave $\varphi_{c}$ is given by
\begin{equation}\label{L1oper}
L_1=-\partial_{x}((c-\varphi_{c})\partial_{x})-3\varphi_{c}+\varphi_{c}^{\prime\prime}+(c-2\omega).
\end{equation}
By a Liouville transform \cite{CS02}, the linearized operator $L_1$, defined on $H^2(\R)$, is transformed into a regular self-adjoint Sturm-Liouville operator, which is a relatively compact perturbation of a second order differential operator with constant coefficients. Thus the spectral information of $L_1$ follows directly from the Sturm-Liouville theory. In particular,  since  $\varphi_c<c-2\omega$, we can introduce the transformation
$$
\varphi_c(x)=\psi(z(x)),\qquad z(x) := \int_{0}^{x}\frac{ds}{\sqrt{c-\varphi_c(s)}}.
$$
Then the operator $L_1$ is transformed into the following linear operator
\begin{equation*}
\begin{split}
\mathcal{L} & = -\partial_z^2+(c-2\omega)+V(z) \\
&:=-\partial_z^2+(c-2\omega)-3\psi(z)+\frac{3\psi''(z)}{4\big(c-\psi(z)\big)}+\frac{5\big(\psi'(z)\big)^2}{16\big(c-\psi(z)\big)^2}, \\
\psi(z) &  = (c-2\omega)\operatorname{sech}^2 \left( \frac{z\sqrt{c-2\omega}}{2} \right).
\end{split}
\end{equation*}
For example,  if $c=6$ and $\omega=1$, the above potential reads
\[
V(z)=\frac{90\operatorname{sech}^2(z)-116\operatorname{sech}^4(z)+44\operatorname{sech}^6(z)}{\big(3-2\operatorname{sech}^2(z)\big)^2}.
\]

\begin{remark}\label{rk spec of L}
Virial identities are a powerful tool in showing the asymptotic stability of solitons for gKdV equations \cite{Ma06}. Although we know that the potential $V$ is a reflectionless potential, like for the gKdV case, we still do not know the explicit information about the negative eigenvalue and the corresponding eigenfunction of the linear operators $\mathcal{L}$, and hence $L_1$. This is the major obstacle for us to prove rigidity property (Theorem \ref{thm1.4}) through a virial identity argument.
\end{remark}
We now give a generalization of a coercivity of quadratic form, which will be useful in showing the orbital stability of trains of solitons.
\begin{lemma}[\cite{EM07}]\label{le4.3}
There exist $\delta > 0, \ C_{\delta} > 0$ and $C > 0$ depending only on $c_{1} > 2\omega$, such that for all $c \geq c_{1}$, $0<\varTheta \in \mathcal C^{2}(\mathbb{R})$ and $v \in H^{1}(\mathbb{R})$, satisfying
\begin{equation*}
\left| \langle \sqrt{\varTheta}v,(1-\partial_{x}^{2})\varphi_{c}\rangle \right| + \left| \langle \sqrt{\varTheta}v,(1-\partial_{x}^{2})\partial_{x}\varphi_{c}\rangle \right| \leq \delta\| v\|_{H^{1}},
\end{equation*}
and
\begin{equation*}
\left| \frac{(\varTheta^{\prime})^2}{4\varTheta} \right| + c \left| \varTheta^{\prime} \right| + \left|\frac{\varTheta^{\prime\prime}}{2} \right| \leq \min \left( \frac{1}{4},\frac{C_{\delta}}{8c} \right) |\varTheta \vert,
\end{equation*}
it holds that
\[
\int_{\mathbb{R}}\varTheta \left[ (c-\varphi_{c})v_{x}^2+(-3\varphi_{c}+\varphi_{c}^{\prime\prime}+(c-2\omega))v^{2} \right] \rmd x +\int_{\mathbb{R}}\varTheta^{\prime}\varphi_{c}^{\prime}v^2 \rmd x \geq C\int_{\mathbb{R}}\varTheta (v_{x}^{2}+v^{2}) \rmd x.
\]
\end{lemma}
\begin{proof}
We note that
\begin{eqnarray*}
\left\langle L_1 \sqrt{\Theta} v, \sqrt{\Theta} v\right\rangle&=& \int_{\mathbb{R}} \Theta\left[\left( c- \varphi_{c}\right) v_{x}^{2}+\left(-3 \varphi_{c}+ \varphi_{c}^{''}+(c-2 \omega)\right) v^{2}\right] \mathrm{d} x \nonumber\\
&& + \int_{\mathbb{R}} \frac{ c-\varphi_{c}}{4} \frac{\left(\Theta^{'}\right)^{2}}{\Theta} v^{2}+\int_{\mathbb{R}}\left( c- \varphi_{c}\right) \Theta^{\prime} v v_{x} \mathrm{~d} x \nonumber\\
&=& \int_{\mathbb{R}} \Theta\left[\left( c- \varphi_{c}\right) v_{x}^{2}+\left(-3 \varphi_{c}+\varphi_{c}^{\prime \prime}+(c-2 \omega)\right) v^{2}\right]+ \Theta^{\prime} \varphi_{c}^{\prime} v^{2} \mathrm{~d} x \nonumber\\
&& + \int_{\mathbb{R}}\left[\left( c- \varphi_{c}\right)\left(\frac{\left(\Theta^{\prime}\right)^{2}}{4 \Theta}-\frac{1}{2} \Theta^{\prime \prime}\right)-\Theta^{\prime} \varphi_{c}^{\prime}\right] v^{2} \mathrm{~d} x .
\end{eqnarray*}
On the other hand, on account of the results on the spectrum of $L_1$ derived in \cite{CS02}, it can be easily seen that there exist $\delta>0$ and $C_{\delta}>0$, such that if for $c \geq c_{1}>2\omega$,
\begin{equation*}
\left|\left\langle w,\left(1-\partial_{x}^{2}\right) \varphi_{c}\right\rangle\right|+\left|\left\langle w,\left(1-\partial_{x}^{2}\right) \partial_{x} \varphi_{c}\right\rangle\right| \leq \delta\|w\|_{H^{1}},
\end{equation*}
then
\begin{equation*}
\left\langle L_1 w, w\right\rangle \geq C_{\delta}\|w\|_{H^{1}}^{2}.
\end{equation*}
Direct calculations imply that
\begin{equation*}
\left\| \sqrt{\Theta} v \right\|_{H^{1}}^{2}=\int_{\mathbb{R}} \Theta\left(v^{2}+v_{x}^{2}\right) \mathrm{d} x+\frac{1}{4} \int_{\mathbb{R}} \frac{\left(\Theta^{\prime}\right)^{2}}{\Theta} v^{2} \mathrm{~d} x-\frac{1}{2} \int_{\mathbb{R}} \Theta^{\prime \prime} v^{2} \mathrm{~d} x .
\end{equation*}
Recall that according to \cite{CS02}, $\varphi_{c}$ takes values in $[0, c-2 \omega]$ and $\varphi_{c}^{\prime}$ takes values in $[-c, c]$. Thus the conclusion of lemma follows.\end{proof}

\section{Nonlinear Liouville property implies asymptotic stability} \label{sec_3}
In this section, assuming the nonlinear Liouville property (Theorem \ref{th1.1}), we prove the asymptotic stability result (Theorem \ref{th1.2}). Fix  $c >2\omega$  and let  $u_{0} \in Y_{+}$  satisfying  $\|u_{0}-\varphi_{c}\|_{H^{1}} \le  a_{0}$  with  $a_{0}=\min \left(\alpha_0, \delta\left(\alpha_0\right)\right)$.
Here  $\alpha_{0}$  is as in Theorem \ref{th1.1}, and  $\delta $ (in Corollary \ref{para improve}) is chosen so that  $c(t)>2\omega$. We decompose  $u$  as
$$
u(t, x)=\varphi_{c(t)}(x-x(t))+v(t, x-x(t))
$$
with  $v, x(t)$, and  $c$  satisfying the estimates of Lemma \ref{le4.1}. To prove Theorem \ref{th1.2}, it suffices to prove that  $v(t)$  tends to $0$ weakly in  $H^{1}(\mathbb{R})$  when $t \to +\infty$  and that there exists  $c^\star>2\omega $ such that  $c(t) \to c^\star$  in  $\mathbb{R}$  as  $t \to +\infty$.

Assuming Theorem \ref{th1.1}, the convergence of  $v(t)$  to zero will be proved by constructing a limit solution to the CH equation which contradicts the nonlinear Liouville Theorem. The convergence of  $c(t)$  will be a consequence of the weak convergence of  $v(t)$  and the monotonicity related to the energy $F(u)$ of the CH equation (see Corollary \ref{cor FR}).

\subsection{Construction of the limit solution.}
Assume that $v(t)$ does not converge
to $0$ weakly in $H^{1}(\mathbb{R})$. Then orbital stability implies that there exists $\tilde{v}_{0} \in H^1(\mathbb{R}),\tilde{v}_0 \not\equiv 0$, and a sequence $t_n \rightarrow+\infty$,  such that
\begin{equation}\label{weak H^1 converge}
  v (t_n) \rightharpoonup \tilde{v}_0 \ \text{ in } \ H^1(\mathbb{R}) \ \text{ as } \ n \to \infty.
\end{equation}
From the weak convergence \eqref{weak H^1 converge}, we deduce that $\|\tilde{v}_0\|_{H^1} \le  \alpha_0$.
On the other hand, extracting a subsequence from $\{t_n\}$, still denoted by $t_n$, we deduce the existence of $\tilde{c} >2\omega$ such that $c(t_n) \to \tilde{c}$ as $n \to \infty$. Now let $\tilde{u}_{0}=\varphi_{\tilde{c}} +\tilde{v}_0$ and
consider $\tilde{u}$ the solution of the CH equation with initial data $\tilde{u}(0)=\tilde{u}_0$. To prove that \eqref{weak H^1 converge} leads to a contradiction if  $\tilde{v}_0 \not\equiv 0$, it suffices to prove the $H^1$-localization of $\tilde{u}$ up to a translation. This will then contradict Theorem \ref{th1.1}. We again decompose as before
\begin{equation*}
  \tilde{u}(t,x)=\varphi_{\tilde{c}(t)}(x-\tilde{x}(t))+\tilde{v}(t,x-\tilde{x}(t)).
\end{equation*}
Let us to show the following $H^1$-localization of the solution $\tilde{u}$ up to a translation.
\begin{proposition}\label{pr H^1 compactness}
  The function $\tilde{u}$ constructed above is $H^1(\R)$-localized, that is,  for any $\varepsilon > 0 $ there exists $R_{\varepsilon} > 0$ such that for all $t \in \R$,
\begin{equation}\label{H1-locali}
  \int_{|x|>R_{\varepsilon}} \left( \tilde{u}^2(t,x+\tilde{x}(t)) + \tilde{u}_{x}^{2}(t,x+\tilde{x}(t)) \right) \rmd x < \varepsilon.
\end{equation}
\end{proposition}
A crucial ingredient in the proof of Proposition \ref{pr H^1 compactness} is the continuity of the flow of the CH equation in $H^1(\mathbb{R})$ under the weak topology.
\begin{proof}
Let  $u_{0} \in Y_{+}$  satisfying  $\|u_{0}-\varphi_{c}\|_{H^{1}} \le  a_{0}$ with  $a_{0}=\min \left(\alpha_0, \delta\left(\alpha_0\right)\right)$. We recall that the solution $u$ satisfies Corollary \ref{cor 1}. In view of \eqref{4.7}, one infers that $\{x(t_n+\cdot)-x(t_n)\}$ is uniformly equi-continuous, the Arzela-Ascoli theorem ensures that there exists subsequence $\{t_n\}$ and $\tilde{x}\in \mathcal{C}(\R)$ such that for all $T>0$,
\begin{equation}\label{conv sequ}
x(t_n+\cdot)-x(t_n) \rightarrow \tilde{x} \text { in } C([-T, T]).
\end{equation}
Since $u(t_n)\in Y_+$, there exists $\tilde{u}_{0}\in Y_+$ and subsequence of $\{t_n\}$ such that
\begin{eqnarray}
&& u(t_{n}, \cdot+x(t_{n}))\rightharpoonup \tilde{u}_0 \text { in } H^{1}(\mathbb{R}) \text { as } n \rightarrow+\infty,\nonumber\\
&& u(t_{n}, \cdot+x(t_{n}))\rightarrow \tilde{u}_0 \text { in } H_{loc}^{1}(\mathbb{R}) \text { as } n \rightarrow+\infty.\label{loca strong}
\end{eqnarray}
On account of \eqref{conv sequ} and Proposition \ref{pr2.5}, for any $t\in \R$,
\begin{eqnarray}
&& u(t_{n}+t, \cdot+x(t_{n}+t))\rightharpoonup \tilde{u}(t,\cdot+\tilde{x}(t) \text { in } H^{1}(\mathbb{R}) \text { as } n \rightarrow+\infty,\label{g weak}\\
&& u(t_{n}+t, \cdot+x(t_{n}+t))\rightarrow\tilde{u}(t,\cdot+\tilde{x}(t) \text { in } H_{loc}^{1}(\mathbb{R}) \text { as } n \rightarrow+\infty.\label{gloca strong}
\end{eqnarray}
In view of \eqref{g weak}, we infer from the uniqueness result in Lemma \ref{le4.1} that $\tilde{x} \in \mathcal{C}^1(\R) $ and satisfies \eqref{4.7}.

The proof of the $H^1$-localization of the asymptotic object $\tilde{u}$ will be proceed by a contradiction. If \eqref{H1-locali} does not hold, then a loss of the local $H^1$ norm occurs (recall that the $H^1$ norm is conserved) as  $t\rightarrow\infty$. This leads to a loss of the local $H^1$ norm for the sequence $\{u(t_n + \cdot, \cdot + x(t_n))$\}. Then the control of the $H^1$ norm on the right and on the left leads to an infinite loss on the left, which contradicts the $H^1$ norm conservation.
To proceed, we set
\begin{equation*}
  \begin{aligned}
  E_0 = \frac{1}{2} \int_{\R} \big(\tilde{u}^2(t,x)+\tilde{u}^{2}_{x}(t,x)\big)\,dx
  = \frac{1}{2} \int_{\R} \big(\tilde{u}^2(t,x+\tilde{x}(t))+\tilde{u}^{2}_{x}(t,x+\tilde{x}(t))\big)\,dx.
\end{aligned}
\end{equation*}
Assume that the $H^1$-localization of $\tilde{u}$ does not hold. Then there exist $\delta >0$ such that for any $y_0 > 0$, there exist $t_0=t_{0}(y_{0}) \in \mathbb{R}$ for which $\tilde{u}(t_{0})$ satisfies
\begin{equation*}
  \frac{1}{2} \int_{|x|<2y_{0}}(\tilde{u}^2(t_{0},x+\tilde{x}(t_0))+\tilde{u}^2_{x}(t_{0},x+\tilde{x}(t_{0})))dx \leq m_{0}-\delta_0.
\end{equation*}
Now take $y_0 \geq 2B$ ($B$ as in Lemma \ref{3.1.}) with in addition
$$
\frac{1}{2}\int(\Psi_K(x+y_{0})-\Psi_K(x-y_{0}))(\tilde{u}^2(0,x)+\tilde{u}^{2}_{x}(0,x)) \geq E_{0}-\delta_0,
$$
and
$$
Ce^{\frac{-y_0}{K}}+E_0 \sup_{|x|>2y_0}(\Psi_K(x+y_0)- \Psi_K(x-y_0))  \leq \frac{\delta_{0}}{10}.
$$
Using the fact that $\Psi_K$ is increasing and $0 < \Psi_K(x) < 1$, we deduce that
\begin{equation*}
  \begin{aligned}
\frac{1}{2} \int(\Psi_K(x+y_{0})-\Psi_K(x-y_{0}))(\tilde{u}^2(t_0,x+\tilde{x}(t_0))+\tilde{u}^2_{x}(t_{0},x+\tilde{x}(t_0)))dx \leq E_{0}-\frac{9}{10}\delta_0.
\end{aligned}
\end{equation*}
Assume that $t_0=t_{0}(y_0)>0$ (the proof of the case $t_0(y_{0})<0$ is similar). Define the local $H^1$-norm of $\tilde{u}$ by
\begin{equation*}
  \begin{aligned}
    &E_{loc}(\tilde{u}(t))=\frac{1}{2}\int_{\R} (\Psi_K(x+y_0)-\Psi_K(x-y_0))(\tilde{u}^2(t,x+\tilde{x}(t))+\tilde{u}^{2}_{x}(t,x+\tilde{x}(t)))dx,
  \end{aligned}
\end{equation*}
then one has
$$
E_{loc}(\tilde{u}(0)) \geq E_{0}-\frac{\delta_0}{10} \quad \text{ and } \quad E_{loc}(\tilde{u}(t_0))\leq E_{0}-\frac{9}{10}\delta_0.
$$
It reveals from \eqref{gloca strong} that there exists $N_0>0$ such that for $n\geq N_0$,
$$
E_{loc}(\tilde{u}(t_n)) \geq E_{0}-\frac{\delta_0}{5} \quad \text{ and } \quad E_{loc}(\tilde{u}(t_n+t_0))\leq E_{0}-\frac{4}{5}\delta_0.
$$
Which in turn implies that
$$
E_{loc}(\tilde{u}(t_n)) - E_{loc}(u(t_n,t_0)) \geq \frac{3}{5}\delta_0.
$$
Note that
$$
E(\tilde{u}(t))=E_{loc}(\tilde{u}(t))+E_{L}(\tilde{u}(t))+E_{R}(\tilde{u}(t)),\quad E(u(t_{n}))=E(u(t_{n}+t_{0})).
$$
Then one has
\begin{equation}\label{pr H^1concer}
E_{R}(u(t_{n}))-E_{R}(u(t_n+t_0))-E_{R}(u(t))-E_{L}(u(t)) \geq \frac{3}{5} \delta_{0}.
\end{equation}
It is deduced from \eqref{pr H^1concer} that for all $n \ge N_0$,
\begin{equation}
  E_{L}(u(t_{n}+t_{0}))\geq E_{L}(u(t_n)) +  \frac{1}{2} \delta_{0}.
\end{equation}
We may assume (extracting a subsequence if necessary) that the sequence $\{t_n\}$
satisfies $t_n+1 \geq nt + t_0$ for all $n \geq N_0$. It is deduced from Corollary \ref{cor 1} that
$$
E_{L}(u(t_{n}+t_{0}))\geq E_{L}(u(t_n+t_0)) -  \frac{1}{10} \delta_{0}, \ \text{for all} \ n \geq N_0.
$$
This inequality implies that
$$
E_{L}(u(t_{n+1}))\geq E_{L}(u(t_n)) +  \frac{2}{5} \delta_{0}, \ \text{for all} \ n \geq N_0.
$$
but this is impossible since $E_L(u(t_n)) \leq  E(u(t_n)) = E(u(0))$.

It remains to consider the case $t_0(y_0) \leq  0$. We can suppose that $y_0$ has been chosen such that $E_R(0) \le  \frac{\delta_0}{10}$. Then we proceed as above, the only difference being that the inequalities are inverted (because $t_n + t_0 < t_n$) and we get a contradiction
by proving that $E_R(0) \geqslant \frac{3 \delta_0}{10}$. This completes the proof of Proposition \ref{pr H^1 compactness}.
\end{proof}

\subsection{Convergence of $c(t)$}\label{conver c}
In this subsection, we will prove the convergence of $c(t)$ by employing Corollary \ref{cor FR} and the weak convergence of $v(t)\rightharpoonup0$ in $H^1$.
Indeed, the almost monotonicity of $F_R$ and the smallness of $u$ on the right imply an almost monotonicity of a local energy. Then,   the modulation of $u$ will imply an almost monotonicity of $F(\varphi_{c(t)})$. Finally, writing $F(\varphi_{c(t)})$ in term of $c(t)$ will yield the convergence.

Recall that
$$
F_R(t)=\frac{1}{2}\int_{\R} \big(u^3+uu^2_x+2\omega u^2\big)(t,x+x(t)) \Psi_K(x+x_0)dx,
$$
We claim that for given $\varepsilon > 0$, there exists $R_0 > 0$ such that for all $x_0 \geq\max{(2B, R_0)}$,
the following is satisfied for all $t \ge 0$
\begin{equation}\label{concerntra}
  \left|F_{R}(t)-\frac{1}{2}\int_{|x| < 2x_{0}}\big(u^3+uu^2_x+2\omega u^2\big)(t,x+x(t)) \Psi_K(x+x_0)dx\right| \le \varepsilon.
\end{equation}
Indeed, from Corollary \ref{cor 1}, one has
\begin{equation*}
  \begin{aligned}
    &\left|\int_{x > 2x_{0}}\big(u^3+uu^2_x+2\omega u^2\big)(t,x+x(t)) \Psi_K(x+x_0)dx \right|\\
    &\le C\int_{\R}\big(u^2+u^{2}_{x}\big)(t,x+x(t))\Psi_K(x+x_0))dx\\
    &\le C\int_{\R}\big(u^2+u^{2}_{x}\big)(0,x+x(0))\Psi_K(x+x_0))dx + Ce^{\frac{-x_0}{K}}.
  \end{aligned}
\end{equation*}
 Taking $R_0$ sufficiently large we find
$$
\left| \int_{x > 2x_{0}}\big(u^3+uu^2_x+2\omega u^2\big)(t,x+x(t)) \Psi_K(x+x_0)dx \right| \le \frac{\varepsilon}{2}.
$$
On the other hand, taking again $R_0$ large enough gives
$$
\left| \int_{x < -2x_{0}}\big(u^3+uu^2_x+2\omega u^2\big)(t,x+x(t)) \Psi_K(x+x_0)dx \right| \le \frac{\varepsilon}{2},
$$
since $\lim_{x \to -\infty} \Psi_K (x) = 0$. Hence, estimate \eqref{concerntra} is proved.
For $v \in H^1(\mathbb{R})$, we set
$$
F_{loc}(v)=\int_{|x|<2x_0} \big(v^3+vv^2_x+2\omega v^2\big)(t,x) \Psi_K(x+x_0)dx.
$$
Once again, we take $R_0$ sufficiently large such that for all real $t$
\begin{equation*}
  |F_{loc}(\varphi_{c(t)})-F(\varphi_{c(t)})| \le \varepsilon.
\end{equation*}
Corollary \ref{cor FR}  and inequality \eqref{concerntra} reveal that for all  $t$, for all $t^{\prime} \le t$,
\begin{equation}\label{local est}
  F_{loc}(u(t,\cdot+x(t))) \le  F_{loc}(u(t^{\prime},\cdot+x(t^{\prime}))) + 2\varepsilon.
\end{equation}
From the weak convergence of $v(t)$ to zero in $H^1(\mathbb{R})$, we deduce that $v(t)\rightarrow 0$ in
$L^{2}_{loc}(\mathbb{R})$. Recall that $u(t, x + x(t)) = \varphi_{c(t)}(x) + v(t, x)$,  we deduce from \eqref{local est} that there exists $t_0 > 0$ such that for all $t$ and $t^{\prime},t_0 \le  t^{\prime} \le  t$,
\begin{equation}\label{energy est}
F(\varphi_{c(t)})\le  F(\varphi_{c(t^{\prime})}) + 5\varepsilon.
\end{equation}
Now, since \eqref{eq:Fderivative to c} gives that $F(\varphi_{c(t)}) = f(c(t))$, with $f'(c) =4\kappa c^2>0$. Hence, we deduce from \eqref{energy est} that for all $\varepsilon> 0$, there exists $t_0 > 0$ such that for all $t$ and $t^{\prime},t_0 \le  t^{\prime} \le  t$,
\begin{equation*}
  c(t) \le  c(t^{\prime}) + \varepsilon.
\end{equation*}
This implies the existence of $c^{\star}>2\omega$ with $c(t) \to c^{\star}$ as $t \to +\infty$. Moreover, from Lemma \ref{le4.1}, $|c^{\star} - c_0|$ is controlled in terms of $\alpha_0$.

\subsection{Strong convergence on $(2\omega t,+\infty)$}\label{conver 2omega}
In this subsection, we give the proof of \eqref{asympt lim} and conclude the proof of  Theorem \ref{th1.2}.

Let $y_0>0$ and $\gamma_0=\frac{\sigma_0}{4}$, it is deduced from Lemma \ref{3.1.} (by arguing backwards in time (from $t$ to $0$) and conservation of $E(u)$) that
\begin{equation*}
\int_{\mathbb{R}}(u^2+u_{x}^2)(t,x)\Psi_{K}\big(x-x(t)-y_{0}\big)\rmd x\leq\int_{\mathbb{R}}(u^2+u_{x}^2)(0,x)\Psi_{K}\big(x-x(0)-\frac{\sigma_0t}{2}-y_{0}\big)\rmd x+Ce^{-\gamma_0y_0}.
\end{equation*}
Therefore, the decay property of $\varphi_c$ reveals that
\begin{equation*}
\int_{x>x(t)+y_0}(v^2+v_{x}^2)(t,x)\rmd x\leq2\int_{\mathbb{R}}(u^2+u_{x}^2)(0,x)\Psi_{K}\big(x-x(0)-\frac{\sigma_0t}{2}-y_{0}\big)\rmd x+Ce^{-\gamma_0y_0}.
\end{equation*}
The weak convergence of $v$ to $0$ implies that, for fixed $y_0$, $\int_{x(t)<x<x(t)+y_0}(v^2+v_{x}^2)(t,x)\rmd x\rightarrow0$ as $t\rightarrow+\infty$, then we have
\[
\lim_{t\rightarrow+\infty}\int_{x>x(t)}(v^2+v_{x}^2)(t,x)\rmd x=0.
\]
It remains to show $\lim_{t\rightarrow+\infty}\int_{x>2\omega t}(v^2+v_{x}^2)(t,x)\rmd x=0$. Indeed, let $0<t'=t'(t)<t$ such that $x(t')-2\omega t'-\frac{c_1^0-2\omega}{4}(t+t')=y_0$. Then for $\sup_{t\geq0}\|v(t)\|_{H^1}$ small enough, we have
\begin{eqnarray*}
&&\int_{\R}(u^2+u_{x}^2)(t,x)\Psi_{K}\big(x-2\omega t\big)\rmd x\\ &&\leq \int_{\R}(u^2+u_{x}^2)(t',x)\Psi_{K}\big(x-2\omega t'-\frac{c_1^0-2\omega}{4}(t-t')\big)\rmd x+Ce^{-\gamma_0y_0}\\ &&\leq \int_{\R}(u^2+u_{x}^2)(t',x)\Psi_{K}\big(x-x(t')+y_{0}\big)\rmd x+Ce^{-\gamma_0y_0}.
\end{eqnarray*}
This concludes the proof of \eqref{asympt lim}.

We are now able to summarize above to conclude the proof of Theorem \ref{th1.2}.

\begin{proof}[Proof of Theorem \ref{th1.2}.]
 Assuming the nonlinear Liouville property (Theorem \ref{th1.1}), one needs to prove that  $v(t)$  tends to $0$ weakly in  $H^{1}(\mathbb{R})$  when $t \to +\infty$  and that there exists  $c^\star>2\omega $ such that  $c(t) \to c^\star$  in  $\mathbb{R}$  as  $t \to +\infty$. By contradiction, if $v(t)$ does not converge
to $0$ weakly in $H^{1}(\mathbb{R})$. Then orbital stability implies that there exists $\tilde{v}_{0} \in H^1(\mathbb{R}),\tilde{v}_0 \not\equiv 0$, and a sequence $t_n \rightarrow +\infty$,  such that $v(t_n) \rightharpoonup \tilde{v}_0 \ \text{in} \ H^1(\mathbb{R}) \ \text{as} \ n \to \infty$, $\|\tilde{v}_0\|_{H^1} \le  \alpha_0$ and $c(t_n) \to \tilde{c}>2\omega$ as $n \to \infty$. Now let $\tilde{u}_{0}=\varphi_{\tilde{c}} +\tilde{v}_0$ and
consider $\tilde{u}$ the solution of the CH equation with initial data $\tilde{u}(0)=\tilde{u}_0$. From Proposition \ref{pr H^1 compactness} and Theorem \ref{th1.1}, the solution $\tilde{u}$ is $H^1$-localized and then must be solitons. Therefore, we have $v\equiv0$, but this is a contradiction and \eqref{weak asympt lim} holds true. The convergence of $c(t)$ and \eqref{asympt lim} has been shown in subsection \ref{conver c} and \ref{conver 2omega}, respectively. This completes the proof of Theorem \ref{th1.2}.
\end{proof}
\section{Proof of  Liouville property} \label{sec 4}
The aim of this section is to
reduce the proof of the Liouville Theorem \ref{th1.1} to the proof of a linear
Liouville problem (see Proposition \ref{pro.linear problem} below). Consider $u$ a solution of the CH equation
which satisfies the assumptions of Theorem \ref{th1.1}. We decompose $u$ as in Section \ref{sec_3}
\[
u(t,x)=\varphi_{c(t)}(x-x(t))+v(t,x-x(t)).
\]
It is deduced from Proposition \ref{pr H^1 compactness}  that $u$ is $H^1$-localized, then Proposition \ref{th3.1} implies that $u(t, \cdot+x(t))$ and $u_x(t, \cdot + x(t))$ decay exponentially in space, and uniformly in time.

Arguing in terms of $v$, we reduce the proof of the Liouville theorem
to the proof of the following. Consider $v$, a solution of \eqref{4.15} (with $N=1$) with sufficiently small
initial data $v_0$ in $H^1(\R)$, such that for all $t \in \mathbb{R}$,
\begin{equation}\label{decay of var}
|v(t,x)|+|v_x(t,x)|\leq Ce^{-\sigma_0|x|},
\end{equation}
and such that the orthogonality conditions \eqref{orthogonal cond} are satisfied, then we need to show $v\equiv0$.

\subsection{Reduction to a linear Liouville theorem} \label{sec 41}
We will argue by contradiction. If Theorem \ref{th1.1} does not hold, then
there exists a sequence $\{v_{n}\}$ of solutions of \eqref{4.15} (with $N=1$) such that $v_{n}(0)\rightarrow 0$ in $H^1(\R)$ and
for all integer $n$, $v_{n}\neq0$, $v_{n}$ satisfies \eqref{decay of var} and the orthogonality conditions \eqref{4.6}. Thanks to the orbital stability of the solitons (see \cite{CS02}), we deduce that
$a_n:=\sup\limits_{t\in\R}\|v_{n}\|_{H^1}\rightarrow 0$ as $n\rightarrow+\infty$. Considering a normalized subsequence of $\{v_{n}\}_n$, we will
construct a nonzero solution of the following linear Liouville problem.

\begin{proposition} \label{pro.linear problem}
Let $u_n$ be a solution of the CH equation \eqref{eq1} initially close to $\varphi_c$. Suppose that $c_n(t)$ and $x_n(t)$
are its modulation parameters and $v_n$ being as in Lemma \ref{le4.1}. We assume that $v_{n}\neq0$
and $v_{n}$ satisfies the decay estimate \eqref{decay of var}. Assume moreover that
$$a_n=\sup\limits_{t\in\R}\|v_{n}\|_{H^1}\rightarrow 0, \quad \text{as}\quad  n\rightarrow+\infty.$$
Then there exist $c>2\omega$, a subsequence of $\{v_{n}\}$ (still denoted by $\{v_{n}\}$) and a sequence of $\{t_{n}\}$, such that
\[
\frac{v_{n}(t+t_n)}{a_n}\rightarrow v(t) \quad \text{in}\quad L^{\infty}_{loc}(\R,H^1(\R)),
\]
where  $v\in \mathcal C(\R,H^1(\R))\cap  L^{\infty}(\R,H^1(\R))$ is a nontrivial solution of
\begin{equation}\label{linearized equa}
v_t+\mathcal J L_1v=\beta(t)\partial_x\varphi_c,
\end{equation}
for some continuous function $\beta(t)$. In addition, for any $t\in\R$, $v(t)$ satisfies \eqref{decay of var} and
\begin{eqnarray}
&& \left \langle v(t),(1-\partial_{x}^{2})\partial_x\varphi_c\right \rangle =\langle v(t),(1-\partial_{x}^{2})\varphi_c\rangle =0.\label{orthogonal cond}
\end{eqnarray}
\end{proposition}
\begin{proof}
 By the definition of $a_n$, there exists a sequence $\{t_{n}\}_n$ such that $\|v_n(t_n)\|_{H^1}\geq \frac{a_n}{2}$. We define $w_n$ by setting for $(t,x)\in\R^2$,
\[w_n(t,x) := \frac{v_{n}(t+t_n)}{a_n}.\]
Then we see that $w_n$ satisfies
\begin{eqnarray}
&& \|w_n(0,x)\|_{H^1}\geq \frac12,\quad  \|w_n(t,x)\|_{H^1}\leq1, \label{w bounded}\\
&& \left \langle w_n(t),(1-\partial_{x}^{2})\partial_x\varphi_{c_n(t+t_n)}\right \rangle =\langle w_n(t),(1-\partial_{x}^{2})\varphi_{c_n(t+t_n)}\rangle =0,\label{orthogonal cond 2}\\
&& |w_n(t,x)|+|\partial_x w_{n}(t,x)|\leq Ce^{-\sigma_0|x|}. \label{exponetial decay2}
\end{eqnarray}
  Moreover, $w_n$ satisfies the equation
\begin{equation}\label{wn equa}
\begin{split}
        \partial_t w_{n}+\mathcal{J}L_{c_n(t+t_n)}w_n & = (\partial_t x_{n} - c_n) \partial_x w_{n}
        + \frac{\partial_t x_n-c_n}{a_n} \partial_{x} \varphi_{c_n(t+t_n)} - \frac{\partial_t c_n}{a_n} \partial_c \varphi_{c_n(t+t_n)} \\ & \quad - \frac{a_n}{2}\partial_x(w_{n}^{2})-a_n(1-\partial_{x}^{2})^{-1}\partial_{x} \left( \frac12 \partial_x w_n^{2}+w_{n}^{2} \right)  \\
        w_n(0) & = \frac{v_n(t_n)}{a_n}.
 \end{split}
  \end{equation}
  Where the operator $L_{c_n(t+t_n)}$ is defined in \eqref{L1oper} (with $c$ replaced by $c_n(t+t_n)$), $x_n(t)$ and $c_n(t)$ are parameters such that $$u_n(t,x):=\varphi_{c_n(t)}(x-x_n(t))+v_n(t,x-x_n(t))$$ satisfies the following orthogonal conditions
  \begin{equation}\label{northo}
    \left \langle v_n(t),(1-\partial_{x}^{2})\partial_x\varphi_{c_n(t)}\right \rangle =\langle v_n(t),(1-\partial_{x}^{2})\varphi_{c_n(t)}\rangle=0.
  \end{equation}
  Since $\|w_n(0)\|_{H^1} \le  1$ and $c_n(0)>2\omega$, there exist a subsequence of $\left\{ w_n(0)\right\}_{n}$ and a subsequence of $\left\{ c_n(0)\right\}_{n}, v_0 \in H^1(\mathbb{R})$ and $c >2\omega$ such that
  \begin{equation*}
    w_n(0) \rightharpoonup v_0 \ \text { in } \ H^1(\mathbb{R}) \ \text { and } \ c_n(0) \rightarrow c.
  \end{equation*}
  Differentiating \eqref{northo} with respect to $t$ and dividing the resulting equation by $a_n$, we have
\begin{align*}
-\frac{\partial_t c_n}{a_n} & \left( \frac12\frac{d}{dc}E(\varphi_{c_n(t+t_n)})-a_n \left\langle w_n(t),(1-\partial_x^2)\partial_c\varphi_{c_n(t+t_n)} \right\rangle \right) \\
& + a_n\left( \left\langle(1-\partial_x^2)(w_n\partial_x w_n),\varphi_{c_n(t+t_n)} \right\rangle - \left\langle w^2_n+\frac12 (\partial_x w_n)^2 ,\partial_x\varphi_{c_n(t+t_n)} \right\rangle \right) = 0, \\
\frac{\partial_t x_n-c_n}{a_n} & \left( \|\partial_x\varphi_{c_n(t+t_n)}\|^2_{L^2} + \|\partial^2_x\varphi_{c_n(t+t_n)}\|^2_{L^2} + a_n \left\langle(1-\partial_x^2)w_n,\partial_x\varphi_{c_n(t+t_n)} \right\rangle \right) \\
& - \partial_t c_n \left\langle (1-\partial_x^2)\partial_x w_n,\partial_x\varphi_{c_n(t+t_n)} \right\rangle - \left\langle w_n,L_{c_n(t+t_n)}\partial^2_x\varphi_{c_n(t+t_n)} \right\rangle\\
& - a_n \left\langle(1-\partial_x^2)(w_n \partial_x w_n)+ \partial_x \left( w^2_n+\frac12 (\partial_x w_n)^2 \right),\partial_x\varphi_{c_n(t+t_n)} \right\rangle=0.
\end{align*}
It is deduced from \eqref{4.5} that
\[\sup_{t\in\R}|c_n(t)-c|\leq Ca_n.\]
Therefore, from the above, one has
\begin{equation}\label{bdd of cx}
   \left| \frac{\partial_t c_n}{a_n} \right| \leq Ca_n, \qquad  \left| \frac{\partial_t x_n-c_n}{a_n}-\frac{1}{\|\partial_x\varphi_{c}\|^2_{L^2}+\|\partial^2_x\varphi_{c}\|^2_{L^2}} \langle w_n,L_{c}\partial^2_x\varphi_{c} \rangle \right|\leq Ca_n.
  \end{equation}
It is deduced from \eqref{w bounded} and \eqref{bdd of cx} that
\[\sup_{n\in\mathbb {N}}\sup_{t\in\R}\big(\|w_n(t)\|_{H^1}+\|\partial_tw_n(t)\|_{L^2}\big)<+\infty.\]
It reveals from   \eqref{exponetial decay2} that the set $\{ w_n(t)\ | \ t\in[-T,T], n\in\mathbb{N}\}$ is relatively compact in $L^2(\R)$ for any $T>0$.  Then by the Arzel\`a--Ascoli theorem, we see that there exists a subsequence of $w_n\in \mathcal {C}(\R;H^1(\R))$, still denoted by $\{w_n\}$, such that
\[
w_n(t)\rightarrow w(t) \quad \text{in }\ L^{\infty}_{loc}(\R;L^2(\R)).
\]
Moreover, it follows from \eqref{w bounded} and \eqref{exponetial decay2} that $w_n(0)\rightarrow w(0)\not\equiv0$.
Therefore, it follows from \eqref{bdd of cx} that
\[
\frac{\partial_t c_n}{a_n}\rightarrow0,\qquad \frac{\partial_t x_n-c_n}{a_n}\rightarrow\frac{1}{\|\partial_x\varphi_{c}\|^2_{L^2}+\|\partial^2_x\varphi_{c}\|^2_{L^2}} \langle w_n,L_{c}\partial^2_x\varphi_{c} \rangle=:\beta(t),
\]
uniformly on any compact interval of $\R$ as $n\rightarrow+\infty$. In addition, $w(t)$ satisfies \eqref{orthogonal cond}, \eqref{decay of var} and \eqref{linearized equa}. This completes the proof of Proposition \ref{pro.linear problem}.
\end{proof}

\subsection{Proof of the linear Liouville theorem}\label{sec 5}
This is the last step in the proof
of Theorem \ref{th1.1}. We show that the function $v$ constructed in Proposition \ref{pro.linear problem} satisfies
$v\equiv0$. This leads to a contradiction and achieves the proof of the Liouville Theorem. In fact, we will give the classification of solutions for more general linear problem related to \eqref{eq3}.

Let us recall the following higher order linearized operators around one soliton profile $\varphi_c$,
\begin{equation*}
L_n=-H''_{n+1}(\varphi_c)+cH''_{n}(\varphi_c).
\end{equation*}
 The idea of classifying the solutions for more general linear problem is by employing the $L^2$-completeness of the squared eigenfunctions \eqref{eq:closure relation}, a suitable change of function and the semigroup estimate initialed in the work of Pego--Weinstein \cite{PW94} (see also \cite{MT14}) which is dealing with the gKdV equation.
\begin{theorem}\label{thm1.4}
Let $v\in \mathcal C(\R,H^1(\R))\cap L^{\infty}(\R,H^1(\R))$ be a solution of
\begin{equation}\label{higher linearized CH}
v_t+\mathcal{J}L_nv=A(t)\varphi_c'+B(t)\frac{\partial \varphi_c}{\partial c},
\end{equation}
where $A(t)$ and $B(t)$ are two  continuous and bounded functions. Assume that for constant $C>0$, $v(t,x)$ satisfies
\begin{equation*}
\forall t\in \R, \ R_{\varepsilon} \gg 1, \quad  \int_{|x|>R_{\varepsilon}}\big(v^2(t,x)+v_x^2(t,x)\big) dx\leq \varepsilon.
\end{equation*}
Then for  all $ t\in \R$,
\begin{eqnarray*}
v(t) = \beta(t)\varphi'_c+\gamma(t)\frac{\partial \varphi_c}{\partial c}.
\end{eqnarray*}
for some $\mathcal C^{1}$ bounded functions $\beta(t)$ and $\gamma(t)$ satisfying
\begin{equation*}
\beta'(t)=A(t)-c^{n-1}\gamma(t), \quad \gamma'(t)=B(t).
\end{equation*}
If we further impose
\begin{equation*}
\forall t\in \R,\quad \int_{\R}v(t,x)(1-\partial_x^2)\varphi_c(x)\rmd x=\int_{\R}v(t,x)(1-\partial_x^2)\varphi'_c(x)\rmd x=0,
\end{equation*}
then  for all $t\in \R$, $v(t)\equiv 0$.
\end{theorem}

\begin{remark} \label{re1.5}
Theorem \ref{thm1.4} states that solutions of \eqref{higher linearized CH} satisfy the linear Liouville property. We refer to \cite{Ma06} for the counterpart results of gKdV and BBM equations.
\end{remark}

Let us first deal with the case $n=1$. The spectral analysis of $\mathcal J L_1$ has been verified in Proposition \ref{pr3.5-3}. In particular, the eigenfunctions of which are eigenfunctions of adjoint recursion operator $\mathcal{R}^\ast(\varphi_c)$. The operator $\mathcal J L_1$ is degenerate, with an eigenvalue at $\lambda=0$ with (generalized) eigenfunctions $\varphi'_c$ and $\partial_c\varphi_c$, which satisfy
\[\mathcal J L_1\varphi'_c=0,\qquad \mathcal J L_1\partial_c\varphi_c=\varphi'_c.\]
The above two zero modes are related, respectively, to infinitesimal changes in the location and speed of the soliton. They give rise to solutions $\varphi'_c$ and $\partial_c\varphi_c-t\varphi'_c$ to the linearized problem
\[v_t=\mathcal J L_1v,\]
which are, respectively, constant and linearly growing with time.

Following the idea of Pego--Weinstein \cite{PW94}, we define $\mathcal L_a:=e^{ax}\mathcal J L_1e^{-ax}$, where $a$ is well chosen such that the exponential weight shifts the essential spectrum of $\mathcal L_a$ into the open left half plane. We want to show some semigroup decay estimates of $e^{t\mathcal L_a}$.
We first consider the essential spectrum of $\mathcal L_a$.
 The essential spectrum of $\mathcal L_a$ is
  \[\sigma_{ess}(\mathcal L_a):= \left\{ z\in \mathbb C : z=c(ik+a)-2\omega\frac{ik+a}{1-(ik+a)^2} \right\},\ k\in\R.\]
The graph of $\sigma_{ess}(\mathcal L_a)$ has a vertical asymptote $\Re z=ac$ when $|k|\rightarrow+\infty$. The maximal of $\Re z$ is
\[
\Lambda:=ac-\frac{2a\omega }{1-a^2}<0,
\]
if
\[
a_1:=-\sqrt{1-\frac{2\omega}{c}}<a_{\ast}:=-\left(\frac{c+\omega-\sqrt{\omega(4c+\omega)}}{c} \right)^{\frac12}<a<0.
\]

Next, we study the discrete spectrum of $\mathcal L_a$. From Proposition \ref{pr3.5-3}, we see that the eigenvalue is $\lambda=0$ with (generalized) eigenfunctions $e^{ax}\varphi'_c$ and $e^{ax}\partial_c\varphi_c$, which satisfy
\[
\mathcal L_a(e^{ax}\varphi'_c)=0,\qquad \mathcal L_a(e^{ax}\partial_c\varphi_c)=e^{ax}\varphi'_c.
\]
Similarly, the generalized eigenspaces of adjoint operator $\mathcal L^{\ast}_a:=-e^{-ax} L_1\mathcal Je^{ax}$ are then
$$
\operatorname{gKer}(\mathcal L^{\ast}_a)=\operatorname{Span}\left\{ e^{-ax}m_{\varphi}, \ e^{-ax}\partial_x^{-1}\left( \frac{\partial m_{\varphi}}{\partial c}\right) \right\}.
$$
Recall that for an operator $A$ defined in $L^2(\R)$,
$$
\operatorname{gKer}(A) := \bigcup\limits_{k=1}^{\infty}\operatorname{Ker}(A^k).
$$

Now we have the following spectral projection for the zero eigenvalue.
\begin{proposition}\label{pro5.1}
Assume that $a_1<a<0$. Then $\lambda=0$ is an eigenvalue for the operators $\mathcal L_a$ and $\mathcal L^{\ast}_a$ with algebraic multiplicity two. Moreover,
\begin{eqnarray*}
&&\operatorname{gKer}(\mathcal L_a)=\operatorname{Ker}(\mathcal L^{2}_a)=\operatorname{Span}\{f_1,f_2\},\\
&&\operatorname{gKer}(\mathcal L^{\ast}_a)=\operatorname{Ker}(\mathcal L^{\ast2}_a)=\operatorname{Span}\{g_1,g_2\},
\end{eqnarray*}
where
\begin{eqnarray*}
&&f_1 := e^{ax}\varphi'_c,\ f_2 := e^{ax}\partial_c\varphi_c,\\
&& g_1:=e^{-ax}\tilde{g}_1= C_1e^{-ax}\partial_x^{-1}\left( \frac{\partial m_{\varphi_c}}{\partial c} \right)+C_2e^{-ax}m_{\varphi_c},\quad   g_2:=e^{-ax}\tilde{g}_2=C_3e^{-ax}m_{\varphi_c},\\
&& C_1=-C_3=-\bigg(\frac{d}{dc}H_1(\varphi_c)\bigg)^{-1},\quad C_2=\frac{1}{2}\bigg(\frac{d}{dc}\int_{\R}\varphi_c dx\bigg)^{2}\bigg(\frac{d}{dc}H_1(\varphi_c)\bigg)^{-2}.
\end{eqnarray*}
In addition, for $j=1,2$, the functions $f_j$ and $g_j$ are bi-orthogonal, with $\langle f_j, g_k\rangle=\delta_{jk}$ for $j,k=1,2$. The spectral projection $\mathcal P$ for the zero eigenvalue of $\mathcal L_a$ and the complementary spectral projection $\mathcal Q$ are given by
\[
\mathcal P w := \sum\limits_{k=1}^2\langle w, g_k\rangle f_k,\ \quad \mathcal Q w := (I-\mathcal P)w=w-\sum\limits_{k=1}^2\langle w, g_k\rangle f_k,\quad \text{for }\  w\in L^2(\R).
\]
These projections satisfy $\mathcal P \mathcal L_aw=\mathcal L_a\mathcal P w$ and $\mathcal Q \mathcal L_aw=\mathcal L_a\mathcal Q w$ for $w\in \operatorname{dom}(\mathcal L_a)$.
\end{proposition}
\begin{proposition}\label{pro5.2}
Assume that $a_1<a<0$.  Let $\mathcal Q$ denote the projection
onto $\operatorname{gKer}(\mathcal L^{\ast}_a)^{\perp}$. Then the initial value problem
\begin{align*}
\begin{cases}
w_t=\mathcal L_a w, \\
w(0)=w_0\in H^1(\R)\cap \operatorname{Range} \mathcal Q,
\end{cases}
\end{align*}
has a unique solution $e^{t\mathcal L_a}w_0\in \mathcal{C}(\mathbb{R}_+;H^1(\mathbb{R}))$, satisfying
\[
\|e^{t\mathcal L_a}w_0\|_{H^1}\leq C e^{-bt}\|w_0\|_{H^1},
\]
for some constant $C>0$ and $0<b<b_{max}$, with
\[
-b_{max} := \inf \left\{ -b : \ \lambda=0 \ \text{is the only eigenvalue of}\ \mathcal L_a \ \text{with}\ \Re \lambda\geq -b>  \Lambda \right\}.
\]
\end{proposition}

In order to show Proposition \ref{pro5.2}, similar to \cite{MW96}, we employ an abstract theorem of Pr\"{u}ss \cite{P84} which is stated as follows.
\begin{proposition}\label{Pruss theorem}
Let $B$ be the infinitesimal generator of a $\mathcal C_0$ semigroup on the
Hilbert space $Z$. Let $b>0$. If there exists $M >0$ such that
\begin{equation}\label{bdd resolv}
\|(\lambda I-B)^{-1}\|_{Z\rightarrow Z}\leq M,
\end{equation}
for all $\lambda$ with $\Re \lambda> -b$, then
\begin{equation*}
\|e^{Bt}\|_{Z\rightarrow Z}\leq e^{-bt}.
\end{equation*}
\end{proposition}
In view of Pr\"{u}ss's result, we only need to prove the boundedness of the resolvent in \eqref{bdd resolv}.

\begin{proof}[Proof of Proposition \ref{pro5.2}]

{\bf{ Step 1.}}  We show that $\mathcal L_a$ generates a $\mathcal C_0$ semigroup in $H^1(\R)$. A straight Fourier transformation upon the constant coefficient operator $\mathcal L^{\infty}_a$ shows that $\mathcal L^{\infty}_a$ generates a $\mathcal C_0$ semigroup in $H^1(\R)$.
Then
\[\mathcal L_a-\mathcal L^{\infty}_a=e^{ax}\mathcal J(\varphi_c\partial_x^2+\varphi_c'\partial_x+\varphi_c''-3\varphi_c)e^{-ax},\]
is bounded in $H^1(\R)$, therefore, $\mathcal L_a$ generates a $\mathcal C_0$ semigroup in $H^1(\R)$.

{\bf{ Step 2.}} We will show \eqref{bdd resolv} with $B=\mathcal L_a \mathcal Q$ and $Z=H^1(\R)$. Notice that
\[ (\lambda I-\mathcal L_a)^{-1}\mathcal Q=(I-C(\lambda))^{-1}(\lambda I-\mathcal L^{\infty}_a)^{-1}\mathcal Q,\]
where
$$
C(\lambda):=\bigg(D_a(-cD_a^2+c-2\omega)+\lambda(1-D_a^2)\bigg)^{-1}D_a(-D_a\varphi_cD_a+3\varphi_c-\varphi_c''),\quad D_a := \partial_x-a.
$$
Firstly, the operator norm of $(\lambda I-\mathcal L^{\infty}_a)^{-1}\mathcal Q$ can be controlled since for any $-b> \Lambda$, there exists $M_1>0$ such that $\|(\lambda I-\mathcal L^{\infty}_a)^{-1}\mathcal Q\|_{H^1\rightarrow H^1}\leq M_1$ for any $\Re \lambda>-b$.

We claim that there exist $R_1>0$ and $M_2>0$ such that $\|(I-C(\lambda))^{-1}\|_{H^1\rightarrow H^1}\leq M_2$ for $\Re \lambda>-b$ and $|\lambda|>R_1$. Indeed, it suffices to show that the operator $C(\lambda):L^2\rightarrow L^2$ is compact for $\lambda\notin \sigma_{ess}(\mathcal L_a)$ and the operator norm $\|C(\lambda)\|_{L^2\rightarrow L^2}$ is strictly less than $1$. To show $C(\lambda)$ is compact, it suffices to show that its adjoint $C^{\ast}(\lambda):L^2\rightarrow L^2$ is compact. For this, observe that
\[C^{\ast}(\lambda)=A(x)B(i\partial_x,\lambda),\]
where $A(x) := -D_a\varphi_cD_a+3\varphi_c-\varphi_c''$ and
\[
B(k,\lambda) := \frac{ik-a}{(ik-a)\left( -c(ik-a)^2+c-2\omega\bigg)+\lambda\bigg(1-(ik-a)^2 \right)}.
\]
Then $C^{\ast}(\lambda)$ is an integral operator with kernel $A(y)\hat{B}(x-y,\lambda)$. Since $A(x)$  decays exponentially as $y\rightarrow\pm\infty$ and $B(\cdot,\lambda)\in L^2(\R)$ for $\lambda\notin \sigma_{ess}(\mathcal L_a)$, which reveals that
\[\int_{\R}\int_{\R}|A(y)\hat{B}(x-y,\lambda)|^2 dxdy<+\infty.\]
Therefore, $C^{\ast}(\lambda):L^2\rightarrow L^2$ is a Hilbert-Schmidt operator  whenever $\lambda\notin \sigma_{ess}(\mathcal L_a)$.

Now we show that there exists $\delta\in(0,1)$ and $M>0$ such that $\|C(\lambda)\|_{L^2\rightarrow L^2}<1-\delta$ for $\Im \lambda>cM\gamma^3$ or $\Re \lambda>cM\gamma^3$ with $\gamma:=\sqrt{\frac{c-2\omega}{c}}$, which can be checked directly within the Fourier transform upon $C(\lambda)$.
Then we see there exists $M_2>0$ and $R_1>0$ such that $\|(I-C(\lambda))^{-1}\|_{H^1\rightarrow H^1}<M_2$ for $\Re \lambda>-b_{max}$ and $|\lambda|>R_1$,
 which reveals that $\|(\lambda I-\mathcal L_a)^{-1}\mathcal Q\|_{H^1\rightarrow H^1}=\|(I-C(\lambda))^{-1}(\lambda I-\mathcal L^{\infty}_a)^{-1}\mathcal Q\|_{H^1\rightarrow H^1}\leq M'_2$ is uniformly bounded for $\Re \lambda>-b_{max}$ and $|\lambda|>R_1$.

Lastly, noticing that the singularities of $(\lambda I-\mathcal L_a)^{-1}$ with $\Re \lambda\geq \Lambda$ are isolated and must be the eigenvalues of $\mathcal L_a$. Now we know $\lambda=0$ is only eigenvalue of $\mathcal L_a$ with a nonnegative real part. Therefore, $-b_{max}<0$. Moreover, for any $R>0$ and $b\in(0,b_{max})$, there exists $M_R>0$ such that $\|(\lambda I-\mathcal L_a)^{-1}\mathcal Q\|_{H^1\rightarrow H^1}\leq M_R$ for $\Re \lambda>-b$ and $|\lambda|\leq R$. Therefore, $\|(\lambda I-\mathcal L_a)^{-1}\mathcal Q\|_{H^1\rightarrow H^1}$ is uniformly bounded for all $\Re \lambda>-b>-b_{max}$. The proof of the Proposition is concluded.
\end{proof}

 The general case  essentially follows a parallel argument as above. Indeed, recall that the spectral information of $\mathcal{J}L_{n}$ with $n\geq1$ was shown in Proposition \ref{pr3.5-3}, which reveals that the spectrum of $\mathcal{J}L_{n}$ lies on the imagery axis, with a double eigenvalue at $\lambda=0$ with (generalized) eigenfunctions $\varphi'_c$ and $\partial_c\varphi_c$. We then define $\mathcal L^{(n)}_a:=e^{ax}\mathcal J L_ne^{-ax}$, where $a$ is well chosen such that the exponential weight shifts the essential spectrum of $\mathcal L^{(n)}_a$ into the open left half plane.
 The essential spectrum of $\mathcal L^{(n)}_a$ with $n\geq2$ is
 \[
 \sigma_{ess}(\mathcal L^{(n)}_a):= \left\{ z\in \mathbb C, \ \ z= \left( \frac{2\omega}{1-(ik+a)^2} \right)^{n-1} \left( c(ik+a)-2\omega\frac{ik+a}{1-(ik+a)^2} \right) \right\},\ k\in\R.
 \]
Unfortunately, in this case, the graph of $\sigma_{ess}(\mathcal L_a)$ has a vertical asymptote $\Re z=0$ when $|k|\rightarrow+\infty$ and the maximal of $\Re z$ tends to $0$.  However, for some constant $K>0$ large, the operator $e^{ax}\mathcal J (L_n+KL_1)e^{-ax}=\mathcal L^{(n)}_a+K\mathcal L_a$ shares a similar spectral gap with the operator  $\mathcal L_a$ with $|a| \ll \sqrt{1-\frac{2\omega}{c}}$.

 \begin{proof}[Proof of Theorem \ref{thm1.4}]
The Cauchy problem of \eqref{higher linearized CH} for $n=1$ is globally well-posed in $H^{1}(\R)$. For general $n$,  \eqref{higher linearized CH} is also well-posed in $H^{1}(\R)$, since from Lemma \ref{le3.5}, one has for all $n\geq2$ $$\mathcal{J}L_n=\mathcal{J}\mathcal{R}(\varphi_c)L_{n-1}=\mathcal{R}^{\ast}(\varphi_c)\mathcal{J}L_{n-1}=\big(\mathcal{R}^{\ast}(\varphi_c)\big)^{n-1}\mathcal{J}L_{1},$$
and the well-posedness of Cauchy problem of \eqref{higher linearized CH} in $H^{1}(\R)$ follows from induction over $n$ and the invertibility in $L^2(\R)$ of the adjoint recursion operator $\mathcal{R}^{\ast}(\varphi_c)$ in Lemma \ref{le3.81}.

Now we restrict ourselves to the case $A(t)=B(t)=0$, since one can show that $\int_0^tA(s)ds$ and $\int_0^tB(s)ds$ are uniformly bounded in time. Indeed, multiplying the equation of $v(t)$ \eqref{higher linearized CH} by $(1-\partial_x^2)\varphi_c$ and $(1-\partial_x^2)\varphi_c'$, respectively, integrating on $\R$, using $L_n\varphi_c'=0$, we obtain
\begin{eqnarray*}
&&A(t)=\frac{\langle v(t), L_n\varphi_c''\rangle+\langle v'(t), (1-\partial_x^2)\varphi_c'\rangle}{\|\varphi_c'\|^2_{L^2}+\|\varphi_c''\|^2_{L^2}},\\
&&B(t)=\frac{\langle v'(t), (1-\partial_x^2)\varphi_c\rangle}{\langle \frac{\partial \varphi_c}{\partial c}, (1-\partial_x^2)\varphi_c\rangle}=\frac{\langle v'(t), (1-\partial_x^2)\varphi_c\rangle}{2\kappa c}.
\end{eqnarray*}

In the case of $n=1$, the linear equation \eqref{higher linearized CH} admits conservation law $\langle L_1v,v \rangle$. Under the orthogonal conditions satisfied by $v$, the quadratic form is positive definite. Therefore, for $v\neq 0$, one has $\langle L_1v,v \rangle=C_0>0$. On the other hand, after a suitable change of functions and semigroup decay estimate in Proposition \ref{pro5.2}, one has $\|v\|_{H^1_{loc}}\rightarrow0$ as $t\rightarrow +\infty$. The above two facts will lead to a contradiction. Indeed, we notice from orthogonal conditions  that $\langle v(t),\tilde{g}_2\rangle=0$ for all $t\in\R$. From the linear equation satisfied by $v$, one sees that $\frac{d}{dt}\langle v(t),\tilde{g}_1\rangle=0$, which reveals that $\langle v(t),\tilde{g}_1\rangle=\langle v(0),\tilde{g}_1\rangle$ for all $t\in\R$.

Now we define $$w(0):=v(0)-\frac{\langle v(0),\tilde{g}_1\rangle}{\langle \varphi'_c,\tilde{g}_1\rangle}\varphi'_c,$$
and consider the solution $w(t)$ of linear equation \eqref{higher linearized CH}  emanating from $w(0)$. It reveals that
$$w(t):=v(t)-\frac{\langle v(0),\tilde{g}_1\rangle}{\langle \varphi'_c,\tilde{g}_1\rangle}\varphi'_c,$$
and $\langle w(t),\tilde{g}_1\rangle=\langle w(t),\tilde{g}_2\rangle=0$. Moreover, $w(t)$ possesses exponential decay.

For $a_{\ast}<a<0$, we take $h(t,x)=e^{ax}w(t,x)\in H^1(\R)$ for all $t\in\R$. It follows that
\[h_t=e^{ax}w_t=e^{ax}\mathcal{J}L_1e^{-ax}h=\mathcal L_ah. \]
We now show $v\equiv0$ by contradiction. Suppose that $v(0)\neq0$, then for all $t\in\R$, there holds the conservation law $\langle L_1w(t),w(t)\rangle=\langle L_1v(0),v(0)\rangle$. Moreover, the orthogonal conditions satisfied by $w(t)$ reveals that the coercivity property  $\langle L_1w(t),w(t)\rangle\geq C\|w(t)\|^2_{L^2}\geq C_0>0$ holds true. Therefore, for all $t\in\R$, one has
\[C_0\leq \langle L_1w(t),w(t)\rangle\leq C\|w(t)\|^2_{H^1},\]
Combining the $H^1$-localized property of $w(t)$, we obtain for some $R_0>0$,
\begin{equation}\label{localized version}
\|w(t)\|^2_{H^1(-R_0,R_0)}\geq \frac{C_0}{2}>0.
\end{equation}
We then write $w(t,x)=e^{-ax}h(t,x)$ and deduce from Proposition \ref{pro5.2} that
\begin{eqnarray*}
&&\|w(t)\|^2_{H^1(-R_0,R_0)}\leq(1+a^2)\int_{-R_0}^{R_0}e^{-2ax}\big(h^2(t,x)+h_x^2(t,x)\big)dx\leq3e^{-2aR_0}\|h(t)\|^2_{H^1(-R_0,R_0)}\\
&&\leq3e^{-2aR_0}\|h(t)\|^2_{H^1}\leq Ce^{-2aR_0-2bt}\|h(0)\|^2_{H^1}\rightarrow0,\quad  \text{as}\quad  t\rightarrow +\infty,
\end{eqnarray*}
which contradicts to \eqref{localized version} and so that $v(t)=v(0)\equiv0$.

For the general case of $n\geq2$, we will consider the operator $L_n+KL_1 $ with the constant $K>0$ large. In this case, there exists some $|a| \ll \sqrt{1-\frac{2\omega}{c}}$ such that the spectrum the operator $\mathcal L^{(n)}_a+K\mathcal L_a$ lies in the open left half plane and the semigroup decay estimate of $e^{t (\mathcal L^{(n)}_a+K\mathcal L_a)}$ holds true. Therefore, following the argument of the case $n=1$ as above, any $H^1$ localized solutions of $f_t+\mathcal J (L_n+KL_1)f=0$ must be $f\equiv0$. The proof of Theorem \ref{thm1.4} is completed.
\end{proof}

We are now able to conclude the proof of Theorem \ref{th1.1}.
\begin{proof}[End of proof of Theorem \ref{th1.1}.]
 From Theorem \ref{thm1.4} (with $n=1$), the function $v$ constructed in Proposition \ref{pro.linear problem} satisfies
 $v \equiv 0$, which is a contradiction. This completes the proof of Theorem \ref{th1.1}.
\end{proof}

With Proposition \ref{pr4.5-1} in hand,  we can extend Theorem \ref{thm1.4} to the case of multi-solitons profile.
\begin{corollary}\label{cor5.1}
Let $v\in \mathcal C(\R,H^1(\R))\cap L^{\infty}(\R,H^1(\R))$ be a solution of
\begin{equation*}
v_t+\mathcal{J}\mathcal{L}_Nv=\sum_{j=1}^N \left( A_j(t)\mathcal{J}G_j(U^{(N)})+B_j(t)\frac{\partial U^{(N)}}{\partial c_j} \right),
\end{equation*}
where $A_j(t)$ and $B_j(t)$ are $2N$  continuous and bounded functions. Assume that for constant $R_{\varepsilon}>0$, $v(t,x)$ satisfies
\begin{equation*}
\forall t\in \R, \quad R_{\varepsilon} \gg 1, \quad \int_{|x|>R_{\varepsilon}}\big(v^2(t,x)+v_x^2(t,x)\big) dx\leq \varepsilon.
\end{equation*}
Then for  all $ t\in \R$,
\begin{eqnarray*}
v(t)\equiv \sum_{j=1}^N \left( \beta_j(t)\mathcal{J}G_j(U^{(N)})+\gamma_j(t)\frac{\partial U^{(N)}}{\partial c_j} \right).
\end{eqnarray*}
for some $\mathcal C^{1}$ bounded functions $\beta_j(t)$ and $\gamma_j(t)$ satisfying
\begin{equation*}
\beta_j'(t)=A_j(t)+c_j^{n-1}\gamma_j(t), \quad \gamma_j'(t)=B_j(t)\ j=1,2,\cdot\cdot\cdot,N.
\end{equation*}
If further we impose
\begin{equation*}
\forall t\in \R,\quad \int_{\R}v(t,x)(1-\partial_x^2)U^{(N)}(x)\rmd x=\int_{\R}v(t,x)(1-\partial_x^2)\mathcal{J}G_j(U^{(N)}(x))\rmd x=0,
\end{equation*}
with $j=1,2,\cdot\cdot\cdot,N$, then  for all $ t\in \R$, there holds $v(t)\equiv 0$.
\end{corollary}

\begin{proof}[Sketch of the proof]
It reveals from Proposition \ref{pr4.5-1} that the spectrum of $\mathcal{J}\mathcal{L}_N$ lies in imagery axis, with an eigenvalue at $\lambda=0$ $2N$-fold. In viewing of exponential decay of $N$-solitons profile $U^{(N)}$, \eqref{eq:linearized n-soliton operator} and \eqref{eq:symbol JLM}, one sees that the operator $\mathcal{J}\mathcal{L}_N$  consists of the element $\mathcal{J}L_1$. Therefore, there exists some $|a| \ll \sqrt{1-\frac{2\omega}{c_N}}$ such that the essential spectrum of the operator $e^{ax}\mathcal{J}\mathcal{L}_Ne^{-ax}$ lie in the open left half plane and possess spectral gap. Moreover, the semigroup decay estimate of $e^{t e^{ax}\mathcal{J}\mathcal{L}_Ne^{-ax}}$ holds true. The rest of the proof follows a parallel argument of one soliton case and we omit the details. The proof of the Corollary is completed.
\end{proof}

\section{Orbital stability of train of solitons}\label{sec 6}
The aim of this section is to revisit and modify the result of orbital stability of the train of $N$ well decoupled solitons in \cite{EM07}. The idea is to combine the classical arguments of the stability of one soliton with the almost monotonicity of the $N-1$ functionals that measure the energy at the right of the $N-1$ first bumps of $u$. The main result of this Section is as follows.
\begin{proposition}\label{thm4.1}
Let $u_0 \in Y_{+}$. For given $N$ velocities $c^0_{1},\cdots,c^0_{N}$, such that $0<2\omega<c^0_{1}<c^0_{2}<\cdots<c^0_{N}$. There exist $\gamma_{0},A>0,L_{0}>0$ and $\alpha_1> 0$, such that if $0<\varepsilon <\alpha_1$, $L>L_{0}$ and $x_1^0<\cdot\cdot\cdot<x_N^0$,
\begin{equation}\label{4.1}
\left\| u_{0}-\sum_{j=1}^{N}\varphi_{c^0_{j}}(\cdot-x^0_{j}) \right\|_{H^1} \leq \varepsilon, \quad \text{with} \quad x^0_{j}-x^0_{j-1} \geq L, \quad j=2,\cdot\cdot\cdot,N.
\end{equation}
 Then there exist $x_{1}(t),\cdots,x_{N}(t)$, such that $u\in\mathcal{C}(\R;H^1(\R))$ is a solution of the CH equation satisfies
\begin{equation*}
\sup_{t\geq0} \left\| u(t,\cdot)-\sum_{j=1}^{N}\varphi_{c^0_{j}}(\cdot -x_{j}(t)) \right\|_{H^1} \leq A(\varepsilon+e^{-\gamma_{0}L}).
\end{equation*}
\end{proposition}

Here we give a sketch of proof, since the main part of the proof has been proved in \cite{EM07}. To prove Proposition \ref{thm4.1}, it suffices to show that there exist $A>0$, $\gamma_{0}>0, \alpha_1>0$ and $L_{0}>0$, such that $\forall L>L_{0}$ and $0<\varepsilon<\alpha_1$, if $u_{0}$ satisfied \eqref{4.1}, then for all $t\geq0$, $u(t) \in \mathcal{U}(A(\varepsilon + e^{-\gamma_{0} L}),L)$ (see \eqref{4.3}). By the continuity of the map $t \mapsto u(t)$ in $H^1(\mathbb{R})$, Proposition \ref{thm4.1} is a consequence of the following.
\begin{proposition}\label{prop4.2}
 There exist $\gamma_{0},A>0,L_{0}>0$ and $\alpha_1> 0$, such that if $0<\varepsilon <\alpha_1$, $L>L_{0}$ and
\begin{equation*}
\left\| u_{0}-\sum_{j=1}^{N}\varphi_{c^0_{j}}(\cdot-x^0_{j}) \right\|_{H^1} \leq \varepsilon, \quad \text{with} \quad x^0_{j}-x^0_{j-1} \geq L, \quad j=2,\cdot\cdot\cdot,N.
\end{equation*}
If there holds for $t^\ast>0$,
\begin{equation}\label{initial assumption}
\forall t\in [0,t^\ast], \quad u(t) \in \mathcal{U}(A(\varepsilon + e^{-\gamma_{0} L}),L),
\end{equation}
then there holds
\begin{equation}\label{finial con}
\forall t\in [0,t^\ast], \quad u(t) \in \mathcal{U}(\frac{A}{2}(\varepsilon + e^{-\gamma_{0} L}),L).
\end{equation}
\end{proposition}
\begin{proof}
Let $A>0$ to be fixed later. Since from \eqref{initial assumption}, $u(t)$ is close to the sum of $N$ well coupled solitons, one may apply Lemma \ref{le4.1} on $[0,t^\ast]$. It reveals that there exists $c_j(t)$ and $x_j(t)$ such that the statement of Lemma \ref{le4.1} holds true. Let us define
\begin{equation*}
d_j(t):=\sum_{k=j}^NE(\varphi_{c_j(t)}),\quad \Delta_0^td_j:=d_j(t)-d_j(0).
\end{equation*}
We need to control the variations of $c_j(t)$, which is essentially quadratic with respect to $\|v(t)\|_{H^1}$. Then we estimate $\|v(t)\|_{H^1}$, which gives the stability result.
\begin{lemma}\label{lemma6.2}
There exists $K>0$, such that for all  $t\in [0,t^\ast]$,
\begin{eqnarray}
&&\sum_{j=1}^N\left|c_j(t)-c_j(0)\right|\leq K\big(\|v(t)\|^2_{H^1}+\|v(0)\|^2_{H^1}+e^{-2\gamma_{0} L}\big),\label{quad control}\\
&&\|v(t)\|^2_{H^1}\leq K \left( \|v(0)\|^2_{H^1}+e^{-2\gamma_{0} L} \right).\label{v control}
\end{eqnarray}
\end{lemma}
\begin{proof}[Proof of Lemma \ref{lemma6.2}.]
We claim the following
\begin{eqnarray}
&&\big|\sum_{j=1}^N\big(F(R_j(t))-F(R_j(0))\big)+\frac12\int_{\R}3R(t)v^2(t)+R(t)v^2_x(t)+2R_x(t)v(t)v'(t)+2\omega v^2(t)\rmd x \big|\nonumber\\ &&\leq K\big(\|v(t)\|^3_{H^1}+\|v(0)\|^2_{H^1}+e^{-2\gamma_{0} L}\big),\label{linear energy control}\\
&&\Delta_0^td_j+\frac{1}{2}\int_{\R}\big(v^2(t)+v^2_x(t)\big)\Psi _{K}(\cdot-y_{j}(t))\leq K \left( \|v(0)\|^2_{H^1}+e^{-2\gamma_{0} L} \right).\label{lin mass control}\\
&& \left| \left( F(R_j(t))-F(R_j(0)) \right) - c_j(t) \left( E(R_j(t))-E(R_j(0)) \right) \right|\leq K|c_j(t)-c_j(0)|^2.\label{ener mass control}
\end{eqnarray}
Indeed, \eqref{linear energy control} follows from the decomposition of $u(t)$ \eqref{decompps} and the orthogonal conditions satisfied by $v$. We expand
\begin{eqnarray*}
&&F(u(t))=\frac12\sum_{j=1}^N\int_{\R}\big(R^3_j+R_jR^2_{jx}+2\omega R^2_j\big)(t)\rmd x+\frac12\int_{\R}\big(3Rv^2+Rv^2_x+2R_xvv'+2\omega v^2\big)(t)\rmd x\\
&&\sum_{j=1}^N\int_{\R}\big(3R^2_jv+2R_jR_{jx}v'+R^2_{jx}v+4\omega R_jv\big)(t)\rmd x+\omega\sum_{j\neq k}\int_{\R}R_j(t)R_k(t)\rmd x\\
&&\frac12\int_{\R}\{(R+v)^3-\sum_{j=1}^N\big(R_j^3+3R_j^2v\big)-3Rv^2\}(t)\rmd x\\
&&\frac12\int_{\R}\{(R+v)(R'+v')^2-\sum_{j=1}^N\big(R_jR^2_{jx}+2R_jR_{jx}v'+R^2_{jx}v\big)-Rv^2_x-2R_xvv'\}(t)\rmd x.
\end{eqnarray*}
From the equation of $\varphi_c$ \eqref{eq:stationary} and orthogonal conditions, we have for all $j$
\[
\int_{\R}\big(3R^2_jv+2R_jR_{jx}v'+R^2_{jx}v+4\omega R_jv\big)(t,x)\rmd x=2c_j(t)\int_{\R}(1-\partial_x^2)R_j(t,x)v(t,x)\rmd x=0.
\]
Since $|R_j(t,x)|\leq Ce^{-\sigma_0|x-x_j(t)|}$ and $|x_j(t)-x_k(t)|\geq L$ for $j\neq k$. Then we have
\[
\left| \int_{\R}R_j(t,x)R_k(t,x)\rmd x \right| \leq Ce^{-\frac{\sigma_0L}{2}}.
\]
In view of the smallness of $\|v(t)\|_{H^1}$, we have
\[
\left| F(u(t))-F(R_j(t))-\frac12\int_{\R}\big(3Rv^2+Rv^2_x+2R_xvv'+2\omega v^2\big)(t,x)\rmd x \right|\leq C\big(e^{-\frac{\sigma_0L}{2}}+\|v(t)\|^3_{H^1}\big).
\]
Therefore, \eqref{linear energy control} follows from the conservation $F(u(t))=F(u(0))$.

Equation \eqref{lin mass control} follows from Lemma \ref{le4.2} and the orthogonal conditions satisfied by $v$. Equation \eqref{ener mass control} follows from $\frac{dF(\varphi_c)}{d c}=c\frac{dE(\varphi_c)}{d c}$.
 Now by the Abel transformation, we have
\begin{equation}\label{abel sum}
\begin{split}
\sum_{j=1}^Nc_j(t)\big(E(R_j(t))-E(R_j(0))\big) = & \sum_{j=1}^Nc_j(t)\big(\Delta_0^td_j-\Delta_0^td_{j+1}\big)+c_N(t)\Delta_0^td_N\\
= & \sum_{j=2}^N\big(c_j(t)-c_{j-1}(t)\big)\Delta_0^td_j+c_1(t)\Delta_0^td_1.
\end{split}
\end{equation}

{\bf Proof of \eqref{quad control}.} Let us define
\[
\Omega(t)=\|v(0)\|^2_{H^1}+\|v(t)\|^2_{H^1}+e^{-\frac{\sigma_0L}{2}}+\left|c_j(t)-c_j(0)\right|^2.
\]
From \eqref{linear energy control}, one has
\[
\left| \sum_{j=1}^N\left( F(R_j(t))-F(R_j(0)) \right) \right|\leq K\Omega(t),
\]
by employing \eqref{ener mass control} and \eqref{abel sum}, one deduces
\begin{equation*}
\left| \sum_{j=2}^N \left( c_j(t)-c_{j-1}(t) \right)\Delta_0^td_j+c_1(t)\Delta_0^td_1 \right| \leq K\Omega(t).
\end{equation*}
Notice that from \eqref{lin mass control}, we have, for $j=1,\cdot\cdot\cdot,N$, $\Delta_0^td_j\leq K\big(\|v(0)\|^2_{H^1}+e^{-2\gamma_{0} L}\big)$. Then in view of \eqref{initial assumption} for $j\geq2$, we have $c_1(t)\Delta_0^td_1\geq -K\Omega(t)$. Since $c_1(t)\geq\sigma_0>0$, we obtain $|\Delta_0^td_1|\leq K\Omega(t)$. Similarly,
from \eqref{initial assumption} and $c_j(t)-c_{j-1}(t)\geq\sigma_0>0$, we finally have, for all $j=1,\cdot\cdot\cdot,N$, $|\Delta_0^td_j|\leq K\Omega(t)$. Therefore,
for all $j=1,\cdot\cdot\cdot,N$, we have
\[
\left| \big(E(R_j(t))-E(R_j(0)) \right|=4\kappa_j c_j(0)|c_j(t)-c_j(0)|+O\big(|c_j(t)-c_j(0)|^2\big)\leq K\Omega(t).
\]
Thus
\[
\sum_{j=1}^N|c_j(t)-c_j(0)|\leq K\Omega(t).
\]
{\bf Proof of \eqref{v control}.} The idea is by employing \eqref{linear energy control}--\eqref{abel sum} and the coercivity property of the following quadratic form
\[
\langle \mathcal L_{loc} v(t),v(t)\rangle:=\int_{\R} \left[ c(t,x)v_x^2(t,x)+\big(c(t,x)-2\omega\big)v^2(t,x)+\left( Rvv_{xx}+R'vv'+(R''-3R)v^2 \right)(t,x) \right]\rmd x,
\]
where $c(t,x)=c_1(t)+\sum_{j=2}^N \left( c_j(t)-c_{j-1}(t) \right)\Psi_K(x-y_j)$. Inserting \eqref{linear energy control} and \eqref{abel sum} into \eqref{ener mass control}, then using \eqref{quad control}, one has
\begin{eqnarray*}
&&\left| \frac12\langle \mathcal L_{loc} v(t),v(t)\rangle-\sum_{j=1}^N(c_j(t)-c_{j-1}(t)) \left(\Delta_0^td_j+\frac{1}{2}\int_{\R}\big(v^2(t)+v^2_x(t)\big)\Psi _{K}(\cdot-y_{j}(t))\rmd x \right) \right|\\
&&\leq K\bigg(\|v(t)\|^2_{H^1}+\|v(0)\|^2_{H^1}+e^{-2\gamma_{0} L}+\sum_{j=1}^N|c_j(t)-c_j(0)|^2\bigg)\\
&&\leq  K\bigg(\|v(t)\|^2_{H^1}+\|v(0)\|^2_{H^1}+e^{-2\gamma_{0} L}\bigg).
\end{eqnarray*}
Since $c_j(t)-c_{j-1}(t)>\sigma_0$, we then from above to have
\[
\langle \mathcal L_{loc} v(t),v(t)\rangle \leq  K\bigg(\|v(t)\|^2_{H^1}+\|v(0)\|^2_{H^1}+e^{-2\gamma_{0} L}\bigg).
\]
On the other hand, Lemma \ref{le4.3} reveals that $\langle \mathcal L_{loc} v(t),v(t)\rangle\geq \lambda_0\|v(t)\|^2_{H^1}$. Therefore, \eqref{v control} holds true and this completes the proof of Lemma \ref{lemma6.2}.
\end{proof}
Now we complete the proof of Proposition \ref{prop4.2}. From Lemma \ref{lemma6.2} and $\|v(0)\|_{H^1}+\sum_{j=1}^N|c_j(0)-c^0_j|\leq K_1\varepsilon$ (see \eqref{4.5}), we have
\begin{eqnarray*}
&&\left\| u(t)-\sum_{j=1}^{N}\varphi_{c^0_{j}}(\cdot -x_{j}(t)) \right\|_{H^1}\\
&&\quad \leq \left\|u(t)-\sum_{j=1}^{N}R_j(t)\right\|_{H^1} + \left\| \sum_{j=1}^{N}R_j(t)-\sum_{j=1}^{N}\varphi_{c^0_{j}}(\cdot -x_{j}(t)) \right\|_{H^1}\\
&&\quad \leq \|v(t)\|_{H^1}+C\sum_{j=1}^{N}|c_j(t)-c^0_j|\\
&&\quad \leq\|v(t)\|_{H^1}+C\sum_{j=1}^{N}|c_j(t)-c_j(0)|+C\sum_{j=1}^{N}|c_j(0)-c^0_j|\\
&&\quad \leq\|v(t)\|_{H^1}+C\left( \|v(0)\|^2_{H^1}+e^{-2\gamma_{0} L} \right) + CK_1\varepsilon\\
&&\quad \leq K\big(\varepsilon + e^{-\gamma_{0} L}\big),
\end{eqnarray*}
where $K>0$ is independent of $A$. By choosing $A=2K$ and $A\alpha_1$ small enough, we have \eqref{finial con}. This completes the proof of Proposition \ref{prop4.2}.
\end{proof}

\section{Asymptotic stability of train of solitons}\label{sec 7}
This section devotes to the proof of Theorem \ref{th1.3} and Corollary \ref{co1.3}. From Proposition \ref{thm4.1}, the CH solution $u(t)$ is close to the sum of $N$ solitons for all time and admits a decomposition as in Lemma \ref{le4.1}. Then the proof of  Theorem \ref{th1.3} proceeds into two steps. First, by employing Theorem \ref{th1.1} and monotonicity properties, we show the convergence of $v(t)$ to $0$ around each solitons. Second, we prove the convergence of $v(t)$ in $H^1(\R)$ in the region $x>2\omega t$ by monotonicity arguments.

\subsection{Convergence around the solitons}
In this subsection, we show the following asymptotic result on $v(t,\cdot+x_j(t))$, $c_j(t)$ and $x'_j(t)$ as $t\rightarrow+\infty$.
\begin{proposition}\label{prop7.1}
Under the assumption of Theorem \ref{th1.3}, for any $j\in \{1,\cdot\cdot\cdot,N\}$, there exists $c_j^{\star}>2\omega$ such that
\begin{eqnarray*}\label{conver so}
v(t,\cdot+x_j(t))\rightharpoonup0 \quad \text{in}\ \quad H^1(\R), \ c_j(t)\rightarrow c_j^{\star},\  \quad x'_j(t)\rightarrow c_j^{\star}, \ \quad \text{as}\ \quad t\rightarrow+\infty.
\end{eqnarray*}
\end{proposition}
\begin{proof}
Consider a solution $u(t)$ satisfying the assumption of Theorem \ref{th1.3}, then by Proposition \ref{thm4.1}, $u(t)$ is uniformly close in $H^1(\R)$ to the superposition of $N$ solitons for all $t\geq0$. From Lemma \ref{le4.1}, one deduce that $v(t)$ is uniformly small in $H^1(\R)$ and $\sum_{j=1}^N|c_j(t)-c_j(0)|$ is also uniformly small.

Assume that $v(t)$ does not converge
to $0$ weakly in $H^{1}(\mathbb{R})$. Then there exists $\tilde{v}_{0} \in H^1(\mathbb{R}),\tilde{v}_0 \not\equiv 0$, and a sequence $t_n$
going to $\infty$ with $n$ such that
\begin{equation}\label{weak H^1 converge2}
  v (t_n) \rightharpoonup \tilde{v}_0 \ \text{in} \ H^1(\mathbb{R}) \ \text{as} \ n \to \infty
\end{equation}
From the weak convergence and stability result, we deduce that $\|v_0\|_{H^1} \le  C\left(\mathrm{e}^{-\sigma_{0} L_0}+\varepsilon\right)$, and therefore $\|v_0\|_{H^1}$ is small by taking $\varepsilon$ small and $L_0$ large. On the other hand, extracting a subsequence from $\{t_n\}$, we deduce the existence of $\tilde{c} >2\omega$ such that $c(t_n)$ converges to $\tilde{c}$ as $n$ goes to $+\infty$.

Now let $\tilde{u}_{0}=\varphi_{\tilde{c}} +\tilde{v}_0$ and
consider $\tilde{u}$ solution of the CH equation with initial data $\tilde{u}(0)=\tilde{u}_0$. Let $\tilde{c}(t)$ and $\tilde{x}(t)$ be the parameters associated to the solution $\tilde{u}(t)$.  To prove that \eqref{weak H^1 converge2} leads to a contradiction if  $\tilde{v}_0 \not\equiv 0$, it suffices to prove the $H^1$-localization of $\tilde{u}$ up to a translation. This will then contradict Theorem \ref{th1.1}. From Proposition \ref{pr H^1 compactness} and Theorem \ref{th1.1}, there exists $c^{\star}>2\omega$ and $x^{\star}\in \R$ such that $\tilde{u}(t)=\varphi_{c^{\star}}(\cdot-c^{\star}t-x^{\star})$. In particular,  $\tilde{u}(0)=\varphi_{c^{\star}}(\cdot-x^{\star})=\varphi_{\tilde{c}} +\tilde{v}_0$. By uniqueness of the decomposition of $\tilde{u}_{0}$,  we have $c^{\star}=\tilde{c}, x^{\star}=0$ and $\tilde{v}_0=0$, which is a contradiction.

The convergence of $c_j(t)$ follows from the monotonicity arguments related to $F_j(t)$. The fact that $x'_j(t)\rightarrow c_j^{\star}$ is a direct consequence of \eqref{4.7}.
\end{proof}

\subsection{Asymptotic behavior on $x>2\omega t$}
In this subsection, we conclude the proof of Theorem \ref{th1.3} by employing the monotonicity of energy.
\begin{proposition}\label{prop7.2}
Under the assumption of Theorem \ref{th1.3}, there holds the following
\begin{eqnarray*}
\|v(t)\|_{H^1(x>2\omega t)}\rightarrow0,\quad \text{as} \quad t\rightarrow+\infty.
\end{eqnarray*}
\end{proposition}
\begin{proof}
First, let $y_0>0$ and $\gamma_0=\frac{\sigma_0}{4}$, by arguing backwards in time (from $t$ to $0$) and conservation of $E(u)$, one has
\begin{equation*}
\int_{\mathbb{R}}(u^2+u_{x}^2)(t,x)\Psi_{K}\big(x-x_N(t)-y_{0}\big)\rmd x\leq\int_{\mathbb{R}}(u^2+u_{x}^2)(0,x)\Psi_{K}\big(x-x_N(0)-\frac{\sigma_0t}{2}-y_{0}\big)\rmd x+Ce^{-\gamma_0y_0}.
\end{equation*}
Therefore, the decay property of $\varphi_c$ reveals that
\begin{equation*}
\int_{x>x_N(t)+y_0}(v^2+v_{x}^2)(t,x)\rmd x\leq2\int_{\mathbb{R}}(u^2+u_{x}^2)(0,x)\Psi_{K}\big(x-x_N(0)-\frac{\sigma_0t}{2}-y_{0}\big)\rmd x+Ce^{-\gamma_0y_0}.
\end{equation*}
By Proposition \ref{prop7.1}, one sees that, for fixed $y_0$, $\int_{x_N(t)<x<x_N(t)+y_0}(v^2+v_{x}^2)(t,x)\rmd x\rightarrow0$ as $t\rightarrow+\infty$, then we have
\[
\lim_{t\rightarrow+\infty}\int_{x>x_N(t)}(v^2+v_{x}^2)(t,x)\rmd x=0.
\]
We need to show that for all $j$, there holds $\int_{x>x_j(t)}(v^2+v_{x}^2)(t,x)\rmd x=0$ as $t\rightarrow+\infty$ by backwards induction on $j$. Assume that for some $j_0\in \{2,\cdot\cdot\cdot,N\}$, we have $\lim_{t\rightarrow+\infty}\int_{x>x_{j_0}(t)}(v^2+v_{x}^2)(t,x)\rmd x=0$,  then we need to show
\begin{equation*}
\lim_{t\rightarrow+\infty}\int_{x>x_{j_0-1}(t)}(v^2+v_{x}^2)(t,x)\rmd x=0.
\end{equation*}
For $t\geq0$ large, there exists $0<t'=t'(t)<t$ such that
\[
x_{j_0}(t')-x_{j_0-1}(t')-\frac{\sigma_0}{2}(t-t')=2y_0,
\]
which follows from $x_{j_0}(t)-x_{j_0-1}(t)\geq\frac{\sigma_0}{2}t\geq2y_0$  and $x_{j_0}(0)-x_{j_0-1}(0)-\frac{\sigma_0}{2}t<0<2y_0$ if $t$ is large.
Therefore, one has
\begin{eqnarray}
&&\int_{\R}(u^2+u_{x}^2)(t,x)\Psi_{K}\big(x-x_{j_0-1}(t)-y_{0}\big)\rmd x\nonumber\\&&\leq\int_{\mathbb{R}}(u^2+u_{x}^2)(t',x)\Psi_{K} \left( x-x_{j_0-1}(t')-\frac{\sigma_0}{2}(t-t')-y_{0} \right)\rmd x+Ce^{-\gamma_0y_0}\nonumber\\ &&\leq
\int_{\mathbb{R}}(u^2+u_{x}^2)(t',x)\Psi_{K}\big(x-x_{j_0}(t')-y_{0}\big)\rmd x+Ce^{-\gamma_0y_0}.\label{local est21}
\end{eqnarray}
Since $t'(t)\rightarrow+\infty$ as $t\rightarrow\infty$, by $H^1_{loc}$ convergence of $v(t,x+x_{j_0}(t))$ and the induction assumption, for fixed $y_0$, we obtain
\[
\lim_{t\rightarrow+\infty}\int_{x>x_{j_0}(t')-2y_0}(v^2+v_{x}^2)(t',x)\rmd x=0.
\]
From Proposition \ref{prop7.1}, one deduces that
\[
\limsup_{t\rightarrow+\infty}\int_{\R}(u^2+u_{x}^2)(t',x)\Psi_{K}\big(x-x_{j_0}(t')-y_{0}\big)\rmd x-\sum_{k=j_0}^NE(\varphi_{c^{\star}_k})\leq Ce^{-\gamma_0y_0}.
\]
Moreover, in viewing of the decomposition of $u(t)$, we have
\begin{eqnarray*}
&&\int_{\R}(v^2+v_{x}^2)(t,x)\Psi_{K}\big(x-x_{j_0}(t)-y_{0}\big)\rmd x\\&&\leq \int_{\R}(u^2+u_{x}^2)(t,x)\Psi_{K}\big(x-x_{j_0}(t)-y_{0}\big)\rmd x-\sum_{k=j_0}^NE(\varphi_{c_k(t)})+Ce^{-\gamma_0y_0},
\end{eqnarray*}
since $c_k(t)\rightarrow c^{\star}_k$ as $t\rightarrow+\infty$, by \eqref{local est21}, we have
\[
\lim_{t\rightarrow+\infty}\int_{x>x_{j_0}(t)+y_0}(v^2+v_{x}^2)(t,x)\rmd x=0,
\]
and therefore
\[
\lim_{t\rightarrow+\infty}\int_{x>x_{j_0}(t)}(v^2+v_{x}^2)(t,x)\rmd x=0,
\]
The above induction argument reveals finally that
\[
\lim_{t\rightarrow+\infty}\int_{x>x_{1}(t)}(v^2+v_{x}^2)(t,x)\rmd x=0.
\]

It remains to show $\lim_{t\rightarrow+\infty}\int_{x>2\omega t}(v^2+v_{x}^2)(t,x)\rmd x=0$. Indeed, let $0<t'=t'(t)<t$ such that $x_1(t')-2\omega t'-\frac{c_1^0-2\omega}{4}(t+t')=y_0$. Then for $\sup_{t\geq0}\|v(t)\|_{H^1}$ small enough, we have
\begin{eqnarray*}
&&\int_{\R}(u^2+u_{x}^2)(t,x)\Psi_{K}\big(x-2\omega t\big)\rmd x\\
&&\leq \int_{\R}(u^2+u_{x}^2)(t',x)\Psi_{K}\left( x-2\omega t'-\frac{c_1^0-2\omega}{4}(t-t') \right)\rmd x+Ce^{-\gamma_0y_0}\\ &&\leq \int_{\R}(u^2+u_{x}^2)(t',x)\Psi_{K}\big(x-x_{1}(t')+y_{0}\big)\rmd x+Ce^{-\gamma_0y_0}.
\end{eqnarray*}
The conclusion is  obtained as before.
\end{proof}

We are now able to conclude the proof of Theorem \ref{th1.3}.
\begin{proof}[End of proof of Theorem \ref{th1.3}.]
 The proof of Theorem \ref{th1.3} follows from Proposition \ref{prop7.1} and Proposition \ref{prop7.2}.
\end{proof}

\begin{proof}[Proof of Corollary \ref{co1.3}.]
Note first that (see for example \cite{M05})
\begin{equation}\label{decoupl meas}
\left\| U^{(N)}(\cdot;c_j^0,-y_j)-\sum_{j=1}^N\varphi_{c_j^0}(\cdot-y_j) \right\|_{H^1}\rightarrow0, \quad \text{ as } \ \inf(y_{j+1}-y_j)\rightarrow+\infty.
\end{equation}
Let $a<a_0$, $L\geq L_0$ be such that $A_0(a+e^{-\gamma_0L})<\frac{\delta_1}{2}$ and
\[
\left\| U^{(N)}(\cdot;c_j^0,-y_j)-\sum_{j=1}^N\varphi_{c_j^0}(\cdot-y_j) \right\|_{H^1}\leq\frac{\delta_1}{2}, \quad \text{ as } \ \inf(y_{j+1}-y_j)>L.
\]
Let $T>0$ be such that $c_{j+1}^0T+x_{j+1}^0>c_{j}^0T+x_{j}^0+2L$ for all $j=1,\cdot\cdot\cdot,N-1$, then for all $t\geq T$, the $N$-solitons $v(t,x):=U^{(N)}(x;c_j^0,-x^0_j-c_j^0t)$ satisfy
\begin{equation}\label{N-solitons meas}
\left\| v(t,x)-\sum_{j=1}^N\varphi_{c_j^0}(x-c_j^0t-x_j^0) \right\|_{H^1}\leq\frac{\delta_1}{2}.
\end{equation}
From the continuous dependence of the CH solution with respect to the initial data, there exists $a_1>0$ such that if $\|u(0)-v(0)\|_{H^1}<a_1$, then $\|u(T)-v(T)\|_{H^1}<\frac{a}{2}$. Therefore, from \eqref{N-solitons meas}, one has
\[
\left\| u(T,x)-\sum_{j=1}^N\varphi_{c_j^0}(x-c_j^0T-x_j^0) \right\|_{H^1}\leq a.
\]
By Proposition \ref{thm4.1}, there exist $x_j(t)$ with $x_{j+1}(t)-x_{j}(t)>L$, for all $t\geq T$ such that
\[
\left\| u(t,\cdot)-\sum_{j=1}^{N}\varphi_{c^0_{j}}(\cdot -x_{j}(t)) \right\|_{H^1} \leq A(\varepsilon+e^{-\gamma_{0}L})<\frac{\delta_1}{2}.
\]
Together with \eqref{N-solitons meas}, this gives \eqref{stability}.

Finally, \eqref{asympt lim3} follows from Theorem  \ref{th1.3} and \eqref{decoupl meas}.
\end{proof}

\section{Applications of other dispersive equations}\label{sec 8}
In this Section, we present our approach to show some counterpart rigidity results of Theorem \ref{thm1.4} for other integrable models, such as KdV and mKdV equations.

We consider the gKdV equations
\begin{equation}\label{mKdV}
u_{t} +\left(u_{xx}+u^{p}\right)_{x}=0, \quad t, x \in \mathbb{R},
\end{equation}
for $p=2,3 $ integer. Let
\begin{equation*}\label{solito pro}
Q(x)=\left(\frac{p+1}{2 \cosh ^{2}\left(\frac{p-1}{2} x\right)}\right)^{\frac{1}{p-1}},
\end{equation*}
be the unique $H^{1}$ positive solution (up to translations) of
\begin{equation*}\nonumber
Q^{\prime \prime}+Q^{p}=Q \quad \text { on }\quad \mathbb{R}.
\end{equation*}
Then, for any $c>0$, $x_{0}\in\mathbb{R}$, the functions
\begin{equation}\nonumber
R_{c, x_{0}}(t, x)=Q_{c}\left(x-x_{0}-c t\right), \quad \text { where } \quad Q_{c}(x)=c^{\frac{1}{p-1}} Q(\sqrt{c} x),
\end{equation}
are solions of the gKdV equations \eqref{mKdV}.
the following quantities are invariant for such $H^{1}$ solutions:
\begin{equation*}\label{mass KdV}
H_{K,1}(u(t)):=\frac{1}{2}\int_{\mathbb{R}} u^{2}(t)=\frac{1}{2}\int_{\mathbb{R}} u^{2}(0),
\end{equation*}
\begin{equation*}\label{Energy KdV}
H_{K,2}(u(t)):=\int_{\mathbb{R}}\left[\frac{1}{2}\left(\partial_{x} u(t)\right)^{2}-\frac{1}{p+1} u^{p+1}(t)\right]=\int_{\mathbb{R}}\left[\frac{1}{2}\left(\partial_{x} u(0)\right)^{2}-\frac{1}{p+1} u^{p+1}(0)\right],
\end{equation*}
where the subscript $K$ denotes the (m)KdV. The case $p=2,3$ we considered are completely integrable, \eqref{mKdV} admits  infinitely many conservation laws
\begin{eqnarray*}\label{m conser}
H_{K,n}(u):=\frac12\int_{\mathbb R} (\partial_x u_{n})^2\rmd x+\int_{\mathbb R}q_{K,n}(u, u_x,..., \partial_x u_{n-1} )\rmd x,
\end{eqnarray*}
where $q_{K,n}$ is a polynomial. Furthermore, invariants $H_{K,n}$  can be defined via the the following famous Lenard recursion
\begin{equation}\label{eq:KdV recursion}
\mathcal{J}\frac{\delta H_{K,n+1}(u)}{\delta u}=\mathcal{K}(u)\frac{\delta H_{K,n}(u)}{\delta u},
\end{equation}
where $\mathcal{J}$ is the operator $\partial_x$ and $\mathcal{K}(u)$ is a skew-adjoint operator as follows
\begin{eqnarray}
&&\mathcal{K}(u)=-\partial_x^3-\frac 43u\partial_x-\frac 23u_x, \quad p=2,\label{KdV skew-ad operator}\\
&&\mathcal{K}(u)=-\partial_x^3-2u^2\partial_x-2u_x\partial_x^{-1}(u\partial_x). \quad p=3.\label{mKdV skew-ad operator}
\end{eqnarray}
Notice that by the scaling and translation invariance of the equation,
it is enough to study  the soliton $Q(x-t)$, the linearized operator around solitons is
\begin{equation}\label{linear operator}
\mathcal{L}_1: v=-\partial_{x}^{2} v+v-p Q^{p-1} v=-\partial_{x}^{2} v+v-\frac{p(p+1)}{2 \cosh ^{2}\left(\frac{p-1}{2} x\right)} v, \quad p=2,3.
\end{equation}
The higher order linearized operator
\begin{equation}\label{higher linear operator}
\mathcal{L}_n:=H''_{K,n+1}(Q)+H''_{K,n}(Q)=\big(\mathcal{J}^{-1}\mathcal{K}(Q)\big)^{n-1}\mathcal{L}_1,\quad n\geq1.
\end{equation}

\begin{proposition}\label{col 7.3}
  Suppose that  $p=2,3$.  Let  $v \in C\left(\mathbb{R}, H^{1}(\mathbb{R})\right) \cap L^{\infty}\left(\mathbb{R}, H^{1}(\mathbb{R})\right)$  be a solution of
\begin{equation}\label{eq38e}
\partial_{t} v=\partial_{x}(\mathcal{L}_n v)+A(t)\left(\frac{2}{p-1} Q+x Q'\right)+B(t) Q', \quad \text { on } \mathbb{R} \times \mathbb{R},
\end{equation}
where $A(t)$ and $B(t)$ are two continuous and bounded functions. Assume that for two constants $C>0$ and $\sigma>0$,
\begin{equation}\label{eq39e}
\forall t, x \in \mathbb{R}, \quad|v(t, x)| \leq C e^{-\sigma|x|},
\end{equation}
Then for all $t \in \mathbb{R}$,
\begin{equation}\label{eq40'}
v(t) \equiv a(t)\left(\frac{2}{p-1} Q+x Q'\right)+b(t) Q',
\end{equation}
for some $C^{1}$ bounded functions $a(t)$ and $b(t)$ satisfying
\begin{equation}\label{eq41'}
a'(t)=A(t), \quad b'(t)=2(-1)^n a(t)+B(t).
\end{equation}
\end{proposition}

\begin{corollary}[Rigidity  for KdV and mKdV] \label{prop 7.1}
Suppose that $p =2,3$ and $n\geq1$. Let $v \in C\left(\mathbb{R}, H^{1}(\mathbb{R})\right) \cap L^{\infty}\left(\mathbb{R}, H^{1}(\mathbb{R})\right)$ be a solution of
\begin{equation}\label{linear problem for mKdV}
\partial_{t} v=\partial_{x}(\mathcal{L}_n v) \quad \text { on } \mathbb{R} \times \mathbb{R}.
\end{equation}
Assume that for two constants $ C>0, \sigma>0 $
\begin{equation*}
\forall t, x \in \mathbb{R}, \quad|v(t, x)| \leq C e^{-\sigma|x|}.
\end{equation*}
Then there exists a constant $b_{0}\in \mathbb{R}$ such that for all $t \in \mathbb{R}$,
\begin{equation}\label{anseer}
v(t) \equiv b_{0} Q'.
\end{equation}
\end{corollary}

Since by direct calculations,
$$
\mathcal{L}_n \left( \frac{2}{p-1} Q+x Q' \right)=(-1)^n2 Q, \quad \mathcal{L}_n Q'=0.
$$
It is easily checked that (\ref{eq40'}) indeed define a solution of (\ref{eq38e}). Corollary \ref{col 7.3} means that there is no other $H^{1}$ bounded solution of (\ref{eq38e}) satisfying (\ref{eq39e}).
\begin{remark} \label{re7.3}
Corollary \ref{prop 7.1} and  Proposition\ref{col 7.3} extend the work of Martel \cite{Ma06} (the case $n=1$). We  give a classification of solutions for a more general problem \eqref{eq38e}.
\end{remark}
\subsection{The KdV case}
In this subsection, we collect the results which is essential in the proof of Corollary \ref{prop 7.1} and Proposition \ref{col 7.3} for the KdV equation. We start with the following commutativity of recursive operators and linearized operators
\begin{eqnarray}
&&\mathcal{R}_{K}(Q)\mathcal{L}_n\partial_x=\mathcal{L}_n\partial_x\mathcal{R}_{K}(Q),\quad \mathcal{R}_{K}(Q)=\partial_x^{-1}\mathcal{K}(Q),\label{com recur}\\
&&\mathcal{R}^{\ast}_{K}(Q)\partial_x\mathcal{L}_n=\partial_x\mathcal{L}_n\mathcal{R}^{\ast}_{K}(Q),\quad \mathcal{R}^{\ast}_{K}(Q)=\mathcal{K}(Q)\partial_x^{-1}. \label{com adrecur}
\end{eqnarray}

The eigenfunctions of $\mathcal{R}_{K}(Q)$ can be given in terms of squared eigenfunctions of the Schr\"odinger equation
\begin{equation}\label{eq:arb eigenvalue}
v_{xx}+\left(\frac43Q(2x)+\kappa^2\right)v=0, \quad \kappa\in \mathbb C,
\end{equation}
we define the Jost solutions $\psi(x,\kappa)$ and $\phi(x,\kappa)$ of \eqref{eq:arb eigenvalue} as follows
\begin{eqnarray*}
&&|\psi(x,\kappa)- e^{i\kappa x}|\rightarrow0, x\rightarrow +\infty,\\
&&|\phi(x,\kappa)- e^{-i\kappa x}|\rightarrow0, x\rightarrow -\infty.
\end{eqnarray*}
It is easy to check that $\psi$ and $\phi$ above are simply
\begin{eqnarray*}
&&\psi(x,\kappa)=\frac{\kappa+i\tanh x}{\kappa+i}e^{i\kappa x},\\
&&\phi(x,\kappa)=\frac{\kappa-i\tanh x}{\kappa+i}e^{-i\kappa x}.
\end{eqnarray*}
For these eigenfunctions, we can check that $\phi^2(\frac{x}{2},\kappa)$ are eigenfunctions of $\mathcal{R}(Q)$ with the corresponding eigenvalue $\kappa^2$, that is,
\begin{equation}\label{eq:arb eigenfuc}
\mathcal{R}(Q)\phi^2 \left( \frac{x}{2},\kappa \right) = \kappa^2\phi^2 \left( \frac{x}{2},\kappa \right).
\end{equation}
To separate continuous and discrete eigenfunctions, we define the set
\begin{equation}\label{eq:set KdV}
\left\{ \phi^2 \left( \frac{x}{2},\kappa \right), \text{for} \ \kappa\in\R;\quad \phi^2 \left( \frac{x}{2},i \right)\sim Q,\quad \left. \frac{\partial \phi^2(\frac{x}{2},\kappa)}{\partial \kappa}\right|_{\kappa=i} \right\},
\end{equation}
Set \eqref{eq:set KdV} consists of continuous and discrete eigenfunctions of the operator $\mathcal{R}(Q)$. In particular, the continuous eigenfunctions $\phi^2(\frac{x}{2},\kappa)$ satisfy \eqref{eq:arb eigenfuc}. The discrete eigenfunctions satisfy
\begin{eqnarray*}
&&\mathcal{R}(Q)\phi^2 \left( \frac{x}{2},i \right) = -\phi^2 \left( \frac{x}{2},i \right) = -\frac23Q,\\
&&\mathcal{R}(Q) \left. \frac{\partial \phi^2(\frac{x}{2},\kappa)}{\partial \kappa} \right|_{\kappa=i} =-\left. \frac{\partial \phi^2(\frac{x}{2},\kappa)}{\partial \kappa} \right|_{\kappa=i} + 2i\phi^2 \left( \frac{x}{2},i \right).
\end{eqnarray*}
Similarly, the eigenfunctions of adjoint recursion operator $\mathcal{R}^{\ast}(Q)$ form the following set
\begin{equation}\label{eq:adjointset KdV}
\left\{ \left( \phi^2 \left( \frac{x}{2},\kappa \right) \right)_x, \text{ for} \ \kappa\in\R;\quad \left(\phi^2 \left(\frac{x}{2},i \right) \right)_x\sim Q',\quad \left( \left.\frac{\partial \phi^2(\frac{x}{2},\kappa)}{\partial \kappa}\right|_{\kappa=i} \right)_x\sim 2Q+x Q^{\prime} \right\},
\end{equation}
Therefore, for any $f\in L^2(\R)$, it can be expanded over \eqref{eq:adjointset KdV} as follows
\begin{equation}\label{eq: KdV decom}
f(x)=\int_{\R}\alpha(\kappa)\left( \phi^2 \left(\frac{x}{2},\kappa \right) \right)_x \rmd \kappa+\beta Q'+\gamma(2Q+x Q^{\prime}),
\end{equation}
with the coefficients $\alpha(\kappa)$, $\beta$ and $\gamma$, which are related to the coefficients in the completeness relations between the two sets \eqref{eq:set KdV} and \eqref{eq:adjointset KdV}.
\subsection{The mKdV case}
This subsection is devoted to the proof of Proposition \ref{prop 7.1} and Corollary \ref{col 7.3} for the KdV and mKdV equations.
We first review the following general Zakharov-Shabat eigenvalue problem
\begin{align}\tag{ZS}\label{ZS}
\begin{cases}
v_{1x}(t,x)+i\zeta v_1(t,x)=q(t,x)v_2(t,x), \\
v_{2x}(t,x) -i\zeta v_2(t,x)=r(t,x)v_1(t,x),\quad x\in \mathbb R,
\end{cases}
\end{align}
where $q,r$ are complex-valued integrable potentials. In this case, given a particular relationship (called involutions) between $q$ and $r$ leads to other integrable PDEs. As special cases, we list (see \cite{AKNS74}.\\
1).$r=-1$,\quad $q_t+q_{xxx}+6qq_x=0$, \quad (KdV)\\
2).$r=\mp q$,\quad $q_t+q_{xxx}\pm6q^2q_x=0$, \quad (mKdV)\\
3).$r=\mp \bar{q}$,\quad $iq_t+q_{xx}\pm|q|^2q=0$, \quad (NLS )\\
4).$r=-q=\frac12u_x$,\quad $u_{xt}=\sin u$, \quad (sine-Gordon).

Let us consider \eqref{ZS} with soliton potential $r(t,x)=-q(t,x)=\frac{\sqrt{2}}{2}Q$, the Jost solutions  $\phi(\zeta,x)$ and $\psi(\zeta,x)$ are uniquely determined by imposing the following asymptotic conditions
\begin{equation}\label{Jost}
 \phi(x,\zeta)=\begin{pmatrix}
 \phi_1\\ \phi_2
\end{pmatrix}=\begin{pmatrix}
 1\\0
\end{pmatrix}e^{-i\zeta x} +o(1),\quad \tilde{\phi}(x,\zeta)=\begin{pmatrix}
 \tilde{\phi}_1\\ \tilde{\phi}_2
\end{pmatrix}=\begin{pmatrix}
 0\\-1\end{pmatrix}e^{i\zeta x} +o(1),
\end{equation}
as $x\rightarrow-\infty$, and
\begin{equation}\label{Jost2}
 \psi(x,\zeta)=\begin{pmatrix}
 \psi_1\\ \psi_2
\end{pmatrix}=\begin{pmatrix}
 0\\1
\end{pmatrix}e^{i\zeta x} +o(1),\quad \tilde{\psi}(x,\zeta)=\begin{pmatrix}
 \tilde{\psi}_1\\ \tilde{\psi}_2
\end{pmatrix}=\begin{pmatrix}
 1\\0\end{pmatrix}e^{-i\zeta x} +o(1),
\end{equation}
as $x\rightarrow+\infty$.
These Jost solutions have simple expression as follows \cite{Y00}.
\begin{eqnarray*}
&&\phi(x,\zeta)=\frac{1}{2i\zeta-1}\begin{pmatrix}
 \tanh x+2i\zeta\\-\operatorname{sech}x
\end{pmatrix}e^{-i\zeta x},\quad \tilde{\phi}(x,\zeta)=\frac{1}{2i\zeta+1}\begin{pmatrix}
-\operatorname{sech}x\\ \tanh x-2i\zeta
\end{pmatrix}e^{i\zeta x}\\
&&\psi(x,\zeta)=\frac{1}{1-2i\zeta}\begin{pmatrix}
 -\operatorname{sech}x\\ \tanh x-2i\zeta
\end{pmatrix}e^{i\zeta x},\quad \tilde{\psi}(x,\zeta)=\frac{1}{2i\zeta+1}\begin{pmatrix}
\tanh x+2i\zeta\\-\operatorname{sech}x
\end{pmatrix}e^{-i\zeta x}.
\end{eqnarray*}
Then we can deduce the following relations which correspond to the spectrum of recursion operator $\mathcal{R}^{\ast}(Q)$
\begin{eqnarray*}
&&\mathcal{R}^{\ast}(Q)\left(\phi^2_2-\phi^2_1 \right) \left( x,\frac{k}{2} \right)=\frac{k^2}{4}\big(\phi^2_2-\phi^2_1\big) \left(x,\frac{k}{2} \right),\\
&&\mathcal{R}^{\ast}(Q)Q'=-Q',\quad \mathcal{R}^{\ast}(Q)(Q+xQ')=-(Q+xQ')-2Q'.
\end{eqnarray*}
 The corresponding decomposition of $z\in L^2(\R)$ is
\begin{eqnarray} \label{decomposition of z 4}
 z(x)=\int_{\R}\big(\phi^2_2-\phi^2_1\big) \left(x,\frac{k}{2} \right)\alpha(k)\rmd k+\beta Q'+\gamma\frac{\partial Q_{c}}{\partial c}|_{c=1},
\end{eqnarray}
where
\[
\big(\phi^2_2-\phi^2_1\big) \left(x,\frac{k}{2} \right) = \frac{1}{(k+i)^2} \left[ (\tanh (x)+ik)^2-\operatorname{sech}^2(x) \right] e^{-ikx},\quad k\in\R.
\]
In particular, the kernel of $\mathcal{R}^{\ast}(Q)$ is spanned by the function $\big(\phi^2_2-\phi^2_1\big)(x,0)=1-Q^2\notin L^2(\R)$.

With \eqref{eq: KdV decom} and \eqref{decomposition of z 4} in hand, we can complete the proof of Proposition \ref{prop 7.1} and Corollary \ref{col 7.3}.

\begin{proof}[Proof of Proposition \ref{col 7.3}.]
The Cauchy problem of \eqref{eq38e} for $n=1$ is globally well-posed in $H^{1}(\R)$. For general $n$,  \eqref{eq38e} is also well-posed in $H^{1}(\R)$, since from a parallel of Lemma \ref{le3.5}, one has for all $n\geq2$
$$\mathcal{J}\mathcal L_n=\mathcal{J}\mathcal{R}_{K}(Q)\mathcal L_{n-1}=\mathcal{R}_K^{\ast}(Q)\mathcal{J}\mathcal L_{n-1}=\big(\mathcal{R}_K^{\ast}(Q)\big)^{n-1}\mathcal{J}\mathcal L_{1},$$
the well-posedness of Cauchy problem of \eqref{eq38e} in $H^{1}(\R)$ follows from inductions over $n$ and the invertibility in $L^2(\R)$ of the adjoint recursion operator $\mathcal{R}_K^{\ast}(Q)$.

Now we restrict ourselves to the case $A(t)=B(t)=0$, since one can show that $\int_0^tA(s)ds$ and $\int_0^tB(s)ds$ are uniformly bounded in time \cite{Ma06}. The spectral information of $\partial_x\mathcal{L}_n$ lie in imagery axis, with an eigenvalue at $\lambda=0$ double. Indeed, it reveals from \eqref{com adrecur} that the operators $\partial_x\mathcal{L}_n$ shares the same eigenfunctions of the adjoint recursion operator $\mathcal{R}^{\ast}_{K}(Q)$ and these eigenfunctions are $L^2$-complete.

 We consider first the $n=1$ case. From the work of \cite{PW94} and \cite{MT14}, for $0<\nu \ll 1$, the operator $e^{\nu x}\partial_x\mathcal{L}_1e^{-\nu x}$ lie in the open left half plane and possess spectral gap. Therefore, the semigroup decay estimate of $e^{t e^{\nu x}\partial_x\mathcal{L}_1e^{-\nu x}}$ holds true. Combining the $L^2$-localized condition of $v$ yield $v\equiv0$.

For the general case of $n\geq2$, we will consider the operator $\mathcal{L}_n+K\mathcal{L}_1 $ instead  with the constant $K>0$ large, then we can verify that there exists some $|\nu_n|<<1$ such that the operator $e^{\nu_n x}\partial_x(\mathcal{L}_n+K\mathcal{L}_1)e^{-\nu_n x}$ lie in the open left half plane and possess spectral gap. Therefore, the semigroup decay estimate of $e^{t e^{\nu_n x}\partial_x(\mathcal{L}_n+K\mathcal{L}_1)e^{-\nu x}}$ holds true. Therefore, following the argument of the case $n=1$ as above, any $L^2$-localized solutions of $w_t+\partial_x(\mathcal{L}_n+K\mathcal{L}_1)w=0$ must be zero, thus any $L^2$-localized solutions of \eqref{eq38e} are also zero. This completes the proof of the Proposition \ref{col 7.3}.
\end{proof}

\begin{proof}[Proof of Corollary \ref{prop 7.1}.]
Let us consider the problem \eqref{eq38e} with $A(t)=B(t)=0$, then Corollary \ref{prop 7.1} follows from Proposition \ref{col 7.3}, which completes the proof of Corollary \ref{prop 7.1}.
\end{proof}

\noindent {\bf Acknowledgments.}
 The work of  Lan is partially supported by  the NSF of China under grant number 12201340. The work of Wang is partially supported by the NSF of China  under grant number 11901092 and Guangdong NSF under grant number 2023A1515010706.


\end{document}